\documentclass{amsart}

\usepackage{amsmath}
\usepackage{amsfonts}
\usepackage{amssymb}
\usepackage{amsthm}
\usepackage{array}
\usepackage{bbm}
\usepackage{chngpage}
\usepackage{comment}
\usepackage{float}
\usepackage[OT2,T1,T2A]{fontenc}
\usepackage{graphicx,caption}
\usepackage[export]{adjustbox}
\usepackage{longtable}
\usepackage{mathrsfs}
\usepackage{mathtools}
\usepackage{wrapfig}
\usepackage{cutwin}
\usepackage{shapepar}
\usepackage{tikz}
\usepackage[lowtilde]{url}
\usepackage{subcaption}
\usetikzlibrary{matrix,arrows}
\usepackage{tabularx}
\usepackage{multirow}
\definecolor{citeclr}{rgb}{0.55, 0.55, 0.64}
\definecolor{linkclr}{rgb}{0, 0.21, 0.9447}
\usepackage[colorlinks,
  linkcolor=linkclr,
  citecolor=citeclr,
  urlcolor=blue!46!cyan,
  ]{hyperref}
\usepackage[nameinlink, nosort]{cleveref}
\usepackage{enumitem}
\usepackage{diagbox}
\usepackage{etoolbox}
\usepackage{algpseudocode}
\usepackage{etoolbox}
\usepackage{suffix}
\usepackage{nicefrac}
\usepackage[russian,english]{babel}

\usepackage[margin=0.9in]{geometry}

\patchcmd{\subsection}{\bfseries}{\itshape}{}{}
\patchcmd{\subsection}{-.5em}{.5em}{}{}
\patchcmd{\subsubsection}{-.5em}{.5em}{}{}

\newtheorem{theorem}{Theorem}[section]
\newtheorem{lemma}[theorem]{Lemma}

\newtheorem{proposition}[theorem]{Proposition}

\theoremstyle{definition}

\newtheorem{remark}[theorem]{Remark}
\newtheorem{figurecap}[theorem]{Figure}

\newtheorem{example}[theorem]{Example}
\newtheorem{conditions}[theorem]{Conditions}

\renewcommand{\C}{\mathbb{C}}

\newcommand{\Q}{\mathbb{Q}}
\newcommand{\R}{\mathbb{R}}
\newcommand{\Z}{\mathbb{Z}}

\newcommand{\cF}{\mathcal{F}}

\newcommand{\cM}{\mathcal{M}}

\newcommand{\cR}{\mathcal{R}}
\newcommand{\cS}{\mathcal{S}}

\DeclareSymbolFont{cyrletters}{OT2}{wncyr}{m}{n}
\DeclareMathSymbol{\sha}{\mathalpha}{cyrletters}{"58}

\newcommand{\eps}{\varepsilon}
\newcommand{\dee}{\partial}

\renewcommand{\Re}{\mathrm{Re}}
\renewcommand{\Im}{\mathrm{Im}}
\renewcommand{\arg}{\mathrm{arg}}
\newcommand*{\onesymb}{\text{\large\usefont{U}{bbold}{m}{n}1}} 
\newcommand{\one}[1]{\raisebox{-0.33pt}{\onesymb}\mspace{-1.5mu}\{#1\}}
\newcommand{\onelr}[1]{\raisebox{-0.33pt}{\onesymb}\mspace{-4.5mu}\left\{#1\right\}}
\renewcommand{\mod}{\mspace{4mu}\mathrm{mod}\mspace{4mu}}

\newlength{\strutheight}
\newcommand{\half}{\frac{1}{2}}
\newcommand{\thalf}{\tfrac{1}{2}}

\newcommand\starsum{\mathop{\sum\nolimits^{*}}} 
\WithSuffix\newcommand\starsum*{\mathop{\sum\nolimits^{\mathrlap{*}}}}
\newcommand{\halfGamma}[1]{\Gamma\!\left(\frac{#1}{2}\right)\!}
\newcommand{\thalfGamma}[1]{\Gamma\!\left(\tfrac{#1}{2}\right)\!\mspace{1.5mu}}

\newcommand{\wtpsi}{\mspace{7.5mu}\lower0.18ex\hbox{$\widetilde{}$}\mkern-7.5mu \psi}

\newcommand{\x}{x}
\newcommand{\rep}{\varphi}
\newcommand{\repdim}{r}
\newcommand{\mf}{f}
\newcommand{\centralchar}{\nu}
\newcommand{\repsq}{\rep^{(2)}}
\newcommand{\barrepsq}{\bar\rep^{(2)}}

\newcommand{\rootnum}{\omega}
\newcommand{\divcount}{\sigma_0}
\newcommand{\subconvexityparam}{\lambda}
\newcommand{\V}{V}
\newcommand{\psiV}{\psi}
\newcommand{\psixi}{\psi_\xi}
\newcommand{\Eta}{H} 
\newcommand{\ff}{\mathfrak{f}}
\newcommand{\qrep}{q}
\newcommand{\ET}{\Omega}
\newcommand{\qmur}{\mathfrak{q}}
\newcommand{\nmur}{\mathfrak{n}}

\newcommand{\prodr}{\prod_{j=1}^\repdim}
\newcommand{\sumr}{\sum_{j=1}^\repdim}
\newcommand{\Favg}{\frac{1}{\#\cF}\sum_{d \in \cF}}
\newcommand{\Favgrn}{\frac{\rootnum_{\cF}}{\#\cF}\sum_{d \in \cF}}
\newcommand{\nsqsum}{\sum_{\substack{n=1 \\ n \neq \square}}^\infty}
\newcommand{\mup}{\big(\xi\sqrt{QD^r}\big)}
\newcommand{\mum}{\big(\xi^{-1}\sqrt{QD^r}\big)}
\newcommand{\mupnp}{\xi\sqrt{QD^r}}
\newcommand{\mumd}{\big(\xi^{-1}\sqrt{QD}\big)}
\newcommand{\xd}{\left(\frac{\pi^\repdim \x}{\qrep d^\repdim}\right)}
\newcommand{\xD}{\left(\frac{\pi^\repdim \x}{\qrep D^\repdim}\right)}
\newcommand{\afefac}{\left(\frac{\pi^\repdim}{\qrep d^\repdim}\right)^{\!s - \half}}

\newcommand{\epsfac}{(QD)^\eps}

\newcommand{\qet}{\min\{D^{\frac{1}{14}} \qrep^{\frac{1}{6}},\, \qrep^{\frac{1}{2}}\}}
\newcommand{\Gfac}{\prod_{j=1}^r \frac{\halfGamma{1-s+\bar\kappa_j}}{\halfGamma{s+\kappa_j}}}
\newcommand{\Vp}{V_s\!\left(\frac{\pi^{\frac{\repdim}{2}} n}{\xi\sqrt{\qrep d^\repdim}}\right)}
\newcommand{\Vm}{V_{1-s}^*\!\left(\frac{\pi^{\frac{\repdim}{2}} n\xi}{\sqrt{\qrep d^\repdim}}\right)}

\newcommand{\Vpsq}{V_s\!\left(\frac{\pi^{\frac{\repdim}{2}} n^2}{\xi\sqrt{\qrep d^\repdim}}\right)}
\newcommand{\Vmsq}{V_{1-s}^*\!\left(\frac{\pi^{\frac{\repdim}{2}} n^2\xi}{\sqrt{\qrep d^\repdim}}\right)}

\newcommand{\VmDsq}{V_{1-s}^*\!\left(\frac{\pi^{\frac{\repdim}{2}} n^2\xi}{\sqrt{\qrep D^\repdim}}\right)}
\newcommand{\Gp}{G}
\newcommand{\Gm}{G^*\mspace{-2mu}}
\newcommand{\qstar}{\qrep_*}

\makeatletter
\newcommand{\hypertargetc}[1]{\Hy@raisedlink{\hypertarget{#1}{}}}
\makeatother

\author{Alex Cowan}
\address{Department of Mathematics, University of Waterloo, Waterloo, ON, Canada}
\email{alex.cowan@uwaterloo.ca}

\title[On the mean value of $\mathrm{GL}_1$ and $\mathrm{GL}_2$ $L$-functions, with applications to murmurations]{On the mean value of $\mathrm{GL}_1$ and $\mathrm{GL}_2$ $L$-functions,\\with applications to murmurations}
\date{\today}

\AtBeginDocument{%
   \def\MR#1{}
}

\begin{document}
\begin{abstract}
  ``Murmurations''
  are a recently-discovered type of fine structure in sums of Dirichlet coefficients averaged over families of $L$-functions.
  The root cause of this phenomenon remains mysterious.
  In the present paper, we demonstrate how murmurations arise from the averaging of approximate functional equations.
  This approach to the study of murmurations explains their empirically observed ubiquity,
  as well as their characteristic scale invariance and the peculiar normalization they demand.
  We implement our new approach to the study of murmurations in the case of quadratic twist families of $\mathrm{GL}_1$ automorphic representations, where we exhibit murmurations unconditionally.

  Our proof centres around estimating mean values of the $L$-functions in our quadratic twist families.
  In particular, we require estimates valid significantly higher in the critical strip than what existing results provide.
  To produce these estimates, we construct a variation of the approximate functional equation which is imbued with a mechanism for dynamically rebalancing error terms while preserving holomorphicity.
  We also generalize and sharpen results of Jutila and Stankus on sums of quadratic characters and fundamental discriminants.
  Mean value estimates are given for $\mathrm{GL}_2$ quadratic twist families as well.
\end{abstract}
\maketitle
{
\tableofcontents
}
\section{Introduction}
\label{sec:intro}

\subsection{Murmurations}
\label{sec:intro_murmurations}

The statistical behaviour of Dirichlet coefficients
is a major topic of interest in modern number theory. Via standard methods in analytic number theory, this subject is so intertwined as to be effectively equivalent to the study of $L$-functions in the critical strip; see for instance the prime number theorem.

Up until only very recently, focus has been on
statistics of Dirichlet coefficients $a_\rep(n)$ of automorphic representations $\rep$
in
asymptotic regimes where either $\rep$ is fixed and statistics are taken as $n$ varies --- the \textit{horizontal aspect} --- e.g.\ the prime number theorem, Sato--Tate; or where $n$ is fixed and $\rep$ varies --- the \textit{vertical aspect} --- e.g.\ vertical Sato--Tate, trace formulas. Recently, the phenomenon of ``murmurations'', inspired by the numerical investigation \cite{HLOP}, has highlighted the richness of statistics of $a_\rep(n)$ in asymptotics where $\rep$ and $n$ vary simultaneously so as to keep constant the ratio of $n$ to the conductor $\qrep$ of $\rep$. Particularly striking is the empirical observation that, in a great many cases \cite{drew_letter,quanta}, the value of $a_\rep(n)$ averaged over $n \approx \mathfrak{n}$ and $\qrep \approx \qmur$ is asymptotically a nonzero function of only $\mathfrak{n}/\qmur$. We refer to this as \textit{scale invariance}.

In this paper, we'll say that asymptotics of the form
\begin{align}
  \label{eq:murmurations_def_intro}
  \frac{1}{\#\cF(\qmur)}\sum_{\rep \in \cF(\qmur)} \frac{1}{\sqrt{\#\cS(\nmur)}} \sum_{n \in \cS(\nmur)} a_\rep(n) \,\sim\, M\!\left(\frac{\nmur}{\qmur}\right)
\end{align}
are \textit{murmurations}.
The left hand side is an appropriately normalized sum of Dirichlet coefficients $a_\rep(n)$ over sets $\cF(\qmur)$ of $L$-functions with conductor roughly $\qmur$ and sets $\cS(\nmur)$ of integers roughly $\nmur$,
and the asymptotic value on the right hand side depends only on the ratio $\nmur/\qmur$ --- i.e.\ is scale-invariant --- as $\nmur, \qmur \to \infty$. The essential characteristic is that the average is taken with the horizontal and vertical aspects coupled.
We give a more precise formulation in \cref{sec:outline_construction}, and we give the conversion to ``murmuration densities'' in the sense of \cite{sarnak_letter} in \cite{ratiosconjecture}.

Theoretical results about statistics of $a_\rep(n)$'s mixing the horizontal and vertical aspects are quite elusive. Determining the statistical structure in the scale invariant regime, even conditionally, has been done in only a few cases: Wang \cite{wang} and Zubrilina \cite{zubrilina} have done so unconditionally for Hecke characters of imaginary quadratic fields and newspaces of holomorphic modular forms in the level aspect respectively, while Lee--Oliver--Pozdnyakov \cite{LOP}, Bober--Booker--Lee--Lowry-Duda \cite{bblld}, and Booker--Lee--Lowry-Duda--Seymour-Howell--Zubrilina \cite{blldshz} prove, under GRH, results for quadratic Dirichlet characters, newspaces of holomorphic modular forms in the weight aspect, and Maass forms respectively. Otherwise, it is not known, even heuristically, what the values of these averages of Dirichlet coefficients should be.

In \cite{ratiosconjecture} we connect the scale invariance of averages of Dirichlet coefficients to \textit{ratios conjectures}, from random matrix theory and rooted in the
work of Conrey--Farmer--Keating--Rubinstein--Snaith \cite{cfkrs}. This connection enabled us to exhibit murmurations and other scale-invariant phenomena in many cases by leveraging existing results in random matrix theory. Most of these existing results are conditional on ratios conjectures, two are conditional on GRH, and one is unconditional. 

There is a very general ``recipe'' \cite{conrey_snaith} for producing the ratios conjectures which the method presented in \cite{ratiosconjecture} takes as input. Ratios conjectures estimate distributions of low-lying zeros in families of $L$-functions, and the connection between distributions of zeros and statistics of Dirichlet coefficients is well known.
It is not clear, however, why random matrix theory's estimates of zero distributions consistently lead to scale-invariance.

The current paper explains that
the observed scale invariance comes from a particular term in the approximate functional equation.
\Cref{sec:outline_construction}
elaborates; see \eqref{eq:afe_murmurations_asymp}, \eqref{eq:afe_murmurations_asymp_smoothed}, and \cref{rem:murmurations_sqrt} 
in particular. The unusual normalization in \eqref{eq:murmurations_def_intro}, observed by other authors in the guise of normalizing Dirichlet coefficients so that the Ramanujan--Petersson reads $\theta = 1/2$, is another salient characteristic of murmurations. \Cref{rem:murmurations_sqrt} explains that the fact that $\theta = 1/2$ is the correct normalization is a consequence of the critical strip being centred at $\sigma = 1/2$ in the analytic normalization.

Our method for constructing murmurations in the present paper goes beyond the framework of the ratios conjecture; it requires only a suitable family, and needn't consider ratios. It's more in line with the recipe presented in \cite[\S 4]{cfkrs}. 

We implement our construction for murmurations in the case of quadratic twist families of $\mathrm{GL}_1$ automorphic representations
\hypertargetc{hyper:gl1repdef}{$\rep = |\mspace{0mu}\cdot\mspace{0mu}|^{i\tau}\chi$}, where $\chi$ is an arbitrary primitive Dirichlet character (possibly the trivial character), and $\tau$ is an arbitrary real number.
\Cref{thm:quadratic_murmurations} uses a sharp cutoff, while \cref{thm:gamma_murmurations} allows for weights.
To state our results we must first introduce some notation.

Let $\rep$ be an irreducible unitary cuspidal automorphic representation on $\mathrm{GL}_\repdim(\Q)$ with $L$-function
\begin{align}
  \label{eq:rep_Lfunc_series}
  &
  \begin{aligned}
  L(s,\rep)
  &= \sum_{n=1}^\infty \frac{a_\rep(n)}{n^{s}}
  \\
  &= \prod_p \prod_{j=1}^\repdim \big(1 - \alpha_j(p) p^{-s}\big)^{-1}
  \end{aligned}
  \shortintertext{
    satisfying the functional equation
  }
  \label{eq:rep_Lfunc_functional_equation}
  \left(\frac{\qrep}{\pi^\repdim}\right)^{\!\frac{s}{2}}\prod_{j=1}^\repdim \halfGamma{s + \kappa_j}
  \,&L(s,\rep)
  =
  \rootnum_\rep \left(\frac{\qrep}{\pi^\repdim}\right)^{\!\frac{1-s}{2}}\prod_{j=1}^\repdim \halfGamma{1 - s + \bar\kappa_j} L(1-s,\bar\rep)
  ,
\end{align}
normalized such that $a_\rep(1) = 1$.
Let \hypertargetc{hyper:thetadef}$\theta$ be such that $a_\rep(p) \ll p^\theta$ and $\Re(\kappa_j) \geqslant -\theta$. (The Ramanujan--Petersson conjecture is that one may take $\theta = 0$ always.) When $\repdim = 1$, we'll write
\hypertargetc{hyper:kappadef}$\kappa_1 \eqqcolon \kappa$.

%
Let $\qstar \coloneqq 4\qrep$ if $2 \parallel \qrep$ and $\qstar \coloneqq \qrep$ otherwise.
Pick $\ell \in (\Z/\qstar)^\times$ such that there exists a positive fundamental discriminant $d = \ell \mod \qstar$.
Fix $1 < D_0 < D$ and set
\begin{align}
  \label{eq:Fdef}
  \cF \coloneqq \big\{D_0 < d < D \,:\, \text{$d$ a fundamental discriminant},\,\, d = \ell \mod \qstar \big\}
  .
\end{align}

Let $\chi_d$ denote the Kronecker symbol $\left(\frac{d}{\cdot}\right)$, and let $\rootnum_{\rep\otimes\chi_d}$ denote the root number of the Rankin--Selberg convolution $L(s,\rep\otimes\chi_d)$. Define
\begin{align}
  \label{eq:rootnumFdef}
  \omega_\cF \coloneqq \Favg \rootnum_{\rep\otimes\chi_d}
  .
\end{align}
When $\rep = |\mspace{0mu}\cdot\mspace{0mu}|^{i\tau}\chi$ and $\tau \Delta D \ll D$, one has
$\kappa = \frac{1 - \chi(-1)}{2} - i\tau$ and
\begin{align*}
  \omega_\cF
  = i^{-\frac{1 - \chi(-1)}{2}}\frac{\tau(\chi)}{\sqrt{\qrep}} \!\left(\frac{qD}{\pi}\right)^{\!i\tau} \!\!\chi_d(\qrep) \chi(d)
  + O\!\left(\frac{\tau \Delta D}{D}\right)
  \!,
\end{align*}
where $\tau(\chi)$ is the Gauss sum. Note that the value of $\chi_d(\qrep) = \big(\frac{\ell}{\qrep}\big)$ is independent of $d \in \cF$.


Throughout the paper we will use the common convention that each statement in which \hypertargetc{hyper:epsdef}{$\eps$} appears holds for all sufficiently small positive real numbers $\eps$, with a threshold which may change from line to line, but not within a single equation.

\begin{theorem}[Murmurations for quadratic twists of a $\mathrm{GL}_1$ automorphic representation]
  \label{thm:quadratic_murmurations}
  Let $\tfrac{5}{6} < \delta < 1$.
  In the range
  $(qD)^{1 - \eps} \ll \x \ll (qD)^{1+\eps}$,
  $\,D^{\delta - \eps} \ll \#\cF \ll D^{\delta + \eps}$,
  $\,\tau \ll D^{1 - \delta - \eps}$,
  \begin{align*}
    \frac{1}{\#\cF} \sum_{d \in \cF}
    &\frac{1}{\sqrt{\x}}
    \sum_{n < \x} n^{i\tau}
    \chi(n) \chi_d(n)
    \\
    ={}&
    \frac{\rootnum_{\cF}}{2\pi i}
    \int_{\frac{3}{4} - i\infty}^{\frac{3}{4} + i\infty}
    \frac{\halfGamma{1-s+\bar\kappa}}{\halfGamma{s+\kappa}}
    \frac{L(2-2s+2i\tau, \bar{\chi}^2)}{L^{(2)}(3-2s+2i\tau, \bar{\chi}^2)}
    \prod_{\mspace{1.75mu}p \nmid 2q} \!\left(1 - \frac{1}{(p+1)(1 - \chi^2(p)p^{3-2s+2i\tau})}\right)
    \left(\frac{\pi \x}{q D}\right)^{\!s-\half} \frac{ds}{s}
    \\&
    {}+{}
    \one{\text{\emph{$\bar\chi^2$ is trivial}}}\mspace{2mu}
    \rootnum_{\cF}
    \frac{6}{\pi^2}\frac{1 + \frac{1}{5}\one{2\nmid q}}{1 + 2i\tau}
    \left(\frac{\pi \x}{q D}\right)^{\!i\tau}
    \prod_{p\mid q}\frac{p}{p+1}
    \prod_{\mspace{1.75mu}p \nmid q} \!\left(1 + \frac{1}{(p + 1)(p^2 - 1)}\right)
    \\&
    {}+{}\mspace{0mu}
    O\big(\qrep^{\frac{1}{6}} D^{\frac{1}{12}\rho}
    + \qrep^{\frac{3}{4}} D^{\rho} + |\tau| D^{\delta - 1}
    \big)(qD)^\eps
    ,
  \end{align*}
  where $\rho = -\frac{1}{14} + \frac{1}{8}\!\left(\delta - \frac{13}{14}\right) + \frac{7}{8}\!\left|\mspace{1.5mu}\delta - \frac{13}{14}\right|$.
\end{theorem}

\begin{figure}[H]
  \includegraphics[width=\textwidth]{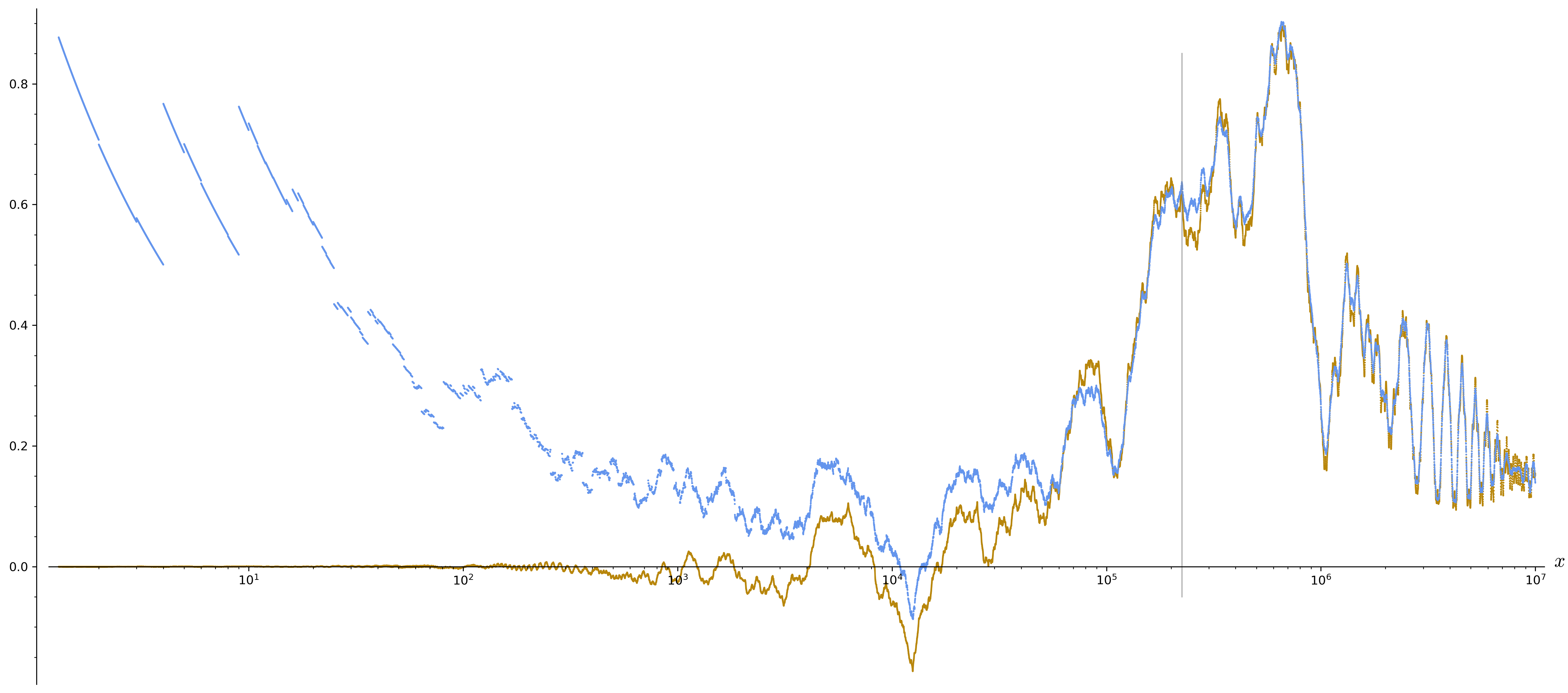}
  \begin{figurecap}
    \label{fig:kronecker_murmuration}
    Real parts of the left (blue) and right (gold) hand sides of \cref{thm:quadratic_murmurations} as functions of $\x$, with
    $\tau = 2$,
    $\chi \mod 7$ sending $3 \mapsto \frac{1 + i\sqrt{3}}{2}$, 
    and
    $\cF = \big\{99,\mspace{-2.5mu}000 < d < 101,\mspace{-2.5mu}000 \,:\, \text{$d$ a fundamental discriminant},\,\, d = 1 \mod 7 \big\}$.
    The vertical grey line indicates the value of $qD/\pi$.
    The right hand side was estimated with a Riemann sum over $|\Im(s)| \leqslant 1,\mspace{-2.5mu}000$ sampled at $200,\mspace{-2.5mu}001$ evenly spaced points, and the products over primes truncated at $p < 30,\mspace{-2.5mu}000$.
    In this example, $\#\cF = 89$.
  \end{figurecap}
\vspace{-\baselineskip}
\end{figure}

\begin{theorem}
  \label{thm:gamma_murmurations}
  Let $\mf$ be a meromorphic function with inverse Mellin transform $\cM^{-1}\mspace{-2mu}\{\mf\}$. Suppose that $\mf$ is holomorphic in the strip $0 < \sigma < 1$ and that $f(\sigma + it) \ll |t|^{-2 - \eps}$ as $|t| \to \infty$ uniformly for $0 < \sigma < 1$.
  For any $\tfrac{3}{4} < \delta < 1$,
  in the range
  $(qD)^{1 - \eps} \ll \x \ll (qD)^{1+\eps}$,
  $D^{\delta - \eps} \ll \#\cF \ll D^{\delta + \eps}$,
  and
  $\tau \ll D^{1 - \delta - \eps}$,
  \begin{align*}
    \frac{1}{\#\cF} \sum_{d \in \cF}
    &\frac{1}{\sqrt\x}
    \sum_{n=1}^\infty \cM^{-1}\mspace{-2mu}\{\mf\}\mspace{-2mu}\!\left(\frac{n}{x}\right)n^{i\tau}\chi(n)\chi_d(n)
    \\
    ={}
    &
    \frac{\rootnum_{\cF}}{2\pi i}\int_{\eps - i\infty}^{\eps + i\infty}
    \frac{\halfGamma{1-s+\bar\kappa}}{\halfGamma{s+\kappa}}
    \frac{L(2-2s+2i\tau, \bar{\chi}^2)}{L^{(2)}(3-2s+2i\tau, \bar{\chi}^2)}
    \prod_{\mspace{1.75mu}p \nmid 2q} \!\left(1 - \frac{1}{(p+1)(1 - \chi^2(p)p^{3-2s+2i\tau})}\right)
    \left(\frac{\pi \x}{q D}\right)^{\!s - \half} \!\mf(s)\mspace{1mu}ds
    \\
    &
    + O\big(
    \qrep^{\frac{3}{4}}
    D^{-\frac{1}{8} + \left|\delta - \frac{7}{8}\mspace{-2mu}\right|}
    (1 + \one{\delta < \tfrac{7}{8}} |\tau|^{\frac{1}{4}})
    + |\tau| D^{\delta - 1}
    \big)(\tau qD)^\eps
    .
  \end{align*}
\end{theorem}

On the left hand sides of both \cref{thm:quadratic_murmurations,thm:gamma_murmurations} are averages of Dirichlet coefficients in the scale-invariant regime, balancing the horizontal and vertical aspects governed by $x$ and $D$ respectively.

The right hand sides' main terms are inverse Mellin transforms of meromorphic functions, and depend only on the scale-invariant parameter $\pi x/qD$. This is a signature characteristic of murmurations.

The error term in \cref{thm:quadratic_murmurations} is $D^{-1/168}$ at its smallest. \Cref{thm:gamma_murmurations} features the substantially smaller $D^{-1/8}$, thanks to the smoothing present in its main terms.

Existing murmuration results have focused on averages short in the horizontal aspect.
One may evaluate \cref{thm:gamma_murmurations} at two different values of $\x$ and take the difference. For any $\tfrac{7}{8} < \eta < 1$, a horizontal sum of length essentially $D^\eta$ can be obtained in this way.
The corresponding ``murmuration density'' in the sense of \cite{sarnak_letter} can be determined from \cite[\S 2]{ratiosconjecture}.

\subsection{Mean values of $L$-functions in families}
\label{sec:intro_mean_values}

Our proofs of \cref{thm:quadratic_murmurations,thm:gamma_murmurations} proceed by
evaluating expressions of the form
\begin{align}
  \label{eq:intro_mean_value}
  \frac{1}{2\pi i}\int_{c - iT}^{c + iT} \frac{1}{\#\cF} \sum_{\rep \in \cF} L(s,\rep)\mspace{2mu}\x^s\,\frac{ds}{s}
\end{align}
in two different ways: by using Mellin inversion in each individual term, and then separately by estimating the sum in the integrand using the approximate functional equation. The latter is a manifestation of the problem of estimating mean values of $L$-functions. Our approach has not appeared outside our companion paper \cite{ratiosconjecture}. Existing results have instead hinged on trace formulas.

The problem of estimating mean values of $L$-functions in families has a long history. We begin our overview with
Jutila \cite{jutila}, who studies $L$-functions of primitive quadratic Dirichlet characters at the critical point $s = \thalf$.
Their work was generalized by Stankus \cite{stankus}, who restricted to fundamental discriminants $d$ lying in an arithmetic progression $d = a \mod q$ with $q$ odd and the gcd $(a,q) = 1$. The smoothed first moment
for $d = 8 \mod 16$ was studied using different methods in \cite{young:first_moment}.

In \cite{tv}, Takhtadzhyan and Vinogradov prove results for the family $L(s,\chi_d)$ of primitive quadratic Dirichlet characters $\chi_d$ with prescribed parity and $|d| < D$, valid not just at the critical point $s = \thalf$, but in the larger region $s \ll D^{\frac{1}{4}}$. Goldfeld and Hoffstein \cite{goldfeld_hoffstein} also prove results for $s \neq \thalf$.
Elliott \cite{elliott} considers a closely related problem and obtains results for $s \ll D^{\frac{1}{13}}$.

Mean values of $\mathrm{GL}_2$ $L$-functions in quadratic twist families are studied in \cite{murty_murty, iwaniec, bfh:1989, bfh:1990, rohrlich, fh, rs, quanli}.
Higher moments of $\mathrm{GL}_1$ $L$-functions are considered in \cite[etc.]{dgh,sound:2000,sono}, and the second moment of $\mathrm{GL}_2$ $L$-functions in \cite{sy,li}.
The second moment of a $\mathrm{GL}_1$ $L$-function can be viewed as the first moment of a $\mathrm{GL}_2$ Eisenstein series.

Moments of $L$-functions in quadratic twist families have various applications, including subconvexity bounds \cite{petrow_young:2020, petrow_young:2023}, non-vanishing \cite{murty_murty, bfh:1989, rohrlich, sound:2000}, and progress towards the Ramanujan--Petersson conjecture \cite{lrs:1995, lrs:1996, di1, di2, di3, di4}.

Precise predictions for the mean values of quadratic twists of some automorphic $L$-functions 
were made by Conrey--Farmer--Keating--Rubinstein--Snaith in \cite[\S 4.4]{cfkrs}. Their conjectures are rooted in random matrix theory, and are highly nontrivial.

None of these many pre-existing results are suitable for our purposes however.
For 
applications to murmurations and the proofs of \cref{thm:quadratic_murmurations,thm:gamma_murmurations},
it'll be critical to have mean value estimates which
\begin{enumerate}[label=(\roman*)]
\item
  hold for $|s|$ large, and
\item
  pertain to families of fundamental discriminants $d$ in a a narrow range of absolute values and in a prescribed arithmetic progression,
\end{enumerate}
Additionally, it'll be beneficial to
\begin{enumerate}[resume,label=(\roman*)]
\item
  avoid restrictions to only $d = 8 \mod 16$ or similar, and
\item
  avoid weights, using a sharp cutoff instead.
\end{enumerate}
\Cref{thm:r1size,thm:r2size} give the mean values of quadratic twists of automorphic $L$-functions uniformly for all $\mathrm{GL}_1$ and $\mathrm{GL}_2$ $L$-functions.
These theorems satisfy each of (i)---(iv) above, with particularly marked improvement with regards to (i).

Define
\begin{align}
  \label{eq:repL2def}
  L(s,\repsq)
  &\coloneqq \sum_{n = 1}^\infty \frac{a_{\rep}(n^2)}{n^{s}}
  .
\end{align}

\begin{theorem}[Mean values in $\mathrm{GL}_1$ quadratic twist families]
  \label{thm:r1size}
  With the notation above around
  any automorphic representation $\rep$ on $\mathrm{GL}_1$, and
  for any $s = \sigma + it$ such that
  $\thalf < \sigma < 1$,
  \begin{align*}
    \Favg
    L(s,\rep\otimes\chi_d)
    &=
    \frac{L(2s, \repsq)}{L^{(2\qrep)}(2s+1, \repsq)}
    \mspace{1mu}
    \,
    \prod_{\mspace{1.75mu}p \nmid 2\qrep} \left(1 - \frac{1}{(p+1)(1 - \alpha(p)^{-2}p^{2s+1})}\right)
    \\
    &+
    O\Bigg(
    \frac{\qrep^{\frac{3}{4}} D^\half}{\#\cF}
    \sqrt{D |s - \kappa|}^{\,\frac{1}{1 + 2\sigma} + \eps}
    +\,
    \qrep^{\frac{3}{4}} \bigg(\frac{D^\half}{\#\cF}\bigg)^{\!\frac{2\sigma - 1}{2\sigma}}
    \Bigg)
    (qD)^\eps
    .
  \end{align*}
\end{theorem}

\Cref{thm:r1size}'s error terms balance at 
$\#\cF \asymp_\qrep D^\half (D |s - \kappa|)^{\frac{\sigma}{2\sigma + 1}}$.

\begin{theorem}[Mean values in $\mathrm{GL}_2$ quadratic twist families]
  \label{thm:r2size}
  With the notation above around
  any automorphic representation $\rep$ on $\mathrm{GL}_2$, and
  for any $s = \sigma + it$ such that
  $\thalf + \theta < \sigma < 1$,
  \begin{align*}
    \Favg
    L(s,\rep\otimes\chi_d)
    &=
    \sum_{n=1}^\infty \frac{a_\rep(n^2)}{n^{2s}} \!\prod_{p \mid \frac{n}{(n,2q)}} \frac{p}{p+1}
    \:+\:
    O\Bigg(
    \frac{\qrep^{\frac{3}{4}} D^\half}{\#\cF}
    \bigg(D \prod_{j=1,2} |s - \kappa_j|^{\frac{1}{2}}\bigg)^{\!\frac{1 + 2\theta}{1 + 2\sigma + 2\theta} + \eps}
    (qD)^\eps
    \Bigg)
    .
  \end{align*}
\end{theorem}

The regions in which \cref{thm:r1size,thm:r2size} yield asymptotics are very large. For example, in \cite{tv}, which pertains to $\rep$ the trivial representation on $\mathrm{GL}_1$, the admissible range is at most $s \ll \x^{\frac{1}{4} - \eps}$, while \cref{thm:r1size} allows up to $s \asymp \x^{2-\eps}$. The main result of \cite{elliott}, which addresses a similar problem, holds for $s \ll \x^{\frac{1}{13} - \eps}$ inside the critical strip.

The methods we used to prove \cref{thm:r1size,thm:r2size}
are invaluable in establishing \cref{thm:quadratic_murmurations}: we require at minimum an asymptotic for the mean value of $L$-functions in our family
which is
valid for $s \gg \x^{\half + \eps}$, a range much larger than what any existing results address. Indeed, for our result to be compelling, each of the limitations (i)---(iv) listed in \cref{sec:intro_mean_values} needed to be removed. In so doing, we have generalized and/or sharpened a couple existing results on sums of fundamental discriminants or quadratic characters.

\subsection{Techniques and ingredients}
\label{sec:intro_techniques}


To establish our results we introduce a variation of the approximate functional equation.
This variation imbues in the weights a holomorphic dependence on $s$ which dynamically rebalances the lengths of the two sums so as to increase the power savings in \cref{thm:quadratic_murmurations} from $1/216$ to $1/168$, and the admissible range of $\delta$, the exponent of the size of the family, from $7/8 < \delta < 1$ to $5/6 < \delta < 1$.

\Cref{sec:lemmas_afe} discusses the approximate functional equation in detail.
Briefly, for readers already familiar with approximate functional equations, we'll arrange for it to be the case that
\hypertarget{hyper:cdef}{}
\begin{align*}
  L(s,\rep)
  &
  ={} \sum_{n=1}^\infty \frac{a_\rep(n)}{n^{s}}
  V_s\!\left(\frac{\pi^{\frac{\repdim}{2}}}{\qrep^\half\xi}n\right)
  \;+\;
  \rootnum_\rep 
  \left(\frac{\pi^\repdim}{\qrep}\right)^{\!s - \half} 
  \prod_{j=1}^\repdim \frac{\halfGamma{1-s+\bar\kappa_j}}{\halfGamma{s+\kappa_j}}
  \sum_{n=1}^\infty \frac{a_{\bar\rep}(n)}{n^{1-s}}
  V_{1-s}^*\!\left(\frac{\pi^{\frac{\repdim}{2}}\xi}{\qrep^\half}n\right)
  ,
\end{align*}
where
\begin{align}
  \nonumber
  \xi \coloneqq \xi(s) &\coloneqq D^{2\alpha} \exp\!\left(\beta\log\!\left((s - s_0)^2\right)\right)\!
  ,
\end{align}
with the branch cut of the logarithm along the negative real axis and $-\pi < \arg(z) \leqslant \pi$.

The real number parameters $\alpha$ and $\beta$ allow us to rebalance the two sums in the approximate functional equation to reduce error terms. The idea to introduce the role played by $D^{2\alpha}$ is not new; see e.g.\ \cite[Thm.\ 5.3]{IK}. It is, however, necessary for establishing \cref{thm:r2size}, as one does not obtain a power savings otherwise.
The parameter $\beta$ is a $t$-aspect analogue of $\alpha$, 
which we haven't encountered before.

It is critical for our application to murmurations that the holomorphicity in $s$ of the weights in the approximate functional equation be preserved by our modifications, as \cref{thm:quadratic_murmurations} is established by taking an inverse Mellin transform. Seeking only to evaluate $L$-functions at a single point, e.g.\ as in \cref{thm:r1size,thm:r2size} or \cite[\S\S 3.3--3.5]{rubinstein:methods}, is substantially less delicate; our proof of \cref{thm:quadratic_murmurations} requires fixing $V_s$ once and for all at the very beginning, and $V_s$ being able to take care of itself with regards to matters of error term minimization instead of requiring our continued intervention is what leads to power savings.

In \cref{thm:r1size,thm:r2size,thm:quadratic_murmurations}, the parameters $\alpha$ and $\beta$ have been ``optimized out'', in that the optimal choice was made as part of the theorem. The maximum power savings in \cref{thm:quadratic_murmurations} takes $\beta = 1/16$. In contrast, the parameter $\alpha$ is not helpful for proving \cref{thm:quadratic_murmurations}, but critical for \cref{thm:r2size}.

A further optimization made while preparing \cref{thm:quadratic_murmurations,thm:gamma_murmurations} was to decompose mean values of $L$-functions into multiple holomorphic terms through successive approximations, and then to integrate those terms along different vertical lines. Without this trick the theorems' error terms dwarfs their main terms.

Over the course of the proof of \cref{thm:quadratic_murmurations} it becomes necessary to generalize results of Fainleib--Saparniyazov \cite{FS}, Jutila \cite{jutila}, and Stankus \cite{stankus} on sums of quadratic characters or powers of fundamental discriminants. In \cref{sec:lemmas_discs} we extend these existing results to apply to arbitrary arithmetic progressions of fundamental discriminants, shedding existing restrictions on $d \mod 16$. Murmurations are rescaled according to the root number of the family of $L$-functions --- note the factor of $\rootnum_{\cF}$ on the right hand sides of \cref{thm:quadratic_murmurations,thm:gamma_murmurations} --- so maximal control is needed over which root numbers appear and in what proportion. As will be discussed later, this comes down to confining $d$ to a union of arithmetic progressions. We also extend the admissible range of the real part of the power in the discriminant sum. 

In addition to generalizing these results, \cref{lemma:disc_power_sum} improves the error terms of the Jutila and Stankus results, allowing us to save a factor of $D^{\frac{10}{21}}$ in the error term of \cref{thm:quadratic_murmurations}. We ultimately save no more than a power of $D^{\frac{1}{168}}$, so in the context of this problem a factor of $D^{\frac{10}{21}}$ is big. Without our improvements, proving \cref{thm:quadratic_murmurations} would have been impossible. Our proof of \cref{lemma:disc_power_sum} uses different methods than Jutila and Stankus.

Producing the error term in \cref{thm:quadratic_murmurations} required us to balance $47$ error terms in 10 parameters. We wrote code to help us do this, which can be found at \cite{github}. More detail is given in \cref{sec:contrib}. Our code also helped guide the optimization of parameters in the proofs of \cref{thm:r1size,thm:r2size}.

\Cref{thm:quadratic_murmurations} requires a subconvexity bound. We use recent work of Petrow--Young \cite{petrow_young:2023} for this purpose.

\subsection{Outline and setup}
\label{sec:intro_outline} 

Our overall strategy for proving our main results 
\cref{thm:quadratic_murmurations,thm:gamma_murmurations} is presented in \cref{sec:outline}.
\Cref{thm:r1size,thm:r2size} are easier and will be proved along the way.
\Cref{sec:lemmas} establishes a litany of lemmas we'll need in \cref{sec:size}, which builds to the proofs of \cref{thm:r1size,thm:r2size} by analyzing averages of individual terms in the approximate functional equation, and in \cref{sec:contrib}, which takes an inverse Mellin transform of these individual terms to prove \cref{thm:quadratic_murmurations,thm:gamma_murmurations}. Many proofs are postponed to \cref{sec:lemma_proofs}. \Cref{sec:glossary} is a glossary, to remind notation.

Throughout this paper $\rep$ denotes an irreducible unitary cuspidal automorphic representation on $\mathrm{GL}_\repdim(\Q)$ with $L$-function \eqref{eq:rep_Lfunc_series} satisfying the functional equation \eqref{eq:rep_Lfunc_functional_equation}. These equations define the quantities $a_\rep(n)$, $\alpha_j(p)$, $\qrep$, $\rootnum_\rep$, and $\kappa_j$.

The completed $L$-function of $\rep$, defined by
\begin{align}
  \label{eq:Lambdadef}
  \begin{aligned}
    \Lambda(s,\rep) &\coloneqq \left(\frac{\qrep}{\pi^\repdim}\right)^{\!\frac{s}{2}} \prod_{j=1}^\repdim \halfGamma{s + \kappa_j} L(s,\rep)
    \\
    \Lambda(1-s,\bar\rep) &\coloneqq \left(\frac{\qrep}{\pi^\repdim}\right)^{\!\frac{1-s}{2}} \prod_{j=1}^\repdim \halfGamma{1 - s + \bar\kappa_j} L(1-s,\bar\rep)
    ,
  \end{aligned}
\end{align}
is entire unless $\repdim = 1$ and $\rep$ is everywhere unramified, in which case $L(s,\rep) = \zeta(s + i\tau)$ for some $\tau \in \R$ \cite[\S 1.1.1]{michel}.
We use the ``analytic normalization'', wherein the functional equation $\Lambda(s,\rep) = \rootnum_\rep \Lambda(1-s,\bar\rep)$ is between $s$ and $1-s$, and conjecturally $a_\rep(p) \ll 1$. Note that, with our definition \eqref{eq:Lambdadef},
\begin{align}
  \label{eq:notLambda}
  \Lambda(s,\rep) \neq \qrep^{\frac{s}{2}} \prod_{j=1}^\repdim \pi^{-\frac{s + \kappa_j}{2}}\halfGamma{s + \kappa_j} L(s,\rep)
  .
\end{align}
The right hand side above is natural from a theoretical perspective: the factor $\prod_{j=1}^\repdim \pi^{-\frac{s + \kappa_j}{2}}\thalfGamma{s + \kappa_j}$ is the local $L$-function at $\infty$. However, it would lead to more cumbersome notation for us in practice: the only difference is that what we call $\rootnum_\rep$ is equal to $\tilde{\rootnum}_\rep \pi^{\half\sumr \kappa_j - \bar\kappa_j}$, where $\tilde{\rootnum}_\rep$ is the root number one would obtain if they were to define $\Lambda$ to be the right hand side of \eqref{eq:notLambda}.

\begin{example}[$\mathrm{GL}_1$ automorphic representations]
  \label{example:GL1}
  If $\rep$ is an irreducible unitary cuspidal automorphic representation on $\mathrm{GL}_1(\Q)$, then $\rep = |\mspace{0mu}\cdot\mspace{0mu}|^{i\tau}\chi$ for some primitive Dirichlet character $\chi$ and some real number $\tau$;
  \begin{align*}
    \Lambda(s,\rep)
    &=
    \left(\frac{\qrep}{\pi}\right)^{\!\frac{s}{2}} \halfGamma{s - i\tau + \kappa_\chi} L(s - i\tau, \chi)
    \\
    &=
    \rootnum_\chi \left(\frac{\qrep}{\pi}\right)^{\!i\tau} \left(\frac{\qrep}{\pi}\right)^{\!\frac{1-s}{2}} \halfGamma{1 - s + i\tau + \kappa_\chi} L(1 - s + i\tau, \bar\chi)
    ,
  \end{align*}
  where $\kappa_\chi = \frac{1 - \chi(-1)}{2}$
  and $\rootnum_\chi = i^{-\kappa_\chi} \frac{\tau(\chi)}{\sqrt{q}}$ where $\tau(\chi)$ is the Gauss sum. I.e.,
  $\rootnum_\rep = \rootnum_\chi (\tfrac{\qrep}{\pi})^{i\tau}$
  and $\kappa = \frac{1 - \chi(-1)}{2} - i\tau$.
\end{example}

\begin{remark}
  \label{rem:theta}
  Let $\theta$ be such that $a_\rep(p) \ll p^{\theta}$ and $\Re(\kappa_j) \geqslant -\theta$ for all $1 \leqslant j \leqslant \repdim$.
\begin{itemize}
\item
  When $\repdim = 1$ one may take $\theta = 0$, as $a_\rep(n) = |n|^{i\tau}\chi(n)$ for some Dirichlet character $\chi$ and real number $\tau$ \cite[\S 2]{goldfeld_hundley}.
\item
  When $\repdim = 2$, one may take $\theta = \tfrac{7}{64}$ \cite{kim_sarnak}. If $\rep$ corresponds to a holomorphic modular form then one may take $\theta = 0$ \cite{deligne}. See \cite[\S 5.11]{IK} for further discussion of the analytic theory of $\mathrm{GL}_2$ $L$-functions.
\item
  For any $\repdim \in \Z_{>0}$, one may take $$\theta = \half - \frac{1}{\repdim^2 + 1}.$$ See \cite[Thm.\ 1.1]{michel} for a history of this result.
\item
  The Ramanujan--Petersson conjecture is that one may take $\theta = 0$ in every circumstance. (More precisely, at the finite places it states that $|\alpha_j(p)| = 1$ for all $p \nmid \qrep$.)
\end{itemize}
The Ramanujan--Petersson conjecture at the infinite place is known as the Selberg eigenvalue conjecture. See \cite[\S 5.1]{IK} or \cite[\S 1.1.3]{michel} for further discussion.
\end{remark}


\section{Strategy}
\label{sec:outline}

We prove \cref{thm:quadratic_murmurations} by evaluating
\begin{align}
  \label{eq:avg_mellin}
  \frac{1}{2\pi i}\int_{c - iT}^{c + iT} \frac{1}{\#\cF} \sum_{\rep \in \cF} L(s,\rep)\mspace{2mu}\x^s\,\frac{ds}{s}
  ,
\end{align}
i.e.\ the average over a family of $L$-functions of a truncated inverse Mellin transform, in two different ways:
\begin{enumerate}[label=(\roman*)]
\item
  Using the approximate functional equation, to express \eqref{eq:avg_mellin} as an explicit meromorphic function.
\item
  Using Mellin inversion, e.g.\ Perron's formula, to express \eqref{eq:avg_mellin} as an average of Dirichlet coefficients.
\end{enumerate}
\Cref{subsec:afe,subsec:mellin} delve into (i) and (ii) respectively.
How they are combined to produce murmurations is discussed in \cref{sec:outline_construction}.
\Cref{sec:rep_quadratic_twists} focuses on the special case of quadratic twist families, which will be the subject of the paper from thereon out. 

\subsection{The approximate functional equation}
\label{subsec:afe}

The specific approximate functional equation we use,
valid whenever $\Lambda(s,\rep)$ is entire and $0 < \Re(s) < 1$, is
\begin{align}
  \label{eq:afe_specific}
  L(s,\rep)
  &
  ={} \sum_{n=1}^\infty \frac{a_\rep(n)}{n^{s}}
  V_s\!\left(\frac{\pi^{\frac{\repdim}{2}}}{\qrep^\half\xi}n\right)
  \;+\;
  \rootnum_\rep 
  \left(\frac{\pi^\repdim}{\qrep}\right)^{\!s - \half} 
  \prod_{j=1}^\repdim \frac{\halfGamma{1-s+\bar\kappa_j}}{\halfGamma{s+\kappa_j}} 
  \sum_{n=1}^\infty \frac{a_{\bar\rep}(n)}{n^{1-s}}
  V_{1-s}^*\!\left(\frac{\pi^{\frac{\repdim}{2}}\xi}{\qrep^\half}n\right)
  ,
\end{align}
where $V_s$ is smoothing function and $\xi$ is essentially a small power of $s$.
We now define $V$ and $\xi$.

We use the common shorthand
\begin{align}
  \label{eq:vertical_line_integral}
  \int_{(c)} \coloneqq \int_{c - i\infty}^{c + i\infty}
\end{align}
for arbitrary $c \in \R$. 

Let $A, B, \alpha, \beta, c_V, s_0 \in \R$. We discuss these parameters later. 
Let $\log$ denote the branch of the logarithm cut along the negative real axis with $-\pi < \arg(z) \leqslant \pi$.
The definitions of $\xi$ and $V$ are
\begin{align}
  \nonumber
  \xi &\coloneqq \xi(s)
  \coloneqq D^{2\alpha} \exp\!\left(\beta \log\!\left((s - s_0)^2\right)\right)
  \\
  \nonumber
  V_s(y) &\coloneqq \frac{1}{2\pi i} \int_{(c_V\!)} \prodr \frac{\halfGamma{s+w+\kappa_j}}{\halfGamma{s+\kappa_j}} \left(e^{\frac{iw}{A}} + e^{-\frac{iw}{A}} - 1\right)^{\!-B} y^{-w}\frac{dw}{w}
  \\
  \nonumber
  V_{1-s}^*(y) &\coloneqq \frac{1}{2\pi i} \int_{(c_V\!)} \prodr \frac{\halfGamma{1-s+w+\bar\kappa_j}}{\halfGamma{1-s+\bar\kappa_j}} \left(e^{\frac{iw}{A}} + e^{-\frac{iw}{A}} - 1\right)^{\!-B} y^{-w}\frac{dw}{w}
  .
\end{align}

A more general version of 
\eqref{eq:afe_specific} is given in \cref{lemma:afe_psi}.
\Cref{lemma:V_is_holomorphic,lemma:V_approx,lemma:V_deriv_approx} list the essential characteristics of 
$V$. In particular,
\begin{align*}
  V_s\!\left(\frac{\pi^{\frac{\repdim}{2}}}{\qrep^\half\xi}n\right) \approx \onelr{n < Q^\half \xi}
\end{align*}
where
\begin{align}
  \label{eq:Qdef}
  Q \coloneqq Q(s) &\coloneqq \frac{\qrep}{(2\pi e)^\repdim} \prod_{j=1}^\repdim |s + \kappa_j|
\end{align}
is basically the ``analytic conductor'' of $L(s, \rep)$ \cite[(5.7)]{IK} reduced by a factor of $(2\pi e)^\repdim$.

The parameters $A, B, \alpha, \beta, c_V, s_0$ control the shapes of $V$ and $\xi$, and affect only the quality of the error term. We impose restrictions on these parameters in \cref{sec:lemmas_afe}. If $\repdim = 1$, one may take $A = 2$, $s_0 = -1$, $c_V = 2$, $B = 16$, and $|\alpha|, |\beta| < 1$.

Our subsequent analysis is relatively insensitive to how we pick $A$, $B$, $s_0$, and $c_V$, provided they satisfy some additional restrictions.
In contrast, the error term we obtain is very sensitive to $\beta$. \Cref{thm:quadratic_murmurations} results from taking $\beta = 1/16$. If one takes $\beta = 0$ then one recovers a more common version of the approximate functional equation, e.g.\ \cite[Thm.\ 5.3]{IK}. Taking $\beta = 0$ leads to a power savings of $1/216$ in \cref{thm:quadratic_murmurations} instead of $1/168$, and the more restrictive $7/8 < \delta < 1$ instead of $5/6 < \delta < 1$. The parameter $\alpha$, which is similar in character to $\beta$, is used to prove \cref{thm:r1size,thm:r2size}; see \cref{sec:size_proof}.

\subsection{Mellin inversion}
\label{subsec:mellin}

Recall that our strategy is to evaluate \eqref{eq:avg_mellin} in two different ways: substituting for $L(s,\rep)$ using the approximate functional equation \eqref{eq:afe_specific}, and using Mellin inversion to obtain an expression involving Dirichlet coefficients. We lay the groundwork for the latter here. Recall the notation established in and surrounding \eqref{eq:rep_Lfunc_series} pertaining to the automorphic $L$-function $L(s,\rep)$.

\textit{Perron's formula} will be for us the approximation
\begin{align}
  \label{eq:perron}
  \sum_{n < \x} a_\rep(n) = \frac{1}{2\pi i} \int_{c_0 - iT}^{c_0 + iT} L(s,\rep) \mspace{2mu}\x^s \,\frac{ds}{s} + O\big(\x^{c_0}T^{-1}(\x T\qrep)^\eps\big)
  ,
\end{align}
where $\x  > 2$, $\x  \not\in \Z$, $T \ll \x^{1-\eps}$ --- these conditions can be loosened if necessary --- and $c_0 > 1 + \theta$. A more precise version of \eqref{eq:perron} is given in \cite[Cor.\ 5.3]{MV}.

Perron's formula follows from writing $L(s,\rep)$ as an absolutely convergent Dirichlet series (whence the necessity for $c_0 > 1 + \theta$), and then applying Mellin inversion for the function $1/s$ term-by-term. To establish \cref{thm:gamma_murmurations} we will follow this procedure using a function other than $1/s$. By replacing $L(s,\rep)$ with its Dirichlet series, one finds, for suitable meromorphic functions $\mf$,
\begin{align}
  \label{eq:approx_general_inverse_mellin}
  \frac{1}{2\pi i}\int_{c_0 - iT}^{c_0 + iT} L(s,\rep)\mspace{2mu}\x^s \mf(s)\mspace{1.5mu}ds \approx \sum_{n=1}^\infty a_\rep(n) \cM^{-1}\mspace{-2mu}\{\mf\}\big(\tfrac{n}{\x}\big)
  .
\end{align}
The discrepancy that ``$\approx$'' implies, which comes from the truncation of the integral to $|\Im(s)| < T$, depends on $\mf$. For example, in Perron's formula \eqref{eq:perron}, one has an error term of roughly $\x/T$. If $\mf(s) = \thalfGamma{s}$, then one may take the $T \to \infty$, obtaining the identity
\begin{align}
  \label{eq:inverse_mellin_halfGamma}
  \frac{1}{4\pi i}\int_{(c_0)} \halfGamma{s}L(s,\rep)\mspace{2mu}\x^s\,ds = \sum_{n=1}^\infty e^{\nicefrac{-n^2}{x^2}} a_\rep(n)
  .
\end{align}
In fact \eqref{eq:inverse_mellin_halfGamma} holds for any $c_0 > 0$ if $L(s,\rep)$ is entire.

\subsection{Constructing murmurations}
\label{sec:outline_construction}

Our strategy is still to evaluate the average inverse Mellin transform \eqref{eq:avg_mellin} in two different ways: by substituting the approximate functional equation \eqref{eq:afe_specific}, and by applying an inverse Mellin transform to the Dirichlet series as in \eqref{eq:approx_general_inverse_mellin} or Perron's formula \eqref{eq:perron}.

Let $\cF_\infty$ be a set of $L$-functions $L(s,\rep) = \sum_n a_\rep(n) n^{-s}$ with conductors $\qrep = q_\rep$. Let $\cS_\infty$ be a subset of $\Z$.
Pick increasing functions \hypertargetc{hyper:murdefs}{$\qmur$, $\Delta \qmur$, $\nmur$, $\Delta \nmur$} from $\R_{\geqslant 0}$ to $\R_{\geqslant 0}$, with $\qmur$ and $\nmur$ unbounded.
Define
\hypertarget{hyper:cSdef}{}
\begin{align*}
  \cF(\x) &\coloneqq \{\rep \in \cF_\infty \,:\, \qmur(\x) - \Delta\qmur(\x) < \qrep < \qmur(\x)\}
  \quad\text{and}\quad
  \cS(\x) \coloneqq \{n \in \cS_\infty \,:\, \nmur(\x) - \Delta\nmur(\x) < n < \nmur(\x)\}
  .
\end{align*}
If
\begin{align}
  \label{eq:murmurationsdef}
  \frac{1}{\#\cF(\x)}\sum_{\rep \in \cF(\x)} \frac{1}{\sqrt{\#\cS(\x)}} \sum_{n \in \cS(\x)} a_\rep(n) \,\sim\, M\!\left(\frac{\nmur(x)}{\qmur(x)}\right)
\end{align}
as $\x \to \infty$
for some nonzero function $M$, we call \eqref{eq:murmurationsdef} a \textit{murmuration}. If weights are present, e.g.\
\begin{align*}
  \frac{1}{\#\cF(\x)}\sum_{\rep \in \cF(\x)} \frac{1}{\sqrt{\#\cS(\x)}} \sum_{n \in \cS(\x)} a_\rep(n) \,\Phi\!\left(\frac{\qrep}{\qmur}, \frac{n}{\nmur}\right) \,\sim\, M_\Phi\!\left(\frac{\nmur}{\qmur}\right)
\end{align*}
for some reasonable function \hypertargetc{hyper:Phidef}{$\Phi$}, we'll call that a murmuration also. Statistics of the form \eqref{eq:murmurationsdef} emerge naturally from the theoretical approach of this paper, as well as our companion paper \cite{ratiosconjecture}'s. They capture fine structure in sums of Dirichlet coefficients.


We'll focus on cases where $\cF_\infty$ is a set of automorphic representations all having the same degree $\repdim$, root number $\rootnum_\rep = \rootnum_{\cF}$, and gamma factors $\prod_j \thalfGamma{s+\kappa_j}$. Take $\nmur(\x) = \Delta \nmur(\x) = \x$. Let $\theta$ be such that $a_\rep(n) \ll n^{\theta + \eps}$ and $\Re(\kappa_j) \geqslant -\theta$ all $\rep \in \cF_\infty$. Let's also fix $c \in (0,1)$, $\,c_0 > 1 + \theta$, and a function $T = T(\qmur)$ positive, increasing, and unbounded as $\qmur \to \infty$.

Let's return to our strategy of evaluating the average inverse Mellin transform \eqref{eq:avg_mellin} in two different ways.
The conjunction of Perron's formula \eqref{eq:perron} and the approximate functional equation \eqref{eq:afe_specific} (for $L(s,\rep) \neq \zeta(s + i\tau)$) gives
\begin{align}
  \label{eq:afe_perron_mellin_conjunction}
  \begin{aligned}
    \frac{1}{\sqrt{\x}} \sum_{n < \x} a_\rep(n)
    ={}
    & \frac{1}{2\pi i} \int_{c - iT}^{c + iT} \sum_{n=1}^\infty \frac{a_\rep(n)}{n^s} V_s\!\left(\frac{\pi^{\frac{\repdim}{2}}}{\qrep^\half\xi}n\right) \x^{s-\half} \,\frac{ds}{s}
    \\
    {}+{} 
    &\frac{1}{2\pi i} \int_{c - iT}^{c + iT} \prod_{j=1}^\repdim \frac{\Gamma\!\left(\frac{1-s+\bar\kappa_j}{2}\right)}{\Gamma\!\left(\frac{s+\kappa_j}{2}\right)}
    \sum_{n=1}^\infty \frac{\rootnum_\rep a_{\bar\rep}(n)}{n^{1-s}}
    V_{1-s}^*\!\left(\frac{\pi^{\frac{\repdim}{2}}\xi}{\qrep^\half}n\right)
    \left(\frac{\pi^\repdim \x}{\qrep}\right)^{\!s-\half} \frac{ds}{s}
    \\
    {}+{}
    & \frac{1}{2\pi i}\left(\int_{c_0 - iT}^{c - iT} + \int_{c + iT}^{c_0 + iT} \right) L(s,\rep) \mspace{2mu}\x^{s-\half} \,\frac{ds}{s}
    \;+\; O\big(\x^{c_0 - \half}T^{-1}(\x T \qrep)^\eps\big)
    .
  \end{aligned}
\end{align}

Averaging \eqref{eq:afe_perron_mellin_conjunction} over $\cF(\x) \eqqcolon \cF$ and writing $y \coloneqq \x/\qmur$ gives
\begin{align}
  \label{eq:afe_perron_mellin_conjunction_avg}
  \begin{aligned}
    \frac{1}{\#\cF} \sum_{\rep \in \cF} \frac{1}{\sqrt{\qmur y}}
    &\sum_{n < \qmur y} a_\rep(n)
    \\
    ={}
    & \frac{1}{2\pi i} \int_{c - iT}^{c + iT} \frac{1}{\#\cF} \sum_{\rep \in \cF} \sum_{n=1}^\infty \frac{a_\rep(n)}{n^s}
    V_{s}\!\left(\frac{\pi^{\frac{\repdim}{2}}}{\qrep^\half\xi}n\right)
    \x^{s-\half} \,\frac{ds}{s}
    \\
    {}+{} 
    &\frac{1}{2\pi i} \int_{c - iT}^{c + iT} \prod_{j=1}^\repdim \frac{\Gamma\!\left(\frac{1-s+\bar\kappa_j}{2}\right)}{\Gamma\!\left(\frac{s+\kappa_j}{2}\right)}
    \frac{1}{\#\cF} \sum_{\rep \in \cF}
    \sum_{n=1}^\infty \frac{\rootnum_{\rep} a_{\bar\rep}(n)}{n^{1-s}}
    V_{1-s}^*\!\left(\frac{\pi^{\frac{\repdim}{2}}\xi}{\qrep^\half}n\right)
    \!
    \left(\frac{\qmur}{\qrep} \pi^\repdim y\right)^{\!s-\half} \frac{ds}{s}
    \\
    {}+{}
    & \frac{1}{2\pi i}\left(\int_{c_0 - iT}^{c - iT} + \int_{c + iT}^{c_0 + iT} \right) \frac{1}{\#\cF} \sum_{\rep \in \cF} L(s,\rep) \mspace{2mu}\x^{s-\half} \,\frac{ds}{s}
    \;+\; O\big(\x^{c_0 - \half}T^{-1}(\x T \qmur)^\eps\big)
    .
  \end{aligned}
\end{align}
Observe that the second term on the right hand side of equation \eqref{eq:afe_perron_mellin_conjunction_avg},
\begin{align}
  \nonumber
  \frac{1}{2\pi i} \int_{c - iT}^{c + iT}
  &
  \prod_{j=1}^\repdim \frac{\Gamma\!\left(\frac{1-s+\bar\kappa_j}{2}\right)}{\Gamma\!\left(\frac{s+\kappa_j}{2}\right)}
  \frac{1}{\#\cF} \sum_{\rep \in \cF}
  \sum_{n=1}^\infty \frac{\rootnum_{\rep} a_{\bar\rep}(n)}{n^{1-s}}
  V_{1-s}^*\!\left(\frac{\pi^{\frac{\repdim}{2}}\xi}{\qrep^\half}n\right)
  \left(\frac{\qmur}{\qrep} \pi^\repdim y\right)^{\!s-\half} \frac{ds}{s}
  \\
  \label{eq:approx_t2}
  &\approx
  \frac{\rootnum_{\cF}}{2\pi i} \int_{c - iT}^{c + iT} \prod_{j=1}^\repdim \frac{\Gamma\!\left(\frac{1-s+\bar\kappa_j}{2}\right)}{\Gamma\!\left(\frac{s+\kappa_j}{2}\right)}
  \sum_{n=1}^\infty
  V_{1-s}^*\!\left(\frac{\pi^{\frac{\repdim}{2}}\xi}{q^\half}n\right)
  \frac{1}{\#\cF} \sum_{\rep \in \cF} \frac{a_{\bar\rep}(n)}{n^{1-s}}
  \left(\pi^\repdim y\right)^{\!s-\half}
  \frac{ds}{s}
\end{align}
is approximately an inverse Mellin transform in the variable $\pi^\repdim y$.
This pattern is encountered repeatedly in \cite{ratiosconjecture}, and is the source of the scale invariance seen in murmurations. (One may also elect to weight the Dirichlet coefficients $a_{\bar\rep}(n)$ by $\rootnum_{\rep}$, like in \eqref{eq:afe_perron_mellin_conjunction_avg}, removing the need to approximate $\rootnum_{\rep} \approx \rootnum_{\cF}$.)

We are led to the investigation of the averaged Dirichlet series
\begin{align}
  \label{eq:avg_dirichlet_series}
  \sum_{n=1}^\infty
  \frac{1}{\#\cF} \sum_{\rep \in \cF}
  \frac{a_{\bar\rep}(n)}{n^{1-s}}
\end{align}
appearing in \eqref{eq:approx_t2}. Analysis of this quantity is essentially what's done heuristically in \cite{cfkrs} or \cite{conrey_snaith} as part of the ratios conjecture recipe.
To observe scale invariance, the average \eqref{eq:avg_dirichlet_series} should be approximately independent of $\qmur$ when $\qmur$ is large. As evidenced by the widespread success the ratios conjecture recipe has enjoyed, this is very often the case for natural families $\cF$.

We also look to bound the contribution of the many other terms and approximations which arise. The right hand side of \eqref{eq:afe_perron_mellin_conjunction_avg} features, in addition to what will be our main term, three terms to be analyzed separately. Furthermore, the approximation \eqref{eq:approx_t2} will come with error terms, as will the estimation of \eqref{eq:avg_dirichlet_series} by a function independent of $\qmur$.

\begin{remark}[The correct normalization for murmurations]
  \label{rem:murmurations_sqrt}
  Each of \cite{zubrilina,LOP,bblld,wang}'s main theorems average Dirichlet coefficients $\lambda(p)$ which are normalized to be of size $\sqrt{p}$ under the Ramanujan--Petersson conjecture. It is a peculiarity of murmurations that this is the appropriate normalization. If one were to instead normalize the Dirichlet coefficients to be of size $p^{\half - \eps}$ the resulting average would be zero, and normalizing the Dirichlet coefficients to be of size $p^{\half + \eps}$ would cause the average to diverge. If one were to replace the Dirichlet coefficients in the main theorems of \cite{zubrilina,LOP,bblld,wang} by iid random variables with variance $\sqrt{p}^2$, the resulting sums would diverge almost surely \cite[Thm.\ 2.5.8]{durrett}.
  
  By partial summation, the normalization in \eqref{eq:murmurationsdef} is essentially the same as averaging Dirichlet coefficients of size $\sqrt{p}$ \cite[\S 2]{ratiosconjecture}. Asymptotics of this form are seen to emerge naturally from our construction above. The source of this normalization is the shift of $1/2$ appearing in the exponent of the term of the approximate functional equation in which the functional equation has been applied:
  \begin{align*}
      L(s,\rep)
      &
      ={} \sum_{n=1}^\infty \frac{a_\rep(n)}{n^{s}}
      V_s\!\left(\frac{\pi^{\frac{\repdim}{2}}}{\qrep^\half\xi}n\right)
      \;+\;
      \rootnum_\rep 
      \left(\frac{\pi^\repdim}{\qrep}\right)^{\!s - \textcolor{red!88!black}{\half}}
      \prod_{j=1}^\repdim \frac{\halfGamma{1-s+\bar\kappa_j}}{\halfGamma{s+\kappa_j}}
      \sum_{n=1}^\infty \frac{a_{\bar\rep}(n)}{n^{1-s}}
      V_{1-s}^*\!\left(\frac{\pi^{\frac{\repdim}{2}}\xi}{\qrep^\half}n\right)
      \!.
  \end{align*}
  Here we have used specifically the approximate functional equation \eqref{eq:afe_specific}, but the same shift of $1/2$ is there always. See \cref{lemma:afe_psi} for a quite general form of the approximate functional equation. \Cref{lemma:afe_psi}'s proof makes it clear that this shift of $1/2$ is inevitable.

  The significance of this $1/2$ is the following. Applying Perron's formula \eqref{eq:perron} or any other form of Mellin inversion \eqref{eq:approx_general_inverse_mellin} to obtain a sum over Dirichlet coefficients, an intermediate step is multiplying $L(s,\rep)$ by $\x^s$, where $\x$ the a newly introduced variable with respect to which the inverse Mellin transform is taken. We have observed in \eqref{eq:afe_perron_mellin_conjunction} and \eqref{eq:afe_perron_mellin_conjunction_avg} that the scale invariant parameter $\x/q$ emerges naturally during this multiplication:
  \begin{align*}
    L(s,\rep)
    &
    ={} \sum_{n=1}^\infty \frac{a_\rep(n)}{n^{s}}
    V_s\!\left(\frac{\pi^{\frac{\repdim}{2}}}{\qrep^\half\xi}n\right)
    \;+\;
    \rootnum_\rep 
    \prod_{j=1}^\repdim \frac{\halfGamma{1-s+\bar\kappa_j}}{\halfGamma{s+\kappa_j}}
    \sum_{n=1}^\infty \frac{a_{\bar\rep}(n)}{n^{1-s}}
    V_{1-s}^*\!\left(\frac{\pi^{\frac{\repdim}{2}}\xi}{\qrep^\half}n\right)
    \left(\frac{\pi^\repdim}{\qrep}\right)^{\!s - \textcolor{red!88!black}{\half}}
    \\
    \Longrightarrow\quad\quad
    L(s,\rep)\x^s
    &
    ={} \sum_{n=1}^\infty \frac{a_\rep(n)}{n^{s}}
    V_s\!\left(\frac{\pi^{\frac{\repdim}{2}}}{\qrep^\half\xi}n\right)\x^s
    \;+\;
    \rootnum_\rep 
    \prod_{j=1}^\repdim \frac{\halfGamma{1-s+\bar\kappa_j}}{\halfGamma{s+\kappa_j}}
    \sum_{n=1}^\infty \frac{a_{\bar\rep}(n)}{n^{1-s}}
    V_{1-s}^*\!\left(\frac{\pi^{\frac{\repdim}{2}}\xi}{\qrep^\half}n\right)
    \left(\frac{\pi^\repdim}{\qrep}\right)^{\!s - \half} \x^s
    \\
    &=
    \sum_{n=1}^\infty \frac{a_\rep(n)}{n^{s}}
    V_s\!\left(\frac{\pi^{\frac{\repdim}{2}}}{\qrep^\half\xi}n\right)\x^s
    \;+\;
    \rootnum_\rep 
    \prod_{j=1}^\repdim \frac{\halfGamma{1-s+\bar\kappa_j}}{\halfGamma{s+\kappa_j}}
    \sum_{n=1}^\infty \frac{a_{\bar\rep}(n)}{n^{1-s}}
    V_{1-s}^*\!\left(\frac{\pi^{\frac{\repdim}{2}}\xi}{\qrep^\half}n\right)
    \left(\frac{\pi^\repdim\x}{\qrep}\right)^{\!s - \half} \x^{\textcolor{red!88!black}{\half}}
    .
  \end{align*}
  One then divides through by $\x^\half$, yielding the peculiar normalization. 
  
  The source of this $1/2$ is the functional equation \eqref{eq:rep_Lfunc_functional_equation} itself:
  \begin{align*}
    \left(\frac{\qrep}{\pi^\repdim}\right)^{\!\frac{s}{2}}\prod_{j=1}^\repdim \halfGamma{s + \kappa_j} L(s,\rep)
    &=
    \rootnum_\rep \left(\frac{\qrep}{\pi^\repdim}\right)^{\!\frac{1-s}{2}}\prod_{j=1}^\repdim \halfGamma{1 - s + \bar\kappa_j} L(1-s,\bar\rep)
    \\
    \Longleftrightarrow\quad\quad
    \prod_{j=1}^\repdim \halfGamma{s + \kappa_j} L(s,\rep)
    &=
    \rootnum_\rep \left(\frac{\pi^\repdim}{\qrep}\right)^{s - \!\textcolor{red!88!black}{\half}}\prod_{j=1}^\repdim \halfGamma{1 - s + \bar\kappa_j} L(1-s,\bar\rep)
    .
  \end{align*}
  In other words, the fact that the correct normalization for observing murmurations takes Dirichlet coefficients to be of size $p^\half$ is a consequence of the fact that the critical strip is centred at $\sigma = 1/2$ in the analytic normalization.
\end{remark}

Let us now return to the approximation \eqref{eq:approx_t2} and disect it in more detail.
For a function \hypertargetc{hyper:gdef}{$g: \C \to \C$}, a real number \hypertargetc{hyper:cgdef}{$0 < c_g < 1$}, and constants \hypertargetc{hyper:K1K2def}{$K_1, K_2 \in \C$}, define
\begin{align*}
  \cR_1 &\coloneqq
  \frac{1}{2\pi i} \int_{c - iT}^{c + iT}
  \frac{1}{\#\cF} \sum_{\rep \in \cF} \sum_{n=1}^\infty \frac{a_\rep(n)}{n^s}
  V_{s}\!\left(\frac{\pi^{\frac{\repdim}{2}}}{\qrep^\half\xi}n\right)
  \x^{s-\half} \,\frac{ds}{s}
  \,-\, K_1
  \\
  \cR_2 &\coloneqq 
  \frac{\rootnum_{\cF}}{2\pi i} \int_{c - iT}^{c + iT}
  \prod_{j=1}^\repdim \frac{\Gamma\!\left(\frac{1-s+\bar\kappa_j}{2}\right)}{\Gamma\!\left(\frac{s+\kappa_j}{2}\right)}
  \Bigg[
  \frac{1}{\#\cF} \sum_{\rep \in \cF}
  \sum_{n=1}^\infty \frac{a_{\bar\rep}(n)}{n^{1-s}}
  V_{1-s}^*\!\left(\frac{\pi^{\frac{\repdim}{2}}\xi}{\qrep^\half}n\right)
  \left(\frac{\qmur}{\qrep} \right)^{\!s-\half}
  - g(s)
  \Bigg]
  (\pi^\repdim y)^{s-\half}\,
  \frac{ds}{s}
  \\
  \cR_H &\coloneqq \frac{1}{2\pi i}\left(\int_{c_0 - iT}^{c - iT} + \int_{c + iT}^{c_0 + iT} \right) \frac{1}{\#\cF} \sum_{\rep \in \cF} L(s,\rep) \mspace{2mu}\x^{s-\half} \,\frac{ds}{s}
  \\
  \cR_P &\coloneqq \x^{c_0 - \half}T^{-1}(\x T \qmur)^\eps
  \\
  \cR_g &\coloneqq
  \frac{\rootnum_{\cF}}{2\pi i} \left(\int_{c_g - i\infty}^{c_g + i\infty} - \int_{c - iT}^{c + iT}\right) \prod_{j=1}^\repdim \frac{\Gamma\!\left(\frac{1-s+\bar\kappa_j}{2}\right)}{\Gamma\!\left(\frac{s+\kappa_j}{2}\right)}
  g(s) (\pi^\repdim y)^{s-\half} \frac{ds}{s}
  -
  K_2
  .
\end{align*}

With these definitions we can write \eqref{eq:afe_perron_mellin_conjunction_avg} as
\begin{align}
  \label{eq:afe_murmurations_asymp}
  \frac{1}{\#\cF} \sum_{\rep \in \cF} \frac{1}{\sqrt{\qmur y}} \sum_{n < \qmur y} a_\rep(n)
  ={}
  &
  \frac{\rootnum_{\cF}}{2\pi i} \int_{c_g - i\infty}^{c_g + i\infty} \prod_{j=1}^\repdim \frac{\Gamma\!\left(\frac{1-s+\bar\kappa_j}{2}\right)}{\Gamma\!\left(\frac{s+\kappa_j}{2}\right)}
  g(s) (\pi^\repdim y)^{s-\half} \frac{ds}{s}
  +
  K_1 + K_2
  \\
  \nonumber
  &+ \cR_1 + \cR_2 + \cR_H + \cR_g + O(\cR_P)
  .
\end{align}

Suppose that, for all $y$ in some open set containing $1$, the integral on the right hand side of \eqref{eq:afe_murmurations_asymp} converges, and that $\cR_1$, $\cR_2$, $\cR_H$, $\cR_g$, $\cR_P$ all go to $0$ as $\qmur \to \infty$. Then we see that the right hand side asymptotically depends only on $y = \x/\qmur$, and so the left hand side is scale-invariant. We have obtained a murmuration in the sense of \eqref{eq:murmurationsdef}.

Variations are also possible. If one desires smoothing in the horizontal aspect, replacing the use of Perron's formula \eqref{eq:perron} with the inverse Mellin transform \eqref{eq:approx_general_inverse_mellin} of some function $\mf$ which decays suitably quickly leads to
\begin{align}
  \label{eq:afe_murmurations_asymp_smoothed}
  \frac{1}{\#\cF} \sum_{\rep \in \cF} \frac{1}{\sqrt{\qmur y}} \sum_{n < \qmur y} \cM^{-1}\mspace{-2mu}\{f\}\mspace{-1mu}\big(\tfrac{n}{x}\big) \mspace{2mu} a_\rep(n)
  &=
  \frac{\rootnum_{\cF}}{2\pi i} \int_{c_g - i\infty}^{c_g + i\infty} \prod_{j=1}^\repdim \frac{\Gamma\!\left(\frac{1-s+\bar\kappa_j}{2}\right)}{\Gamma\!\left(\frac{s+\kappa_j}{2}\right)}
  g(s) (\pi^\repdim y)^{s-\half} f(s)\mspace{1.5mu}ds
  +
  K_1^{(f)} + K_2^{(f)}
  \\
  \nonumber
  &+ \cR_1^{(f)} + \cR_2^{(f)} + \cR_H^{(f)} + \cR_g^{(f)}
  ,
\end{align}
for analogously defined quantities $*^{(f)}$. The weights present in these sorts of variations can improve error terms relative to \eqref{eq:afe_murmurations_asymp}, as is observed comparing \cref{thm:quadratic_murmurations} to \cref{thm:gamma_murmurations}. In general, the analysis in situations like \eqref{eq:afe_murmurations_asymp_smoothed} is substantially less delicate.

\Cref{sec:rep_quadratic_twists} onwards focuses on establishing \eqref{eq:afe_murmurations_asymp} for quadratic twist families on $\mathrm{GL}_1$. Our main result is \cref{thm:quadratic_murmurations}. The variation \eqref{eq:afe_murmurations_asymp_smoothed} for quadratic twists is \cref{thm:gamma_murmurations}. An intermediate step is to determine the average \eqref{eq:avg_dirichlet_series} of $L$-functions over the family, and this is recorded in \cref{thm:r1size,thm:r2size}, for both $\mathrm{GL}_1$ and $\mathrm{GL}_2$.

\subsection{Quadratic twist families of automorphic representations}
\label{sec:rep_quadratic_twists}

Let $\chi$ be a nontrivial even primitive Dirichlet character of conductor $q_\chi$ coprime to $\qrep$. The $L$-function of the Rankin--Selberg convolution $\rep\otimes\chi$ is given by \cite[Prop.\ 1.1]{zhang}
\begin{align}
  \label{eq:rankin_selberg_gl1}
  &\begin{aligned}
  L(s,\rep\otimes\chi)
  &= \sum_{n=1}^\infty \frac{a_\rep(n)\chi(n)}{n^s}
  \\
  \phantom{\Lambda(s,\rep\otimes\chi)}
  &= \prod_p \prod_{j=1}^\repdim \big(1 - \alpha_j(p)\chi(p)p^{-s}\big)^{-1}.
  \end{aligned}
  \shortintertext{
    The functional equation of $L(s,\rep\otimes\chi)$ is
  }
  \nonumber
  &\mspace{-3.8mu}\Lambda(s,\rep\otimes\chi)
  \coloneqq \left(\frac{\qrep q_\chi^\repdim}{\pi^\repdim}\right)^{\!\frac{s}{2}} \prod_{j=1}^\repdim \halfGamma{s + \kappa_j} L(s,\rep\otimes\chi)
  \\
  \nonumber
  &\phantom{\Lambda(s,\rep\otimes\chi)} = \rootnum_{\rep\otimes\chi} \Lambda(1-s, \bar\rep\otimes\bar\chi).
\end{align}
The root number is given by
\begin{align}
  \label{eq:root_number_def}
  \rootnum_{\rep\otimes\chi} \coloneqq \centralchar_\rep(q_\chi) \chi(\qrep) \rootnum_{\rep} \rootnum_\chi^\repdim
  ,
\end{align}
where
\hypertargetc{hyper:centralchardef}{$\centralchar_\rep$} is the central character of $\rep$,
and $\rootnum_\chi$ is the root number of $L(s,\chi)$ \cite[Props.\ 1.4, 1.5]{zhang}. 

The approximate functional equation \eqref{eq:afe_specific} for $L(s,\rep\otimes\chi)$ is
\begin{align}
  \begin{aligned}
  L(s,\rep\otimes\chi)
  &= \sum_{n=1}^\infty \frac{a_\rep(n)\chi(n)}{n^s} V_{s}\left(\frac{\pi^{\frac{\repdim}{2}}}{\qrep^\half q_\chi^{\frac{\repdim}{2}}\xi}n\right)
  \\
  &+ \rootnum_{\rep\otimes\chi} \left(\frac{\pi^\repdim}{\qrep q_\chi^\repdim}\right)^{\!s - \half}
  \prod_{j=1}^\repdim \frac{\halfGamma{1-s+\bar\kappa_j}}{\halfGamma{s+\kappa_j}}
  \sum_{n=1}^\infty \frac{a_{\bar\rep}(n)\overline{\chi}(n)}{n^{1-s}} V_{1-s}^*\left(\frac{\pi^{\frac{\repdim}{2}} \xi}{\qrep^\half q_\chi^{\frac{\repdim}{2}}}n\right)
  .
  \end{aligned}
  \label{eq:afe_rep}
\end{align}

Recall the definition \eqref{eq:Fdef} for $\cF$, and $\rootnum_{\cF}$ implicitly. Write
\hypertargetc{hyper:DeltaD}{$\Delta D \coloneqq D - D_0$}.
It follows from \eqref{eq:root_number_def}, quadratic reciprocity, and the fact \hypertargetc{hyper:rootnumchid}{$\rootnum_{\chi_d} = 1$} \cite[\S 5.9]{IK}, that, up to factors corresponding to the infinite place,
the value of $\rootnum_{\rep\otimes\chi_d}$ is determined by a congruence condition modulo $4\qrep$ if $2 \parallel \qrep$ and modulo $\qrep$ otherwise.
The construction of $\cF$ ensures
\begin{itemize}
\item
  $\Lambda(s,\rep\otimes\chi_d)$ is entire for every $d \in \cF$, 
\item
  the gamma factors of $L(s,\rep)$ and $L(s,\rep\otimes\chi_d)$ are all identical, and
\item
  the local factors at the finite places of the root numbers $\rootnum_{\rep\otimes\chi_d}$ are all identical. 
\end{itemize}

It is a general feature of murmurations that they are scaled by the root number of the family, and to be observed, root numbers in the family mustn't average to zero. It is for this reason that control over $d \mod \qrep$ is imperative for our application.

Perron's formula \eqref{eq:perron} applied to the average inverse Mellin transform \eqref{eq:avg_mellin} gives
\begin{align}
  \label{eq:quadratic_perron}
  \frac{1}{2\pi i}
  \int_{c - iT}^{c + iT} \Favg L(s,\rep \otimes \chi_d)\mspace{2mu}\x^{s-\half}\,\frac{ds}{s}
  &=
  \Favg \frac{1}{\sqrt{\x}} \sum_{n<\x} a_\rep(n)\chi_d(n)
  \\
  \nonumber
  &+ \frac{1}{2\pi i}\left(\int_{c - iT}^{1 + \theta + \eps - iT} + \int_{1 + \theta + \eps + iT}^{c + iT} \right) \frac{1}{\#\cF} \sum_{\rep \in \cF} L(s,\rep) \,\x^{s-\half} \,\frac{ds}{s}
  \\
  \nonumber
  &+ O\big(\x^{\half + \theta}T^{-1}(\x T D)^\eps\big)
  .
\end{align}

Substituting the approximate functional equation \eqref{eq:afe_rep} into the average inverse Mellin transform \eqref{eq:avg_mellin} for the family $\cF$ of quadratic twists 
gives
\begin{align}
  \nonumber
  \frac{1}{2\pi i}
  &
  \int_{c - iT}^{c + iT} \Favg L(s,\rep \otimes \chi_d)\mspace{2mu}\x^{s-\half}\,\frac{ds}{s}
  \\
  \label{eq:afe_r1_unweighted_t1ns}
  &=
  \frac{1}{2\pi i} \int_{c - iT}^{c + iT} \Favg \sum_{\substack{n=1 \\ n \neq \square}}^\infty \frac{a_\rep(n)\chi_d(n)}{n^s}V_s\!\left(\frac{\pi^{\frac{\repdim}{2}}n}{\xi\sqrt{\qrep d^\repdim}}\right)\x^{s-\half}\,\frac{ds}{s}
  \\
  \label{eq:afe_r1_unweighted_t2ns}
  &+
  \frac{1}{2\pi i} \int_{c - iT}^{c + iT} \prod_{j=1}^\repdim \frac{\halfGamma{1-s+\bar\kappa_j}}{\halfGamma{s+\kappa_j}} \Favg \sum_{\substack{n=1 \\ n \neq \square}}^\infty \frac{\rootnum_{\rep\otimes\chi_d} a_{\bar\rep}(n)\chi_d(n)}{n^{1-s}}V_{1-s}^*\!\left(\frac{\pi^{\frac{\repdim}{2}}\xi n}{\sqrt{\qrep d^\repdim}}\right)\xd^{\!s-\half}\frac{ds}{s}
  \\
  \label{eq:afe_r1_unweighted_t1s}
  &+
  \frac{1}{2\pi i} \int_{c - iT}^{c + iT} \Favg \sum_{n=1}^\infty \frac{a_\rep(n^2)\chi_d(n^2)}{n^{2s}}V_s\!\left(\frac{\pi^{\frac{\repdim}{2}} n^2}{\xi \sqrt{\qrep d^\repdim}}\right)\x^{s-\half}\,\frac{ds}{s}
  \\
  \label{eq:afe_r1_unweighted_t2s}
  &+
  \frac{1}{2\pi i} \int_{c - iT}^{c + iT} \prod_{j=1}^\repdim \frac{\halfGamma{1-s+\bar\kappa_j}}{\halfGamma{s+\kappa_j}} \Favg \sum_{n=1}^\infty \frac{\rootnum_{\rep\otimes\chi_d} a_{\bar\rep}(n^2)\chi_d(n^2)}{n^{2-2s}}V_{1-s}^*\!\left(\frac{\pi^{\frac{\repdim}{2}} \xi n^2}{\sqrt{\qrep d^\repdim}}\right)\xd^{\!s-\half}\frac{ds}{s}
  .
\end{align}
We have split the Dirichlet series into sums over squares and nonsquares. The latter we will analyze using \cref{lemma:stankus_2nd_moment}, which is an extension of a result of Jutila and Stankus \cite{jutila,stankus}.

The integrands of the four terms \eqref{eq:afe_r1_unweighted_t1ns}, \eqref{eq:afe_r1_unweighted_t2ns}, \eqref{eq:afe_r1_unweighted_t1s}, and \eqref{eq:afe_r1_unweighted_t2s} are bounded in \cref{lemma:t1ns_size,lemma:t2ns_size,lemma:t1s_size,lemma:t2s_approx_step1,lemma:t2s_approx_step2,lemma:t2s_approx_step3,lemma:t2s_approx_step4}. The sole main term comes from \eqref{eq:afe_r1_unweighted_t2s}, and is given in \cref{lemma:t2s_approx_step4}.

\section{Groundwork} 
\label{sec:lemmas}

\subsection{The approximate functional equation}
\label{sec:lemmas_afe}

Here and throughout we will make good use of \textit{Stirling's formula}, often without comment. There are multiple versions of this approximation to the $\Gamma$ function.
Gradshteyn and Ryzhik \cite[8.344]{GR} give, for $|s| \to \infty$ with $-\pi < \arg(s) < \pi$ and bounded away from $\pm\pi$,
\begin{align}
  \label{eq:stirling_log}
  \log \Gamma(s) = (s - \thalf)\log s - s + \log\sqrt{2\pi} + O\!\left(\tfrac{1}{s}\right)
  .
\end{align}
The same reference gives more precise expansions too if necessary. It follows from \eqref{eq:stirling_log} that
\begin{align}
  \label{eq:stirling_arg}
  \Gamma(s) = \sqrt{2\pi} \,s^{s - \half} e^{-s} \left(1 + O\!\left(\tfrac{1}{s}\right)\right),
\end{align}
and, writing $s = \sigma + it$,
\begin{align}
  \label{eq:stirling}
  |\Gamma(s)| = \sqrt{2\pi} \,|s|^{\sigma - \half} e^{-\sigma} e^{-\arg(s) t} \left(1 + O\!\left(\tfrac{1}{s}\right)\right)
  .
\end{align}
In many contexts $\sigma$ is confined to a vertical strip and may be regarded as absolutely bounded. In such cases, where one is interested only in the $t$-aspect of \eqref{eq:stirling}, the following slightly simpler version of Stirling's formula is handy: 
\begin{align}
  \label{eq:stirling_t}
  |\Gamma(\sigma + it)| = \sqrt{2\pi} \,|t|^{\sigma - \half} e^{-\frac{\pi}{2}|t|} \left(1 + O\!\left(\tfrac{1}{t}\right)\right)
  .
\end{align}
The error term $O\!\left(\tfrac{1}{t}\right)$ above should be understood to be $\ll 1$ when $\sigma > 0$ is bounded away from $0$. (Recall that the above approximations are valid for $\arg(s)$ bounded away from $\pm \pi$, so the poles of $\Gamma$ at the negative integers don't enter into this discussion.)

\begin{remark}
  \label{rem:stirling}
  We will often deal with products and ratios of gamma functions. Using only \eqref{eq:stirling}, one deduces e.g.\
  \begin{align}
    \nonumber
    |\Gamma(s_1)\Gamma(s_2)|
    =
    2\pi |s_1|^{\sigma_1 - \half} |s_2|^{\sigma_2 - \half} e^{-\sigma_1 - \sigma_2} e^{-\arg(s_1) t_1 - \arg(s_2)t_2}
    \left(1 + O\!\left(\frac{\Gamma(s_2)}{s_1} + \frac{\Gamma(s_1)}{s_2}\right) \right)
    .
  \end{align}
  However, returning to \eqref{eq:stirling_log}, we obtain the same approximation with a much better error term:
  \begin{align}
    \nonumber
    |\Gamma(s_1)\Gamma(s_2)|
    =
    2\pi |s_1|^{\sigma_1 - \half} |s_2|^{\sigma_2 - \half} e^{-\sigma_1 - \sigma_2} e^{-\arg(s_1) t_1 - \arg(s_2)t_2}
    \left(1 + O\!\left(\frac{1}{s_1} + \frac{1}{s_2}\right) \right)
    .
  \end{align}
\end{remark}

Having introduced Stirling's formula, we now move on to discussing the approximate functional equation.
Define
\begin{align}
  \label{eq:Gdef}
  \begin{aligned}
  \Gp(s) &\coloneqq \prod_{j=1}^\repdim \halfGamma{s + \kappa_j}
  \\
  \Gm(s) &\coloneqq \prod_{j=1}^\repdim \halfGamma{s + \bar\kappa_j}
  \end{aligned}
  ,
\end{align}
the gamma factors for $\rep$ and $\bar\rep$ respectively.

\begin{conditions}
  \label{cond:afe_general}
Let $\psiV(s,w)$ be meromorphic in $s$ for $0 \leqslant \Re(s) \leqslant 1$.
Choose $c_V \in \R_{>0}$ such that
\begin{itemize} 
\item
  $L(s+w,\rep)$ and $L(1-s+w,\bar\rep)$ converge absolutely whenever $\Re(w) = c_V$,
\item
  $\Lambda(s+w,\rep)$ is holomorphic on the line $\Re(w) = -c_V$, and
\item
  $\psiV(s,w)$ is meromorphic whenever $|\Re(w)| < c_V$.
\item
  $\psiV(s,w)$ is holomorphic on the lines $\Re(w) = \pm c_V$ and at $w = 0$.
\end{itemize}
\end{conditions}

Define
\begin{align}
  \label{eq:wtVdef}
  \begin{aligned}
    \V_+(s,y) &\coloneqq \frac{1}{2\pi i} \int_{(c_V\!)} \frac{\Gp(s+w)}{\Gp(s)} \psiV(s, w)\,y^{-w}\frac{dw}{w}
    \\
    \V_-(s,y) &\coloneqq \frac{1}{2\pi i} \int_{(c_V\!)} \frac{\Gm(s+w)}{\Gm(s)} \psiV(1-s,-w)\,y^{-w}\frac{dw}{w}
    .
  \end{aligned}
\end{align}

\begin{lemma}[{Approximate functional equation}]
  \label{lemma:afe_psi}
  Subject to \cref{cond:afe_general},
  \begin{align*}
    \Gp(s)
    L(s,\rep) \psiV(s,0)
    ={}&
    \Gp(s)
    \sum_{n=1}^\infty \frac{a_\rep(n)}{n^{s}} \V_+\!\!\left(s, \frac{\pi^{\frac{\repdim}{2}}n}{\qrep^\half}\right)
    \;+\;
    \rootnum_\rep\left(\frac{\pi^\repdim}{\qrep}\right)^{\!s - \half}
    \Gm(1-s)
    \sum_{n=1}^\infty \frac{a_{\bar\rep}(n)}{n^{1-s}}
    \V_-\!\!\left(1-s, \frac{\pi^{\frac{\repdim}{2}}n}{\qrep^\half}\right)
    \\
    &
    - \sum_{\substack{|\Re(w')| < c_V \\ w' \neq 0}} \underset{w'=w}{\mathrm{Res}} \,\Lambda(s+w,\rep) \frac{\psiV(s,w)}{w}
    .
  \end{align*}
\end{lemma}
The \hyperlink{proof:afe_psi}{proof of \cref*{lemma:afe_psi}} is the first of many we will postpone to \cref{sec:lemma_proofs}.
See the glossary \cref{sec:glossary} for the notation in \cref{lemma:afe_psi}.

We now describe the choice of $\V$ we make for our applications.
\begin{conditions}
\label{cond:afe_specific}
Let $A, B, \alpha, \beta, s_0$ be complex numbers. 
Choose parameters such that the following conditions are satisfied:
\begin{itemize}
\item
  $A \in \R_{>0}$,
\item
  $B \in \Z_{>0}$,
\item
  $\Re(s_0) < 0$.
\item
  $|\beta| < 1$,
\item
  $0 < \sigma < 1$
\item
  $\theta < c_V < \frac{\pi}{3}A$
\item
  for all $\sigma_w$ such that $\frac{\pi}{3}A > \sigma_w > -\min\!\left\{\frac{\pi}{3}A,\,\sigma + \Re(\kappa_1),\dots,\sigma + \Re(\kappa_\repdim)\right\}$,
  $$\frac{B}{A} > \pi|\beta| + \frac{\pi\repdim}{2} + \half\sumr \left(\left|\frac{\sigma_w - 1}{\sigma + \Re(\kappa_j)}\right| + 1\right),$$
\end{itemize}
\end{conditions}
When $\repdim = 1$, one may take $A = 2$, $B = 16$, $s_0 = -1$, and $c_V = 2$.

Let
\begin{align}
  \label{eq:xidef}
  \xi \coloneqq \xi(s) &\coloneqq D^{2\alpha} \exp\!\left(\beta \log\!\left((s - s_0)^2\right)\right)
  ,
\end{align}
where the branch cut of the logarithm is along the negative real axis and $-\pi < \arg(z) \leqslant \pi$.

We take $\psiV$ in \eqref{eq:wtVdef} to be the function
\begin{align}
  \label{eq:psidef}
  \psiV(s,w) \gets \psixi(s,w) \coloneqq \xi^w \left(e^{i\frac{w}{A}} + e^{-i\frac{w}{A}} - 1\right)^{-B}
  .
\end{align}
(In \eqref{eq:psidef}, $\gets$ denotes variable assignment, like in pseudocode. One could write ``$=$'' or ``$\coloneqq$'' instead, as long as one understands that the quantity $\psiV$ in \eqref{eq:wtVdef} etc.\ pertains to any function $\psiV$ satisfying the conditions listed above \eqref{eq:wtVdef}, not just $\psixi$.)

For parameters conforming to \cref{cond:afe_specific}, define
\begin{align}
  \label{eq:Vdef}
  \begin{aligned}
  V_s(y) &\coloneqq \frac{1}{2\pi i} \int_{(c_V\!)} \frac{\Gp(s+w)}{\Gp(s)} \left(e^{i\frac{w}{A}} + e^{-i\frac{w}{A}} - 1\right)^{-B} y^{-w}\frac{dw}{w}
  \\
  V_s^*\mspace{-2mu}(y) &\coloneqq \frac{1}{2\pi i} \int_{(c_V\!)} \frac{\Gm(s+w)}{\Gm(s)} \left(e^{i\frac{w}{A}} + e^{-i\frac{w}{A}} - 1\right)^{-B} y^{-w}\frac{dw}{w}
  \end{aligned}
  ,
\end{align}
so that, when $\psiV$ in \eqref{eq:wtVdef} is taken to be $\psixi$ from \eqref{eq:psidef}, we have
\begin{align*}
  V_s\big(\xi^{-1}y\big) &= \V_+(s, y)
  \shortintertext{and}
  V_{1-s}^*\big(\xi y\big) &= \V_-(1-s, y)
  .
\end{align*}

As we will be taking an inverse Mellin transform, it is imperative that, after applying the approximate functional equation to our integrand, it remains holomorphic. This is what the next lemma ensures.

\begin{lemma}
  \label{lemma:V_is_holomorphic}
  $V_s\big(\xi^{-1}y\big)$ and $V_{1-s}^*\big(\xi y\big)$ are holomorphic in the strip $0 < \sigma < 1$.
\end{lemma}
\begin{proof}
  We only prove the result for $V_s$. The situation for $V_{1-s}^*$ is similar.
  First we look at the factor
  \begin{align*}
    \frac{\Gp(s+w)}{\Gp(s)} = \prodr \frac{\halfGamma{s+w+\kappa_j}}{\halfGamma{s+\kappa_j}}
  \end{align*}
  in the integrand of $V_s(\xi^{\pm 1}y)$.
  As $\Gamma$ is never $0$, the denominator does not affect holomorphicity (provided one removes the removable singularities as the poles of the denominator, which we take to be understood). If $\Re(w) = c_V > \theta$, which is assumed in \cref{cond:afe_specific}, then $\Re(s + w + \kappa_j) > 0$, avoiding the poles of the numerator.

  Next we look at the only other factor in $V_s(\xi^{\pm 1}y)$'s integrand which depends on $s$,
  \begin{align*}
    \xi(s)^{\mp w} = D^{\mp 2\alpha w} \exp\!\left(\mp w\beta \log\big((s-s_0)^2\big)\right)
    .
  \end{align*}
  \Cref{cond:afe_specific} require that $\Re(s_0) < 0$, and we are taking $\log$ to have its branch cut along the negative real axis, so $\log\!\big((s-s_0)^2\big)$ is holomorphic in the half-plane $\Re(s) > 0$. Holomorphicity then follows from the convergence of the integral, which \cref{lemma:V_approx} establishes.
\end{proof}

Recall the definition \eqref{eq:Qdef} for the ``reduced analytic conductor'':
\begin{align}
  \nonumber
  Q \coloneqq Q(s) &\coloneqq \frac{\qrep}{(2\pi e)^\repdim} \prod_{j=1}^\repdim |s + \kappa_j|
  .
\end{align}

\Cref{lemma:V_approx} below is our central result around the approximate functional equation. 

\begin{lemma}
  \label{lemma:V_approx}
  For all $\sigma_w$ such that $\tfrac{\pi}{3}A > \sigma_w > -\min\!\left\{\frac{\pi}{3}A,\, \sigma + \Re(\kappa_1),\dots,\sigma + \Re(\kappa_\repdim)\right\}$,
  \begin{align*}
    V_s\!\left(\frac{\pi^{\frac{\repdim}{2}}}{\qrep^\half \xi}n\right)
    &= \onelr{n < Q^\half \xi} + O\!\left(\left(\frac{Q^\half \xi}{n}\right)^{\!\sigma_w} \cR_\sigma\right)
    \\
    \shortintertext{and}
    V_{1-s}^*\!\left(\frac{\pi^{\frac{\repdim}{2}}\xi}{\qrep^\half}n\right)
    &= \onelr{n < Q^\half \xi^{-1}} + O\!\left(\left(\frac{Q^\half \xi^{-1}}{n}\right)^{\!\sigma_w} \cR_{1-\sigma}\right)
    ,
  \end{align*}
  where
  \begin{align*}
    \cR_\sigma &\coloneqq
    \exp\!\left(\sumr \sigma + |\Re(\kappa_j)| + \frac{|\sigma_w|}{2} \left(\left|\frac{\sigma_w - 1}{\sigma + \Re(\kappa_j)}\right| + 3\right)\right)
    \frac{\repdim + B}{|\sigma_w|}
  \end{align*}
  is a constant depending only on $\sigma$, $\sigma_w$, $\Re(\kappa_1), \dots, \Re(\kappa_\repdim)$, $\repdim$, and $B$.
\end{lemma}
\hyperlink{proof:V_approx}{Click here to jump to the proof of \cref*{lemma:V_approx}.}

The following lemma is needed for taking linear approximations, which is the way we'll make the approximation $d \approx D$ for all $d \in \cF$.

\begin{lemma}
  \label{lemma:V_deriv_approx}
  Let $D > d > D - \Delta D > 0$. Suppose
  \begin{align*}
    \frac{B}{A} - \left(\pi|\beta| + \frac{\pi\repdim}{2} + \half\sumr \left(\left|\frac{\sigma_w - 1}{\sigma + \sigma_j}\right| + 1\right)\right)
    \;>\; \frac{r\Delta D}{D}
    .
  \end{align*}
  Then, For all $\sigma_w$ such that $\min\!\left\{\tfrac{\pi}{3}A,\,\frac{D}{r\Delta D}\right\} > \sigma_w > -\min\!\left\{\frac{\pi}{3}A,\,\frac{D}{r\Delta D},\, \sigma + \Re(\kappa_1),\dots,\sigma + \Re(\kappa_\repdim)\right\}$,
  \begin{align*}
    V_s^\circ\!\left(\frac{\xi^{\pm 1}y}{D^{\frac{\repdim}{2}}}\right) - V_s^\circ\!\left(\frac{\xi^{\pm 1}y}{d^{\frac{\repdim}{2}}}\right)
    \ll
    \left(\frac{Q^\half \xi^{\pm 1}}{n}\right)^{\!\sigma_w}
    \cR_\sigma \frac{\repdim |\sigma_w| \Delta D}{D}
    ,
  \end{align*}
  where $\cR_\sigma$ is as in \cref{lemma:V_approx}. Here $V_s^\circ$ means either $V_s$ or $V_s^*$.
\end{lemma}
\hyperlink{proof:V_deriv_approx}{Click here to jump to the proof of \cref*{lemma:V_deriv_approx}.}

\subsection{Sums over fundamental discriminants}
\label{sec:lemmas_discs}

Let $\phi$ denote Euler's totient function,
let $\divcount(m) \coloneqq \sum_{k \mid m} 1$ denote the divisor counting function,
and let $\starsum*{}$ denote sums which run over fundamental discriminants.
For any $k \in \Z_{>0}$, let $\chi_k^0$ denote the trivial character mod $k$.

\Cref{lemma:sqfree_character_estimate} is a helper lemma we go on to use in the proof of \cref{lemma:disc_power_sum}.

\begin{lemma}
  \label{lemma:sqfree_character_estimate}
  For any $\ff, m \in \Z_{>0}$, any Dirichlet character $\chi \mod \ff$, and any $z \in \C$ with $\Re(z) > -\thalf$,
  \begin{align*}
    \sum_{n < \x} n^z \mu(n)^2 \chi(n)\chi_m^0(n)
    &=
    \one{\text{\emph{$\chi$ is trivial}}}\frac{\x^{z+1}}{z+1}\frac{6}{\pi^2}\prod_{p \mid m\ff}\frac{p}{p+1}
    \\
    &\hspace{2cm}
    + O\!\left(\ff^{\frac{1}{6}}\x^{\Re(z) + \half}\!\left(\x^{\frac{1}{14}} + |\Im(z)|^{\frac{1}{6}}\right) \big(m\ff\x|\Im(z)|\big)^\eps\right)
    .
  \end{align*}
\end{lemma}
\hyperlink{proof:sqfree_character_estimate}{Click here to jump to the proof of \cref*{lemma:sqfree_character_estimate}.} The idea of the proof is to approximate the Dirichlet series
\begin{align*}
  \sum_{n < \x} n^z \mu(n)^2 \chi(n)\chi_m^0(n) n^{-s} \approx \frac{L(s-z,\chi)}{L(2s-2z,\chi^2)}
\end{align*}
and then invoke Perron's formula \eqref{eq:perron}. The subconvexity bound \cite{petrow_young:2023} is used.

\Cref{lemma:disc_power_sum} is both a generalization and a sharpening of \cite[Lemma 1]{jutila} and \cite[Lemma 1]{stankus}. The generalizations are that the result now applies to even positive integers $q$ in addition to odd ones, and that the admissible range $\Re(z) > -\thalf$ has been extended from $\Re(z) \geqslant 0$. This extension is necessary for us to obtain a result; without it, we would have been too constrained while balancing of error terms in \cref{sec:contrib_unweighted} to avoid the error term of \cref{thm:quadratic_murmurations} dwarfing the main term.

The sharpening is in the presence of $q^{\frac{1}{6}} (D^{\frac{1}{14}} + |\Im(z)|^{\frac{1}{6}})$ in the $\min$ of the error term. Without this sharpening, the factor of $|z|+1$ in the error term would have led to growth in the $t$ aspect too large to allow for an asymptotic result in \cref{thm:quadratic_murmurations}, i.e.\ we would not have been able to observe murmurations with a sharp cutoff.

Define
\begin{align}
  \label{eq:etadef}
  \eta_{m,\ff,\ell} \coloneqq
  \thalf\one{4 \nmid \ff} + \one{\text{$4\mid \ff$ and $\ell = 1\mod 4$}} + \tfrac{1}{4}\one{(m\ff,2) = 1}
  .
\end{align}

\begin{proposition}
  \label{lemma:disc_power_sum}
  For $(\ell,q) = 1$ and $\Re(z) > -\thalf$,
  \begin{align*}
    \starsum*_{\substack{0 < d < D \\ d = \ell \mod q \\ (d,m) = 1}} d^z
    &=
    \eta_{m,q,\ell}\,\frac{D^{z+1}}{z+1}\frac{6}{\pi^2\phi(q)} \prod_{p \mid 4qm} \frac{p}{p+1}
    \\
    &
    +
    O\!\left(D^{\Re(z) + \half}
    \min\!\left\{
    q^{\frac{1}{6}} \big(D^{\frac{1}{14}} + |\Im(z)|^{\frac{1}{6}}\big)
    ,\,
    q^{\frac{1}{2}} (|z|+1)
    \right\}
    \big(qm D|\Im(z)|\big)^\eps
    \right)
    .
  \end{align*}
\end{proposition}
\hyperlink{proof:disc_power_sum}{Click here to jump to the proof of \cref*{lemma:disc_power_sum}.} This proof is in two steps, one for each error term. The two steps are quite dissimilar. The first leverages \cref{lemma:sqfree_character_estimate} above, while the second proceeds with more direct and elementary manipulations of Dirichlet characters and is similar to \cite{stankus}.

Define
\begin{align}
  \label{eq:Etadef}
  \Eta_{\cF} \coloneqq 
  \eta_{\qrep,\qrep,\ell}
  ,
\end{align}
where $\eta_{m,\ff,\ell}$ is as defined in \eqref{eq:etadef}.

\begin{lemma}
  \label{lemma:Fsize}
  \begin{align*}
    \#\cF
    = \Delta D \,\frac{6\Eta_{\cF}}{\pi^2\phi(\qrep)} \prod_{p\mid 2\qrep} \frac{p}{p+1}
    + O\!\left(D^\half \qet (\qrep D)^\eps \right)
    .
  \end{align*}
\end{lemma}
\begin{proof}
  Follows from \cref{lemma:disc_power_sum} with $z = 0$.
\end{proof}

\Cref{lemma:stankus_2nd_moment} is a generalization of bounds of Jutila \cite{jutila} and Stankus \cite{stankus}, which sum only over $n$ odd. We will use it to bound the terms \eqref{eq:afe_r1_unweighted_t1ns} and \eqref{eq:afe_r1_unweighted_t2ns}, the sums over nonsquare indices $n$ in the Dirichlet series of the approximate functional equation averaged over our family of quadratic twists. The bounding quantity on the right hand side of \cref{lemma:stankus_2nd_moment} has acquired an additional factor of $(\log D)^2$ relative to \cite[Lemma 5]{stankus} due to the presence of even integers in the $n$ sum.

Eliminating the restriction to $n$ odd in the hypothesis of \cref{lemma:stankus_2nd_moment} is what allows us to consider families $\cF$ which contain fundamental discriminants $d = 1 \mod 4$. These make up two thirds of all fundamental discriminants (the natural densities of positive fundamental discriminants $d = 1\mod 4$, $d = 8 \mod 16$, and $d = 12 \mod 16$ are respectively $2/\pi^2$, $1/2\pi^2$, and $1/2\pi^2$, per \cref{lemma:disc_power_sum}), so it's nice that we're able to include them. Moreover, it's thanks to this generalization, in tandem with the related one in \cref{lemma:disc_power_sum}, that we're able to meaningfully consider twists of automorphic representations $\rep$ with conductor $\qrep$ even.

\begin{lemma}
  \label{lemma:stankus_2nd_moment}
  Let $\ff,\ell \in \Z_{>0}$ and $N, D \in \R_{>3}$ be such that
  $\ell$ is coprime to $\ff$ and $N,D \not\in \Z$.
  \begin{align*}
    \sum_{\substack{0 < n < N \\ n \neq \square \\ (n,\ff) = 1}} \,\left|\, \starsum*_{\substack{0 < d < D \\ d = \ell \mod \ff}} \chi_d(n) \,\right|^2
    \ll
    \ff ND(\log D)^4\log N
    + \ff^{\frac{1}{3} + \eps} N^\half D \min\!\big\{D^{\frac{1}{7}}, \ff^{\frac{2}{3}}\big\}
    .
  \end{align*}
\end{lemma}
\hyperlink{proof:stankus_2nd_moment}{Click here to jump to the proof of \cref*{lemma:stankus_2nd_moment}.} The main ingredient is \cite[\foreignlanguage{russian}{Теорема} 1]{FS} of Fainleib and Saparniyazov, transcribed as \cref{thm:FS}, which was also used by Stankus to prove \cite[Lemma 5]{stankus}.


\subsection{Euler products}
\label{sec:lemmas_euler_prod}

\Cref{lemma:euler_prod_rep_square} processes the Dirichlet series which ultimately becomes the main constituent of the right hand sides of \cref{thm:r1size,thm:quadratic_murmurations,thm:gamma_murmurations}. 

Recall the definition \eqref{eq:repL2def}:
\begin{align}
  \nonumber
  L(s,\repsq)
  &\coloneqq \sum_{n = 1}^\infty \frac{a_{\rep}(n^2)}{n^{s}}
  \\
  \nonumber
  &= \prod_p \prod_{j=1}^\repdim \left(1 - \alpha_j(p)^2 p^{-s}\right)^{-1}
  .
\end{align}

\begin{lemma}
  \label{lemma:euler_prod_rep_square}
  For $\repdim = 1$ and $\Re(s) > \thalf$,
  \begin{align*}
    \sum_{n=1}^\infty \frac{a_\rep(n^2)}{n^{2s}} \prod_{p \mid \frac{n}{(n,2\qrep)}} \frac{p}{p + 1}
    &=
    \frac{L(2s, \repsq)}{L^{(2\qrep)}(2s+1, \repsq)}
    \mspace{1mu}
    \prod_{p \nmid 2\qrep} \left(1 - \frac{1}{(p+1)(1 - \alpha_1(p)^{-2}p^{2s+1})}\right)
    .
  \end{align*}
\end{lemma}
\hypertarget{proof:euler_prod_rep_square}{}
\begin{proof}[Proof of \cref{lemma:euler_prod_rep_square}]
  When $\Re(s) > \thalf$ the left hand side of \cref{lemma:euler_prod_rep_square} converges absolutely. \Cref{lemma:prod_convergence} states that the right hand side converges as well. \Cref{lemma:euler_prod_rep_square} then follows from a calculation. First we identify the factor of $L(2s,\repsq)$:
  \begin{align*}
    \sum_{n=1}^\infty \frac{a_\rep(n^2)}{n^{2s}} \prod_{p \mid \frac{n}{(n,2\qrep)}} \frac{p}{p + 1}
    ={}
    &\prod_{p \nmid 2\qrep} 1 + \frac{p}{p+1}\left(\frac{\alpha_1(p)^2}{p^{2s}} + \frac{\alpha_1(p)^4}{p^{4s}} + \dots\right)
    \\
    {}\cdot{}
    &\prod_{p \mid 2\qrep} \left(1 + \frac{\alpha_1(p)^2}{p^{2s}} + \frac{\alpha_1(p)^4}{p^{4s}} + \dots\right)
    \\
    ={}
    &\prod_{p \nmid 2\qrep} \left(1 - \frac{\alpha_1(p)^2}{(p+1)p^{2s}}\right)\left(1 + \frac{\alpha_1(p)^2}{p^{2s}} + \frac{\alpha_1(p)^4}{p^{4s}} + \dots\right)
    \\
    {}\cdot{}
    &\prod_{p \mid 2\qrep} \left(1 + \frac{\alpha_1(p)^2}{p^{2s}} + \frac{\alpha_1(p)^4}{p^{4s}} + \dots\right)
    \\
    ={}
    &
    \prod_{p} \left(1 + \frac{\alpha_1(p)^2}{p^{2s}} + \frac{\alpha_1(p)^4}{p^{4s}} + \dots\right)
    \cdot 
    \prod_{p \nmid 2\qrep} \left(1 - \frac{\alpha_1(p)^2}{(p+1)p^{2s}}\right)
    \\
    ={}
    &
    L(2s,\repsq)
    \prod_{p \nmid 2\qrep} \left(1 - \frac{\alpha_1(p)^2}{(p+1)p^{2s}}\right)
    .
    \shortintertext{Now we bring out the factor of $L^{(2\qrep)}(2s+1,\repsq)^{-1}$:}
    L(2s,\repsq)
    \prod_{p \nmid 2\qrep} \left(1 - \frac{\alpha_1(p)^2}{(p+1)p^{2s}}\right)
    ={}&
    \frac{L(2s,\repsq)}{L^{(2\qrep)}(2s+1,\repsq)} \prod_{p \nmid 2\qrep} \left(1 - \frac{\alpha_1(p)^2}{(p+1)p^{2s}}\right) \left(1 - \frac{\alpha_1(p)^2}{p^{2s+1}}\right)^{-1}
    \\
    ={}&
    \frac{L(2s,\repsq)}{L^{(2\qrep)}(2s+1,\repsq)}
    \prod_{p \nmid 2\qrep} 
    \frac{1 - \frac{\alpha_1(p)^2}{(p+1)p^{2s}}}{1 - \frac{\alpha_1(p)^2}{p\cdot p^{2s}}}
    \\
    ={}&
    \frac{L(2s,\repsq)}{L^{(2\qrep)}(2s+1,\repsq)}
    \prod_{p \nmid 2\qrep} 
    \left(1 - \frac{\frac{\alpha_1(p)^2}{(p+1)p^{2s}} - \frac{\alpha_1(p)^2}{p\cdot p^{2s}}}{1 - \frac{\alpha_1(p)^2}{p\cdot p^{2s}}}\right)
    \\
    ={}&
    \frac{L(2s,\repsq)}{L^{(2\qrep)}(2s+1,\repsq)}
    \prod_{p \nmid 2\qrep} 
    \left(1 - \left(\frac{1}{p+1} - \frac{1}{p}\right)\frac{\alpha_1(p)^2}{p^{2s}}\frac{1}{1 - \frac{\alpha_1(p)^2}{p^{2s+1}}}\right)
    \\
    ={}&
    \frac{L(2s,\repsq)}{L^{(2\qrep)}(2s+1,\repsq)}
    \prod_{p \nmid 2\qrep} 
    \left(1 + \frac{1}{p+1}\frac{\alpha_1(p)^2}{p^{2s+1}}\frac{1}{1 - \frac{\alpha_1(p)^2}{p^{2s+1}}}\right)
    \\
    ={}&
    \frac{L(2s,\repsq)}{L^{(2\qrep)}(2s+1,\repsq)}
    \prod_{p \nmid 2\qrep} 
    \left(1 - \frac{1}{p+1} \cdot \frac{1}{1 - \left(\frac{\alpha_1(p)^2}{p^{2s+1}}\right)^{-1}}\right)
    .
    \qedhere
  \end{align*}
\end{proof}

\begin{lemma}
  \label{lemma:prod_convergence}
  For all $s \in \C$ such that $\Re(s) > -\half + \theta$, the infinite product
  \begin{align*}
    \prod_{p \nmid 2\qrep} \prodr \left(1 - \frac{1}{(p+1)(1 - \alpha_j(p)^{-2}p^{2s+1})}\right)
  \end{align*}
  converges absolutely.
\end{lemma}
\begin{proof}
  Follows from the fact that $|\alpha_j(p)| \ll p^\theta$.
\end{proof}

\section{Analyzing terms in the approximate functional equation averaged over a family of quadratic twists}
\label{sec:size}

The first part of this section is dedicated to analysis of the individual terms in the decomposition of \eqref{eq:afe_rep}, first into the two terms, \textit{term 1} and \textit{term 2} --- term 2 is the one in which the functional equation has been applied --- and then further decomposed into sums over squares \eqref{eq:afe_t1s} \eqref{eq:afe_t2s} and nonsquares \eqref{eq:afe_t1ns} \eqref{eq:afe_t2ns} in the Dirichlet series. The result of this decomposition is
\begin{align}
  \nonumber
  \Favg
  &L(s,\rep \otimes \chi_d)
  \\
  \label{eq:afe_t1ns}
  &=
  \Favg \sum_{\substack{n=1 \\ n \neq \square}}^\infty \frac{a_\rep(n)\chi_d(n)}{n^s}
  \Vp
  \\
  \label{eq:afe_t2ns}
  &+
  \prod_{j=1}^\repdim \frac{\halfGamma{1-s+\bar\kappa_j}}{\halfGamma{s+\kappa_j}} \Favg \sum_{\substack{n=1 \\ n \neq \square}}^\infty
  \frac{\rootnum_{\rep\otimes\chi_d} a_{\bar\rep}(n)\chi_d(n)}{n^{1-s}}
  \Vm
  \left(\frac{\pi^\repdim}{qd^\repdim}\right)^{\!s - \half}
  \\
  \label{eq:afe_t1s}
  &+
  \Favg \sum_{n=1}^\infty \frac{a_\rep(n^2)\chi_d(n^2)}{n^{2s}}
  \Vpsq
  \\
  \label{eq:afe_t2s}
  &+
  \prod_{j=1}^\repdim \frac{\halfGamma{1-s+\bar\kappa_j}}{\halfGamma{s+\kappa_j}} \Favg \sum_{n=1}^\infty \frac{\rootnum_{\rep\otimes\chi_d} a_{\bar\rep}(n^2)\chi_d(n^2)}{n^{2-2s}}
  \Vmsq
  \left(\frac{\pi^\repdim}{qd^\repdim}\right)^{\!s - \half}
  .
\end{align}
The analysis in this section of the four terms above form the basis of the proofs of \cref{thm:r1size,thm:r2size,thm:quadratic_murmurations,thm:gamma_murmurations}.  Two types of results are presented: bounds and approximations.
\Cref{lemma:t1ns_size,lemma:t2ns_size,lemma:t1s_size,lemma:t2s_size} give upper bounds on the sizes of \eqref{eq:afe_t1ns} \eqref{eq:afe_t2ns} \eqref{eq:afe_t1s} \eqref{eq:afe_t2s} respectively.

\Cref{thm:r1size,thm:r2size} need an approximation to the term 1 square sum \eqref{eq:afe_t1s}, not an upper bound. This is handled by \cref{lemma:t1s_approx}. In contrast, the proofs of \cref{thm:quadratic_murmurations,thm:gamma_murmurations} use only the upper bound from \cref{lemma:t1s_size}.

The main terms on the right hand sides of \cref{thm:quadratic_murmurations,thm:gamma_murmurations} come from term 2's sum over squares \eqref{eq:afe_t2s}, so this term also needs to be appropriately estimated. This is done in four steps, by \cref{lemma:t2s_approx_step1,lemma:t2s_approx_step2,lemma:t2s_approx_step3,lemma:t2s_approx_step4}. Each of these successive approximations preserves the holomorphicity of the main term (or, in the case of \cref{lemma:t2s_approx_step4}, meromorphicity). This is important for two reasons.

First, the estimate we are left with after all four lemmas are applied in sequence will be further manipulated as part of the integrand of an inverse Mellin transform, so it is essential that it be meromorphic.
  
Second, by breaking the process of approximating \eqref{eq:afe_t2s} into four steps, \cref{lemma:t2s_contrib_r1}, which processes the inverse Mellin transform of \eqref{eq:afe_t2s}, is able to shift the contour of integration in between each step. The abscissas at which the error terms from \cref{lemma:t2s_approx_step1,lemma:t2s_approx_step2,lemma:t2s_approx_step3,lemma:t2s_approx_step4} are smallest differ, so this decomposition helps the error term of \cref{thm:quadratic_murmurations,thm:gamma_murmurations}. This trick is essential for our application: if one were to insist that the errors of these successive approximations all be integrated along the same vertical line segments, then no choice of abscissa would yield an asymptotic result.

In \cref{sec:size_proof} these bounds and approximations are used to prove \cref{thm:r1size,thm:r2size}.

\subsection{Bounds}
\label{sec:size_bounds}

\begin{lemma}[Term 1 non-square sum \eqref{eq:afe_t1ns} bound]
  \label{lemma:t1ns_size}
  \begin{align*}
    &\Favg \nsqsum \frac{a_\rep(n)\chi_d(n)}{n^s} \Vp
    \ll \frac{\qrep^\half D^\half}{\#\cF} \mup^\half \left(1 + \mup^{\half - \sigma + \theta}\right) \epsfac
    .
  \end{align*}
\end{lemma}
\hypertarget{proof:t1ns_size}{}
\begin{proof}
  The main ingredient in this proof is the bound of \cref{lemma:stankus_2nd_moment}. To be able to apply the lemma, we use Cauchy--Schwarz to isolate the second moment character sum.

  We first apply \cref{lemma:V_approx} to truncate the sum in $n$ of the quantity in question:
  \begin{align}
    \label{eq:t1ns_size_0}
    &\Favg \nsqsum \frac{a_\rep(n)\chi_d(n)}{n^s} \Vp
    \ll \sum_{\substack{n < \mup^{1+\eps} \\ n \neq \square}} \frac{a_\rep(n)}{n^s} \Favg \chi_d(n)
    .
  \end{align}
  
  We now split the right hand side of \eqref{eq:t1ns_size_0} into a difference of two sums --- a decomposition which follows immediately from the definition of $\cF$:
  \begin{align}
    \label{eq:t1ns_size_2}
    \eqref{eq:t1ns_size_0}
    \ll \sum_{\substack{n < \mup^{1+\eps} \\ n \neq \square}} \frac{a_\rep(n)}{n^s} \frac{1}{\#\cF} \left(\,
    \starsum*_{\substack{d < D \\ d = \ell \mod \qrep}} \chi_d(n) 
    \;\;-
    \starsum*_{\substack{d < D_0 \\ d = \ell \mod \qrep}} \chi_d(n) 
    \right)
    .
  \end{align}

  By the triangle inequality, the difference \eqref{eq:t1ns_size_2} is $\ll$ its largest term. The bound we will use,
  \begin{align}
    \label{eq:t1ns_size_lemma_bound}
    \sum_{\substack{0 < n < N \\ n \neq \square \\ (n,q) = 1}} \,\left|\, \starsum*_{\substack{0 < d < D \\ d = \ell \mod q}} \chi_d(n) \,\right|^2
    \ll
    qND(\log D)^4\log N
    + q^{\frac{1}{3} + \eps} N^\half D \min\!\left\{D^{\frac{1}{7}}, q^{\frac{2}{3}}\right\}
  \end{align}
  from \cref{lemma:stankus_2nd_moment}, 
  is increasing in $D$, so the $d < D$ sum is the larger term:
  \begin{align}
    \label{eq:t1ns_size_3}
    \eqref{eq:t1ns_size_2}
    \ll
    \frac{1}{\#\cF} \sum_{\substack{n < \mup^{1+\eps} \\ n \neq \square}} \frac{a_\rep(n)}{n^s} \starsum*_{\substack{d < D \\ d = \ell \mod \qrep}} \chi_d(n)
    .
  \end{align}

  Cauchy--Schwarz gives
  \begin{align}
    \label{eq:t1ns_size_4}
    \eqref{eq:t1ns_size_3}
    &\ll
    \frac{1}{\#\cF}
    \left(\sum_{\substack{n < \mup^{1+\eps} \\ n \neq \square}} \left|\,\frac{a_\rep(n)}{n^s} \,\right|^2\right)^{\!\half}
    \left(\sum_{\substack{n < \mup^{1+\eps} \\ n \neq \square}} \left|\,\starsum*_{\substack{d < D \\ d = \ell \mod \qrep}} \chi_d(n) \,\right|^2\right)^{\!\half}
    .
    \shortintertext{Using the bound $a_\rep(n) \ll n^{\theta + \eps}$,}
    \nonumber
    \eqref{eq:t1ns_size_4}
    & \ll
    \frac{1}{\#\cF}
    \left( \left(1 + \mup^{1 - 2\sigma + 2\theta}\right)|\thalf - \sigma + \theta|^{-1} \right)^{\!\half}
    \left( D \mup \epsfac \right)^{\!\half}
    \\
    \nonumber
    & \ll
    \frac{\qrep^\half D^\half}{\#\cF} \mup^\half \left(1 + \mup^{\half - \sigma + \theta}\right) |\thalf - \sigma + \theta|^{-\half} \epsfac
    .
  \end{align}
  If $\thalf - \sigma + \theta = 0$, then the bound is instead
  \begin{align*}
    &\frac{\qrep^\half D^\half}{\#\cF} \mup^\half \log\!\mup \epsfac
    \ll \frac{\qrep^\half D^\half}{\#\cF} \mup^\half \epsfac
    .
  \end{align*}
  \Cref{lemma:t1ns_size} follows.
\end{proof}

\begin{lemma}[Term 2 non-square sum \eqref{eq:afe_t2ns} bound]
  \label{lemma:t2ns_size}
  \begin{align*}
    &\Favg \nsqsum
    \frac{\rootnum_{\rep\otimes\chi_d} a_{\bar\rep}(n)\chi_d(n)}{n^{1-s}} \Vm \Gfac \afefac
    \\
    &\hspace{2cm}
    \ll \frac{\qrep^\half D^\half}{\#\cF} (QD^r)^{\half - \sigma} \mum^\half \left(1 + \mum^{\sigma - \half + \theta}\right) \epsfac
    .
  \end{align*}
\end{lemma}
\hypertarget{proof:t2ns_size}{}
\begin{proof}
  Similar to the \hyperlink{proof:t1ns_size}{proof of \cref*{lemma:t1ns_size}}.
\end{proof}

\begin{lemma}[Term 1 square sum \eqref{eq:afe_t1s} bound]
  \label{lemma:t1s_size}
  \begin{align*}
    \Favg \sum_{n=1}^\infty \frac{a_\rep(n^2)\chi_d(n^2)}{n^{2s}} \Vpsq
    & \ll \left(1 + \mup^{\half - \sigma + \theta}\right) \epsfac
    .
  \end{align*}
\end{lemma}
\hypertarget{proof:t1s_size}{}
\begin{proof}
  Applying the triangle inequality and \cref{lemma:V_approx},
  \begin{align*}
    \Favg \sum_{n=1}^\infty \frac{a_\rep(n^2)\chi_d(n^2)}{n^{2s}} \Vpsq
    &\ll
    \sum_{n^2 < \mup^{1+\eps}} \frac{|a_\rep(n^2)|}{n^{2\sigma}}
    \\
    &\ll
    \left(1 + \mup^{\half - \sigma + \theta}\right) \epsfac
    ,
  \end{align*}
  where we have used the bound $a_\rep(n) \ll n^{\theta + \eps}$.
\end{proof}

\begin{lemma}[Term 2 square sum \eqref{eq:afe_t2s} bound]
  \label{lemma:t2s_size}
  \begin{align*}
    \Favg
    \sum_{n=1}^\infty
    \frac{\rootnum_{\rep\otimes\chi_d} a_{\bar\rep}(n^2)\chi_d(n^2)}{n^{2-2s}}
    \Vmsq
    &\Gfac \afefac
    \\
    &\ll (QD^r)^{\half - \sigma} \left(1 + \mum^{\sigma - \half + \theta}\right) \epsfac
    .
  \end{align*}
\end{lemma}
\hypertarget{proof:t2s_size}{}
\begin{proof}
  Similar to the \hyperlink{proof:t1s_size}{proof of \cref*{lemma:t1s_size}}.
\end{proof}

\subsection{Approximations}
\label{sec:size_approximations}

\Cref{lemma:t1s_approx} is used to produce the main term of \cref{prop:avgsize}. When $\repdim = 1$, one may apply \cref{lemma:euler_prod_rep_square} to the right hand side.

\begin{lemma}[Term 1 square sum approximation]
  \label{lemma:t1s_approx}
  If $\sigma > \thalf + \theta$,
  \begin{align*}
    \Favg \sum_{n=1}^\infty \frac{a_\rep(n^2)\chi_d(n^2)}{n^{2s}} \Vpsq
    &=
    \sum_{n=1}^\infty \frac{a_\rep(n^2)}{n^{2s}} \prod_{p\mid\frac{n}{(n,2q)}}\frac{p}{p+1}
    \\
    &
    + O\!\left(\frac{D^{\half}}{\#\cF}\qet + \mup^{\half - \sigma + \theta}\right) \epsfac
    .    
  \end{align*}
\end{lemma}
\hyperlink{proof:t1s_approx}{Click here to jump to the proof of \cref*{lemma:t1s_approx}.} It contains many of the elements used in the proofs of the next four lemmas.

An approximation of the square sum \eqref{eq:afe_t2s} is needed in \cref{lemma:t2s_contrib_r1,lemma:t2s_contrib_r1_smoothed}.
Each of \cref{lemma:t2s_approx_step1,lemma:t2s_approx_step2,lemma:t2s_approx_step3,lemma:t2s_approx_step4} feeds into the next, in the sense that the left hand side of each is the main term on the right hand side of the previous. As such, they may easily be applied successively.

\begin{lemma}[Term 2 square sum approximation, step 1]
  \label{lemma:t2s_approx_step1}
  \begin{align*}
    \Favg \sum_{n=1}^\infty \frac{a_{\bar\rep}(n^2)\chi_d(n^2)}{n^{2-2s}} \Vmsq
    &= \sum_{n=1}^\infty \frac{a_{\bar\rep}(n^2)}{n^{2-2s}} \VmDsq \Favg \chi_d(n^2)
    \\
    &+ O\!\left(\frac{\#\cF}{D} \left(1 + \mum^{\sigma - \half + \theta}\right) \epsfac\right)
    .
  \end{align*}
\end{lemma}
\hyperlink{proof:t2s_approx_step1}{Click here to jump to the proof of \cref*{lemma:t2s_approx_step1}.} It uses \cref{lemma:V_deriv_approx} as part of a first order series expansion approximating $d \approx D$ for all $d \in \cF$.

\begin{lemma}[Term 2 square sum approximation, step 2]
  \label{lemma:t2s_approx_step2}
  \begin{align*}
    \sum_{n=1}^\infty \frac{a_{\bar\rep}(n^2)}{n^{2-2s}}
    &\VmDsq \Favg \chi_d(n^2)
    \left(\frac{\pi^\repdim \x}{\qrep d^\repdim}\right)^{\!s - \half}
    \\
    &= \sum_{n=1}^\infty \frac{a_{\bar\rep}(n^2)}{n^{2-2s}} \prod_{p \mid 2qn} \frac{p}{p + 1} \VmDsq \frac{6\Eta_{\cF} \Delta D}{\pi^2 \phi(\qrep)\#\cF} \frac{1 - \left(1 - \frac{\Delta D}{D}\right)^{1 - r(s-\half)}}{(1 - r(s-\half))\frac{\Delta D}{D}}
    \xD^{\!s - \half}
    \\
    &+ O\Bigg(
    \frac{\qrep^{\frac{1}{6}} D^\half}{\#\cF} \left(D^{\frac{1}{14}} + |s|^{\frac{1}{6}}\right) \left(1 + \mum^{\sigma - \half + \theta}\right) \xD^{\!\sigma - \half} \epsfac 
    \Bigg)
    .
  \end{align*}
\end{lemma}
\hyperlink{proof:t2s_approx_step2}{Click here to jump to the proof of \cref*{lemma:t2s_approx_step2}.} The main idea is to use \cref{lemma:disc_power_sum} to evaluate the sum over $d$ on the left hand side. Only one of the two terms in the $\min$ of \cref{lemma:disc_power_sum}'s error term is preserved, as the other will not yield a good enough bound for our purposes.

\begin{lemma}[Term 2 square sum approximation, step 3]
  \label{lemma:t2s_approx_step3}
  \begin{align*}
    \sum_{n=1}^\infty
    \frac{a_{\bar\rep}(n^2)}{n^{2-2s}}
    &
    \prod_{p \mid 2qn} \frac{p}{p + 1} \VmDsq \frac{6\Eta_{\cF}\Delta D}{\pi^2 \phi(\qrep)\#\cF} \frac{1 - \left(1 - \frac{\Delta D}{D}\right)^{1 - \repdim(s-\half)}}{(1 - \repdim(s-\half))\frac{\Delta D}{D}} 
    \\
    &= \sum_{n=1}^\infty \frac{a_{\bar\rep}(n^2)}{n^{2-2s}} \prod_{p \mid \frac{n}{(n,2q)}} \frac{p}{p + 1} \VmDsq \frac{1 - \left(1 - \frac{\Delta D}{D}\right)^{1 - \repdim(s-\half)}}{(1 - \repdim(s-\half))\frac{\Delta D}{D}} 
    \\
    &+ O\!\left(
    \frac{D^\half}{\#\cF} \qet \left(1 + \mum^{\sigma - \half + \theta}\right) \min\!\left\{1,\, \frac{\Delta D}{|s| D}\right\} (\qrep D)^\eps
    \right)
    .
  \end{align*}
\end{lemma}
\hyperlink{proof:t2s_approx_step3}{Click here to jump to the proof of \cref*{lemma:t2s_approx_step3}.} Here \cref{lemma:Fsize} is used to relate $\#\cF$ and $\Delta D$.

\begin{lemma}[Term 2 square sum approximation, step 4]
  \label{lemma:t2s_approx_step4}
  For $\sigma < \thalf - \theta$,
  \begin{align*}
    &\sum_{n=1}^\infty
    \frac{a_{\bar\rep}(n^2)}{n^{2-2s}}
    \prod_{p \mid \frac{n}{(n,2q)}}
    \frac{p}{p+1}
    \VmDsq \frac{1 - \left(1 - \frac{\Delta D}{D}\right)^{1 - \repdim(s-\half)}}{(1 - \repdim(s-\half))\frac{\Delta D}{D}}
    \\
    &\hspace{4cm}
    =
    \sum_{n=1}^\infty \frac{a_{\bar\rep}(n^2)}{n^{2-2s}} \prod_{p \mid \frac{n}{(n,2q)}} \frac{p}{p+1}
    \frac{1 - \left(1 - \frac{\Delta D}{D}\right)^{1 - \repdim(s-\half)}}{(1 - \repdim(s-\half))\frac{\Delta D}{D}}
    \\&
    \hspace{4cm}
    + O\!\left(\mum^{\sigma - \half + \theta} \min\!\left\{ 1, \frac{D}{|s|\Delta D} \right\} \epsfac \right)
    .
  \end{align*}
\end{lemma}
\hyperlink{proof:t2s_approx_step4}{Click here to jump to the proof of \cref*{lemma:t2s_approx_step4}.} The idea is that, as the Dirichlet series on the right is absolutely convergent, its tail, which $V_{1-s}^*$ removes, may be added back.

\subsection{Proofs of \cref{thm:r1size,thm:r2size}}
\label{sec:size_proof}

\Cref{prop:avgsize} is a precursor to \cref{thm:r1size,thm:r2size}, prior to optimizing out $\xi$'s parameters $\alpha$ and $\beta$ and simplifying the error term.
\Cref{lemma:et_process} performs this optimization; we shall not prove that our choice indeed minimizes the error term with respect to $\alpha$ and $\beta$, but only that \cref{thm:r1size,thm:r2size} hold.

Notation may be recalled with the help of the glossary found in \cref{sec:glossary}.

\begin{proposition}
  \label{prop:avgsize}
  For any $s \in \C$ such that
  $\thalf + \theta < \sigma < 1$
  \begin{align*}
    \Favg L(s,\rep\otimes\chi_d)
    &=
    \frac{a_{\rep}(n^2)}{n^{2s}} \prod_{p \mid \frac{n}{(n,2q)}} \frac{p}{p+1}
    \\
    &+
    O\Bigg(
    \frac{\qrep^\half D^\half}{\#\cF} \mup^\half 
    \;+\;
    \frac{\qrep^\half D^\half}{\#\cF} (QD^r)^{\half - \sigma} \mum^{\sigma + \theta} 
    \\
    &\hspace{0.5cm}
    +
    \mup^{\half - \sigma + \theta}
    \;+\;
    (QD^r)^{\half - \sigma} \mum^{\sigma - \half + \theta}
    \;+\;
    \frac{D^\half}{\#\cF}\qet
    \Bigg)
    .
  \end{align*}
\end{proposition}

\hypertarget{proof:avgsize}{}
\begin{proof}[Proof of \cref{prop:avgsize}]
  Averaging the approximate functional equation \eqref{eq:afe_rep} over $\cF$ gives
  \begin{align}
    \nonumber
    \Favg
    &L(s,\rep\otimes\chi_d)
    \\
    &=
    \label{eq:is_t1ns}
    \Favg \sum_{\substack{n=1 \\ n \neq \square}}^\infty \frac{a_\rep(n)\chi_d(n)}{n^s}V_s\!\left(\frac{n}{\xi \sqrt{\qrep d^\repdim}}\right)
    \\
    \label{eq:is_t2ns}
    &+
    \Favg \sum_{\substack{n=1 \\ n \neq \square}}^\infty \frac{\rootnum_{\rep\otimes\chi_d} a_{\bar\rep}(n)\chi_d(n)}{n^{1-s}}V_{1-s}^*\!\left(\frac{n\xi}{\sqrt{\qrep d^\repdim}}\right) \Gfac \afefac
    \\
    \label{eq:is_t1s}
    &+
    \Favg \sum_{n=1}^\infty \frac{a_\rep(n^2)\chi_d(n^2)}{n^{2s}}V_s\!\left(\frac{n^2}{\xi\sqrt{\qrep d^\repdim}}\right)
    \\
    \label{eq:is_t2s}
    &+
    \Favg \sum_{n=1}^\infty \frac{\rootnum_{\rep\otimes\chi_d} a_{\bar\rep}(n^2)\chi_d(n^2)}{n^{2-2s}}V_{1-s}^*\!\left(\frac{n^2\xi}{\sqrt{\qrep d^\repdim}}\right) \Gfac \afefac
    .
  \end{align}
  \begin{align}
    \shortintertext{
      By \cref{lemma:t1ns_size},
    }
    \label{eq:is_t1ns_eval}
    \eqref{eq:is_t1ns}
    &\ll \frac{\qrep^\half D^\half}{\#\cF} \mup^\half \epsfac
    \shortintertext{
      By \cref{lemma:t2ns_size},
    }
    \label{eq:is_t2ns_eval}
    \eqref{eq:is_t2ns}
    &\ll \frac{\qrep^\half D^\half}{\#\cF} (QD^r)^{\half - \sigma} \mum^{\sigma + \theta} \epsfac
    \shortintertext{
      By \cref{lemma:t2s_size},
    }
    \label{eq:is_t2s_eval}
    \eqref{eq:is_t2s}
    &\ll (QD^r)^{\half - \sigma} \mum^{\sigma - \half + \theta} \epsfac
  \end{align}
  In \eqref{eq:is_t1ns_eval} \eqref{eq:is_t2ns_eval} \eqref{eq:is_t2s_eval} we have used the assumption that $\sigma > \thalf + \theta$ to omit certain terms.

  By \cref{lemma:t1s_approx}
  \begin{align}
    \label{eq:is_t1s_mt}
    \eqref{eq:is_t1s}
    ={}&
    \frac{a_{\rep}(n^2)}{n^{2s}} \prod_{p \mid \frac{n}{(n,2q)}} \frac{p}{p+1}
    \\
    \label{eq:is_t1s_et}
    &
    + O\!\left(\frac{D^\half}{\#\cF}\qet + \mup^{\half - \sigma + \theta}\right) \epsfac
    .
  \end{align}
  The main term of \cref{prop:avgsize} is \eqref{eq:is_t1s_mt},
  and the error terms are \eqref{eq:is_t1ns_eval} \eqref{eq:is_t2ns_eval} \eqref{eq:is_t2s_eval} \eqref{eq:is_t1s_et}.
\end{proof}

\begin{remark}
  \label{rem:avgsize_variant}
  A variant of \cref{prop:avgsize} is
  \begin{align*}
    \Favg L(s,\rep\otimes\chi_d)
    &= \Favg L^{(d)}(2s, \repsq)
    \\
    &+
    O\Bigg(
    \frac{\qrep^\half D^\half}{\#\cF} \mup^\half
    \;+\;
    \frac{\qrep^\half D^\half}{\#\cF} (QD^r)^{\half - \sigma} \mum^{\sigma + \theta}
    \\
    &\hspace{0.5cm}
    +
    \mup^{\half - \sigma + \theta}
    \;+\;
    (QD^r)^{\half - \sigma} \mum^{\sigma - \half + \theta}
    \Bigg)
    .
  \end{align*}
\end{remark}
Although both the left and right hand sides above are averages over $d \in \cF$, it is certainly not the case that $L(s,\rep\otimes\chi_d) \approx L^{(d)}(2s, \repsq)$. \Cref{rem:avgsize_variant} highlights that the sums over nonsquares \eqref{eq:is_t1ns} and \eqref{eq:is_t2ns} wash out when an average over $\cF$ is taken, as argued heuristically in \cite[\S 4.4]{cfkrs} and \cite[\S 2.2]{conrey_snaith}. In the region $\sigma < \thalf - \theta$, the counterpart \eqref{eq:is_t2s} of the sum over squares \eqref{eq:is_t1s} dominates, as can be seen from a simple invocation of the functional equation.

\hypertarget{proof:avgsize_variant}{}
\begin{proof}[Proof of \cref{rem:avgsize_variant}]
  The assumption $\sigma > \thalf + \theta$ also allows us to write
  \begin{align}
    \label{eq:is_t1s_mt_rem}
    \eqref{eq:is_t1s}
    &= \Favg \sum_{n=1}^\infty \frac{a_\rep(n^2)\chi_d(n^2)}{n^{2s}}
    \\
    \label{eq:is_t1s_et_rem}
    &+ \Favg \sum_{n=1}^\infty \frac{a_\rep(n^2)\chi_d(n^2)}{n^{2s}}\left(V_s\!\left(\frac{n^2}{\xi\sqrt{\qrep d^\repdim}}\right) - 1\right)
    ,
  \end{align}
  as \eqref{eq:is_t1s_mt_rem} is absolutely convergent.
  The error term \eqref{eq:is_t1s_et_rem} is bounded by
  \begin{align}
    \nonumber
    \eqref{eq:is_t1s_et_rem}
    &\ll \sum_{n^2 > \xi\sqrt{QD^r}} \frac{n^{2\theta}}{n^{2\sigma}} \;\epsfac
    \\
    \label{eq:is_t1s_et_rem_eval}
    &\ll \mup^{\half - \sigma + \theta} \epsfac
    .
  \end{align}
  The main term \eqref{eq:is_t1s_mt_rem} is identified to be
  \begin{align}
    \nonumber
    \eqref{eq:is_t1s_mt_rem}
    &= \Favg \sum_{\substack{n=1 \\ (n,d) = 1}}^\infty \frac{a_\rep(n^2)}{n^{2s}}
    \\
    \label{eq:is_t1s_mt_rem_eval}
    &= \Favg L^{(d)}(2s, \repsq)
    .
  \end{align}
  The rest of the argument from the \hyperlink{proof:avgsize}{proof of \cref*{prop:avgsize}} is preserved.
\end{proof}

For brevity, we give names to the error terms in \cref{prop:avgsize}:
\begin{align}
  \label{eq:ETdef}
  \begin{aligned}
  \ET_1 &\coloneqq \frac{\qrep^\half D^\half}{\#\cF} \mup^\half 
  \\
  \ET_2 &\coloneqq \frac{\qrep^\half D^\half}{\#\cF} (QD^r)^{\half - \sigma} \mum^{\sigma + \theta} 
  \\
  \ET_3 &\coloneqq \mup^{\half - \sigma + \theta} 
  \\
  \ET_4 &\coloneqq (QD^r)^{\half - \sigma} \mum^{\sigma - \half + \theta} 
  \\
  \ET_5 &\coloneqq \frac{D^\half}{\#\cF}\qet 
  \\
  \ET &\coloneqq \ET_1 + \ET_2 + \ET_3 + \ET_4 + \ET_5
  .
  \end{aligned}
\end{align}

\begin{lemma}
  \label{lemma:et_process}
  If $r = 1$, then
  \begin{align*}
    &\text{if }\,
    \#\cF \leqslant D^\half (|s - \kappa| D)^{\frac{\sigma}{2\sigma + 1}}
    \text{ then }\,
    \ET \ll \frac{\qrep^{\frac{3}{4}} D^\half}{\#\cF} \sqrt{|s - \kappa| D}^{\,\frac{1}{1 + 2\sigma}}
    \shortintertext{and}
    &\text{if }\,
    \#\cF \geqslant D^\half (|s - \kappa| D)^{\frac{\sigma}{2\sigma + 1}}
    \text{ then }\,
    \ET \ll \qrep^{\frac{3}{4}} \left(\frac{D^\half}{\#\cF}\right)^{\!\frac{2\sigma - 1}{2\sigma}}
    .
  \end{align*}  
  If $r \geqslant 2$, then
  \begin{align*}
    \ET \ll
    \frac{\qrep^{\frac{3}{4}} D^\half}{\#\cF} \sqrt{D \prodr |s - \kappa_j|^{\frac{1}{r}}}^{\,\frac{1 + 2\theta}{1 + 2\sigma + 2\theta}r}
    .
  \end{align*}
\end{lemma}

\begin{proof}
  The $q$ aspect of \cref{lemma:et_process} follows from a cursory examination of \eqref{eq:ETdef}. Recall that $\sigma > \thalf + \theta$.

  Regarding the $D$- and $s$- aspects,
  define $\gamma$ and $\delta$ so that
  \begin{align}
    \nonumber
    \prodr |s - \kappa_j| &\eqqcolon D^{r\gamma} 
    \quad\text{and}\quad
    \#\cF \eqqcolon D^\delta
    .
  \end{align}

  It will also be convenient to define
  \begin{align}
    \nonumber
    z \coloneqq \alpha + \beta\gamma
    \quad\text{and}\quad
    w \coloneqq 1 + \gamma
  \end{align}
  Set $\beta = 0$; it's not useful here.
  With $\log_D$ denoting the logarithm to base $D$,
  \begin{align}
    \label{eq:et1_zw}
    \log_D(\ET_1) &\sim_q \half - \delta
    + \half\left(\frac{r}{2}w + 2z\right)
    \\
    \log_D(\ET_2) &\sim_q \half - \delta
    + (\sigma + \theta)\left(\frac{r}{2}w + 2z\right) + \left(\half - \sigma\right)rw
    \\
    \log_D(\ET_3) &\sim_q \left(\sigma - \half + \theta\right)\left(\frac{r}{2}w - 2z\right) + \left(\half - \sigma\right)rw
    \\
    \label{eq:et4_zw}
    \log_D(\ET_4) &\sim_q \left(\half - \sigma + \theta\right)\left(\frac{r}{2}w + 2z\right)
    \\
    \log_D(\ET_5) &\sim_q \half - \delta + \log_D(\qet)
    \label{eq:et5_zw}
  \end{align}

  Let's first look at the case $\fbox{$r = 1$}$. Here $\theta = 0$.
  Define
  \begin{align}
    \nonumber
    z_1 &\coloneqq -\frac{w}{4}\frac{2\sigma - 1}{2\sigma + 1}
    \\
    \nonumber
    z_2 &\coloneqq -\frac{w}{4} + \frac{2\delta - 1}{4\sigma}
    \shortintertext{and}
    \nonumber
    e_1 &\coloneqq \frac{1}{2 + 4\sigma}w + \half - \delta
    \\
    \nonumber
    e_2 &\coloneqq -\frac{1}{\sigma}\left(\sigma - \half\right)\!\left(\delta - \half\right)
    .
  \end{align}

  On can verify by substituting into \eqref{eq:et1_zw}---\eqref{eq:et5_zw} that
  \begin{align}
    \label{eq:e1ineq}
    &\text{if }
    \delta \leqslant \half + \frac{\sigma w}{2\sigma + 1}
    \text{ then }
    e_1 \geqslant \log_D(\ET_1)\Big|_{z=z_1}, \dots, \log_D(\ET_5)\Big|_{z=z_1}
    \shortintertext{and}
    &\text{if }
    \delta \geqslant \half + \frac{\sigma w}{2\sigma + 1}
    \text{ then }
    e_2 \geqslant \log_D(\ET_1)\Big|_{z=z_2}, \dots, \log_D(\ET_5)\Big|_{z=z_2}
    .
    \label{eq:e2ineq}
  \end{align}

  If $\fbox{$r \geqslant 2$}$, define
  \begin{align}
    \nonumber
    z_3 &\coloneqq -\frac{rw}{4}\frac{2\sigma - 1 - 2\theta}{2\sigma + 1 + 2\theta}
    \shortintertext{and}
    \nonumber
    e_3 &\coloneqq \frac{1 + 2\theta}{2 + 4\sigma + 4\theta}rw + \half - \delta
  \end{align}
  On can verify by substituting into \eqref{eq:et1_zw}---\eqref{eq:et5_zw} that
  \begin{align}
    \label{eq:e3ineq}
    e_3 \geqslant \log_D(\ET_1)\Big|_{z=z_3}, \dots, \log_D(\ET_5)\Big|_{z=z_3}
  \end{align}

  The $D$- and $s$- aspects of \cref{lemma:et_process} follow from \eqref{eq:e1ineq}, \eqref{eq:e2ineq}, and \eqref{eq:e3ineq}; we choose $\alpha$ so that $z = z_1,z_2,z_3$ as appropriate. 
\end{proof}

\Cref{prop:avgsize} and \cref{lemma:et_process} combine to yield \cref{thm:r1size,thm:r2size}.

\section{Inverse Mellin transform of the average $L$-function}
\label{sec:contrib}

\subsection{Outline}
\label{sec:contrib_outline}

Recall that our strategy for proving \cref{thm:quadratic_murmurations}, described in \cref{sec:outline_construction}, is to evaluate the truncated inverse Mellin transform
\begin{align*}
  \frac{1}{2\pi i}\int_{c - iT}^{c + iT} \Favg L(s,\rep\otimes\chi_d)\,\x^{s-\half}\,\frac{ds}{s}
\end{align*}
in two different ways. First, using Perron's formula,
\begin{align*}
  \frac{1}{2\pi i}\int_{c - iT}^{c + iT} \Favg L(s,\rep\otimes\chi_d)\,\x^{s-\half}\,\frac{ds}{s}
  =
  \Favg \frac{1}{\sqrt{\x}} \sum_{n < x} a_\rep(n)\chi_d(n)
  \;+\;
  O(x^{c-\half} T^{-1} (xT\qrep)^\eps)
  .
\end{align*}
Second, from the approximate functional equation \eqref{eq:afe_rep} averaged over $\cF$ and decomposed as \eqref{eq:afe_t1ns} \eqref{eq:afe_t2ns} \eqref{eq:afe_t1s} \eqref{eq:afe_t2s},
\begin{align}
  \nonumber
  \frac{1}{2\pi i}
  &
  \int_{c - iT}^{c + iT} \Favg L(s,\rep \otimes \chi_d)\,\x^{s-\half}\,\frac{ds}{s}
  \\
  \label{eq:afe_unweighted_t1ns_re}
  &=
  \frac{1}{2\pi i} \int_{c - iT}^{c + iT} \Favg \sum_{\substack{n=1 \\ n \neq \square}}^\infty \frac{a_\rep(n)\chi_d(n)}{n^s}
  \Vp
  \x^{s-\half}\,\frac{ds}{s}
  \\
  \label{eq:afe_unweighted_t2ns_re}
  &+
  \frac{1}{2\pi i} \int_{c - iT}^{c + iT} \prod_{j=1}^\repdim \frac{\halfGamma{1-s+\bar\kappa_j}}{\halfGamma{s+\kappa_j}} \Favg \sum_{\substack{n=1 \\ n \neq \square}}^\infty \frac{\rootnum_{\rep\otimes\chi_d} a_{\bar\rep}(n)\chi_d(n)}{n^{1-s}}
  \Vm
  \xd^{\!s-\half}\frac{ds}{s}
  \\
  \label{eq:afe_unweighted_t1s_re}
  &+
  \frac{1}{2\pi i} \int_{c - iT}^{c + iT} \Favg \sum_{n=1}^\infty \frac{a_\rep(n^2)\chi_d(n^2)}{n^{2s}}
  \Vpsq
  \x^{s-\half}\,\frac{ds}{s}
  \\
  \label{eq:afe_unweighted_t2s_re}
  &+
  \frac{1}{2\pi i} \int_{c - iT}^{c + iT} \prod_{j=1}^\repdim \frac{\halfGamma{1-s+\bar\kappa_j}}{\halfGamma{s+\kappa_j}} \Favg \sum_{n=1}^\infty \frac{\rootnum_{\rep\otimes\chi_d} a_{\bar\rep}(n^2)\chi_d(n^2)}{n^{2-2s}}
  \Vmsq
  \xd^{\!s-\half}\frac{ds}{s}
  .
\end{align}

\Cref{lemma:t1ns_contrib_r1,lemma:t2ns_contrib_r1,lemma:t1s_contrib_r1,lemma:t2s_contrib_r1} below process the four terms \eqref{eq:afe_unweighted_t1ns_re}, \eqref{eq:afe_unweighted_t2ns_re}, \eqref{eq:afe_unweighted_t1s_re}, and \eqref{eq:afe_unweighted_t2s_re} respectively. The sole main term comes from \eqref{eq:afe_unweighted_t2s_re}.

After setting $\repdim = 1$ and applying \cref{lemma:t1ns_contrib_r1,lemma:t2ns_contrib_r1,lemma:t1s_contrib_r1,lemma:t2s_contrib_r1}, we'll be left with the approximation
\begin{align}
  \nonumber
  \Favg
  &\frac{1}{\sqrt{\x}} \sum_{n < x} a_\rep(n)\chi_d(n)
  \;+\; O(x^{c-\half} T^{-1} (xT\qrep)^\eps)
  \\
  \label{eq:after_contrib_lemmas_mt}
  ={}&
  \frac{\rootnum_{\cF}}{2\pi i}\int_{\eps - iT}^{\eps + iT} \frac{\halfGamma{1-s+\bar\kappa}}{\halfGamma{s+\kappa}}
  \frac{L(2-2s, \barrepsq)}{L^{(2\qrep)}(3-2s, \barrepsq)}
  \mspace{1mu}
  \prod_{p \nmid 2\qrep}
  \left(1 - \frac{1}{(p+1)(1 - \bar\alpha_1(p)^{-2}p^{3-2s})}\right)
  \\
  \nonumber
  &\hspace{8cm}
  \cdot
  \frac{1 - \left(1 - \frac{\Delta D}{D}\right)^{1 - (s-\half)}}{(1 - (s-\half))\frac{\Delta D}{D}}
  \left(\frac{\pi\x}{\qrep D}\right)^{\!s - \half} \frac{ds}{s}
  \\
  \label{eq:after_contrib_lemmas_et}
  &+ O\!\left(D^\rho \qrep^{\frac{3}{4}} + \frac{|\kappa|\Delta D}{D}\right)\!\epsfac
  ,
\end{align}
with the exponent $\rho$ in the error term depending on the parameters $\alpha$ and $\beta$ imbued in the approximate functional equation's weights $\V$. These parameters, as well as $T$, will be ``optimized out'', i.e.\ set to the values which minimize the error term. We won't prove that our choices do in fact minimize the error term, only that the claimed error term is obtained.

It will at this point in the proof of \cref{thm:quadratic_murmurations} be desirable to
\begin{enumerate}[label=\arabic*.]
\item
  Remove the factor of
  $$\frac{1 - \left(1 - \frac{\Delta D}{D}\right)^{1 - (s-\half)}}{(1 - (s-\half))}$$
  appearing in the integrand of \eqref{eq:after_contrib_lemmas_mt}, and
\item
  Extend the integral so that it runs from $c - i\infty$ to $c + i\infty$, instead of being truncated at $|\Im(s)| < T$.
\end{enumerate}
We achieve these two goals in \cref{lemma:eliminate_derivative_factor}, at the cost of introducing a factor of $1/12$ in the exponent of $D$ in our error term. 
The scalar $1/12 = 1/4 - 1/6$ comes from the subconvexity bound of Petrow and Young \cite{petrow_young:2023}.
At this point \cref{thm:quadratic_murmurations} will quickly follow.

The proof \cref{thm:gamma_murmurations}, covered in \cref{sec:contrib_weighted}, is similar but substantially more straightforward in comparison. The simplicity is thanks to the relatively quick decay of the factor $\mf$ in the integrand, which makes the situation much less delicate.

\subsection{Proof of \cref{thm:quadratic_murmurations}}
\label{sec:contrib_unweighted}

\Cref{lemma:t1ns_contrib_r1,lemma:t2ns_contrib_r1,lemma:t1s_contrib_r1,lemma:t2s_contrib_r1}, for handling \eqref{eq:afe_unweighted_t1ns_re} \eqref{eq:afe_unweighted_t2ns_re} \eqref{eq:afe_unweighted_t1s_re} \eqref{eq:afe_unweighted_t2s_re}, in turn rely on \cref{lemma:t1ns_size,lemma:t2ns_size,lemma:t1s_size,lemma:t2s_approx_step1,lemma:t2s_approx_step2,lemma:t2s_approx_step3,lemma:t2s_approx_step4}, which bound individual terms in the integrand.
We integrate each of these terms separately. 
Prior to invoking the bounds from \cref{lemma:t1ns_size,lemma:t2ns_size,lemma:t1s_size,lemma:t2s_approx_step1,lemma:t2s_approx_step2,lemma:t2s_approx_step3,lemma:t2s_approx_step4}, we shift the contour of integration to run along a vertical line with real part well suited for that individual term, using \cref{lemma:t1s_size,lemma:t2s_size} to bound the contribution of integrals along horizontal segments. This procedure reduces the error term, and is mandatory lest the error term dwarf the main term. The sole main term comes emerges from \cref{lemma:t2s_contrib_r1}.

We will be left with many error terms to balance, and there will be many parameters involved: in addition to the relative sizes of $D$, $T$, and $\#\cF$, we'll have the seven real parts of the lines along which we integrate our seven terms from \cref{lemma:t1ns_size,lemma:t2ns_size,lemma:t1s_size,lemma:t2s_approx_step1,lemma:t2s_approx_step2,lemma:t2s_approx_step3,lemma:t2s_approx_step4}. Importantly, we'll also still have in reserve the flexibility to choose the parameter $\beta$ appearing in the definition \eqref{eq:xidef} of $\xi$. The parameter $\alpha$ will turn out not to be helpful here.

In what follows we consider only the case where $(qD)^{1-\eps} \ll \x \ll (qD)^{1+\eps}$
--- which we will abbreviate \hypertargetc{hyper:asymppmdef}{$\x \asymp (qD)^{1 \pm \eps}$} --- 
as that is the region of interest for murmurations; the asymptotic case in which one of $\x$ or $qD$ is much larger than the other is easier to understand and has been explored elsewhere. Our formulas for murmurations do hold outside the range $\x \asymp (qD)^{1 \pm \eps}$, and one can use them to recover the more classically studied asymptotic cases. See for example \cite[\S 3]{conrey_snaith} for the idea of the connection.

To optimize our error term in the various cases of interest it was helpful to implement the situation in code \cite{github}. Our implementation has $47$ error terms and $10$ parameters. Numerical experimentation and minimization in a variety of cases guided our calculations and helped us find the balance of error terms we present here.

Throughout the current subsection set $\repdim = 1$. This implies that 
\begin{align}
  \label{eq:chidef}
  \rep &\eqqcolon |\mspace{0mu}\cdot\mspace{0mu}|^{i\tau}\chi,
  \\
  \label{eq:kappadef}
  \kappa &\coloneqq \kappa_1 = \frac{1 - \chi(-1)}{2} - i\tau,
  \\
  \nonumber
  \theta &= 0.
\end{align}
Impose that $\tau \ll (D/\#\cF)^{1 - \eps}$.
Let $\gamma$ and $\delta$ be such that
\begin{align}
  \label{eq:gammadeltadef}
  T \eqqcolon D^\gamma \quad\text{and}\quad \#\cF \eqqcolon D^\delta,
\end{align}
and assume further that $\tfrac{5}{6} < \delta < 1$.
The larger context is such that we do not have control over $\delta$, but we do have control over $\gamma$: we are given a family of $L$-functions (governed by $\delta$), but may choose the manner in which we integrate related quantities (governed by $\gamma$).

To optimize our parameters, as a function of $\delta$, we take
\begin{align}
  \label{eq:alphahatdef}
  \alpha &\,\gets\, \hat{\alpha} \,\coloneqq\, 0
  \\
  \label{eq:betahatdef}
  \beta &\,\gets\, \hat{\beta} \,\coloneqq\, \frac{2 - 3(\delta - \frac{5}{6})}{24 + 36(\delta - \frac{5}{6})}
  \\
  \label{eq:gammahatdef}
  \gamma &\,\gets\, \hat{\gamma} \,\coloneqq\, \half + \frac{3}{4}\left(\delta - \frac{5}{6}\right)
  .
\intertext{
  (Here ``$\gets$'' denotes variable assignment, like in pseudocode. One could write ``$=$'' instead, as long as one understands that prior appearances of these variables do not inherit the values above.) The assignments \eqref{eq:alphahatdef}, \eqref{eq:betahatdef}, \eqref{eq:gammahatdef} ``optimize out'' the parameters $\alpha$, $\beta$, and $\gamma$, i.e.\ sets them to the values which minimizes the error term. The exponent $\rho$ in \eqref{eq:after_contrib_lemmas_et} of the error term becomes
  }
  \label{eq:rhohatdef}
  \rho &\,\gets\, \hat{\rho} \,\coloneqq
  \begin{cases}
    -\frac{3}{4}\left(\delta - \frac{5}{6}\right) & \text{if $\delta \leqslant \tfrac{13}{14}$,}
    \\
    \delta - 1 & \text{if $\delta \geqslant \tfrac{13}{14}$.}
  \end{cases}
\end{align}

The best possible power savings is attained when
$T \asymp D^{\frac{4}{7}}$, $\#\cF \asymp D^{\frac{13}{14}}$, $\beta = \tfrac{1}{16}$. The real parts of the lines of integration will be given throughout the proof as they arise. 


When $\delta \leqslant \tfrac{13}{14}$, the constraints on the error term are \eqref{eq:t1ns_contrib_step2_mt_eval} from the proof of \cref{lemma:t1ns_contrib_r1}, \eqref{eq:t2ns_contrib_step2_mt_eval} from the proof of \cref{lemma:t2ns_contrib_r1}, and the error term in Perron's formula \eqref{eq:perron}. When $\delta \geqslant \tfrac{13}{14}$, the constraints instead come from the error term called $R_1$ in the proof of \cref{lemma:t2s_contrib_r1}, which is to do with the approximation $d \approx D$ for $d \in \cF$. If $\delta = \tfrac{13}{14}$ precisely, then the error term is additionally constrained by \eqref{eq:t1s_contrib_mt_eval} from the proof of \cref{lemma:t1s_contrib_r1} (but the error term is not constrained by \eqref{eq:t1s_contrib_mt_eval} if $\delta > \tfrac{13}{14}$ or $\delta < \tfrac{13}{14}$).

\begin{lemma}[Term 1 non-square contribution]
  \label{lemma:t1ns_contrib_r1}
  If
  $0 < c \leqslant \thalf$,
  $r = 1$,
  $x \asymp (\qrep D)^{1 \pm \eps}$,
  $T = D^{\hat{\gamma}}$,
  $\#\cF = D^{\delta}$,
  and
  $\beta = \hat{\beta}$,
  \begin{align*}
    &\frac{1}{2\pi i}\int_{c - iT}^{c + iT} \Favg \nsqsum \frac{a_\rep(n)\chi_d(n)}{n^s} \Vp \x^{s-\half} \,\frac{ds}{s}
    \,\ll\, D^{\hat{\rho}} \qrep^{\frac{3}{4}} \epsfac
    .
  \end{align*}
\end{lemma}

\hypertarget{proof:t1ns_contrib_r1}{}
\begin{proof}
  The integrand is holomorphic in the strip $0 < \sigma < 1$. Replace the integral with one along the line segments $c - iT \longrightarrow \eps - iT \longrightarrow \eps + iT \longrightarrow c + iT$, i.e.\ write

  \begin{align}
    \nonumber
    \frac{1}{2\pi i}
    &\int_{c - iT}^{c + iT} \Favg \nsqsum \frac{a_\rep(n)\chi_d(n)}{n^s} \Vp \x^{s-\half} \,\frac{ds}{s}
    \\
    \label{eq:t1ns_contrib_step1_mt}
    &=
    \frac{1}{2\pi i}\int_{\eps - iT}^{\eps + iT} \Favg \nsqsum \frac{a_\rep(n)\chi_d(n)}{n^s} \Vp \x^{s-\half} \,\frac{ds}{s}
    \\
    \label{eq:t1ns_contrib_step1_et}
    &+
    \frac{1}{2\pi i}\left(\int_{c - iT}^{\eps - iT} + \int_{\eps + iT}^{c + iT}\right) \Favg \nsqsum \frac{a_\rep(n)\chi_d(n)}{n^s} \Vp \x^{s-\half} \,\frac{ds}{s}
    .
  \end{align}
  
  Let's bound \eqref{eq:t1ns_contrib_step1_et} before continuing with \eqref{eq:t1ns_contrib_step1_mt}. We will look at only one of the two terms. The other is handled in the same way and satisfies the same bound. By \cref{lemma:t1ns_size},
  \begin{align}
    \nonumber
    \int_{\eps + iT}^{c + iT}
    &\Favg \nsqsum \frac{a_\rep(n)\chi_d(n)}{n^s} \Vp x^{s-\half}\,\frac{ds}{s}
    \\
    \nonumber
    &\ll \int_{\eps + iT}^{c + iT} \frac{\qrep^\half D^\half}{\#\cF} \mup^\half \left(1 + \mup^{\half - \sigma + \theta}\right) \epsfac x^{\sigma - \half} \frac{ds}{s}
    \\
    \label{eq:t1ns_contrib_step1_et_eval1}
    &\ll \frac{\qrep^\half D^\half}{\#\cF} \mup^\half \left(1 + \mup^{\half - c + \theta}\right) x^{c - \half} s^{-1} \epsfac \bigg\vert_{s = c + iT}
    \\
    \label{eq:t1ns_contrib_step1_et_eval1_redundant}
    &+ \frac{\qrep^\half D^\half}{\#\cF} \mup^\half \left(1 + \mup^{\half - c + \theta}\right) x^{c - \half} s^{-1} \epsfac \bigg\vert_{s = \eps + iT}
    .
  \end{align}
  The term \eqref{eq:t1ns_contrib_step1_et_eval1_redundant} is at least factor of $T$ smaller than \cref{eq:t1ns_contrib_step1_mt}, so we needn't bother with it. Turning now to \eqref{eq:t1ns_contrib_step1_et_eval1},
  we have assumed that $\tau \ll (D/\#\cF)^{1 - \eps} \ll D^{1 - \delta - \eps}$. We have also assumed that $\delta > \tfrac{5}{6}$. Since $T = D^{\hat{\gamma}} \gg D^{\half}$, we have $|s - \kappa| \asymp T$ along the horizontal segment $\eps + iT \longrightarrow c + iT$. When $c \leqslant \thalf$,
  \begin{align}
    \nonumber
    \eqref{eq:t1ns_contrib_step1_et_eval1}
    &\ll
    \qrep^\half D^{\half-\delta} \left(\qrep^{\half} D^{\frac{1}{2}} T^{\frac{1}{2} + 2\hat\beta} \right)^{1-c}
    (qD)^{c - \half}
    T^{-1} (qDT)^\eps
    \\
    \label{eq:t1ns_contrib_step1_et_eval2}
    &\ll
    \qrep^{\half + \frac{c}{2}} D^{\half + \frac{c}{2} - \delta}  T^{(1-c)(\frac{1}{2} + 2\hat\beta) - 1}
    (qDT)^\eps
    .
  \end{align}
  
  The exponent of $\qrep$ in \eqref{eq:t1ns_contrib_step1_et_eval2} is at most $\tfrac{3}{4}$ (since we are assuming $c \leqslant \thalf$), so \eqref{eq:t1ns_contrib_step1_et_eval2}, which bounds the contribution of the horizontal integrals \eqref{eq:t1ns_contrib_step1_et}, confirms to \cref{lemma:t1ns_contrib_r1}'s stated bound in the $\qrep$ aspect.

  Writing $T = D^{\hat{\gamma}}$ in \eqref{eq:t1ns_contrib_step1_et_eval2}, the exponent of $D$ becomes
  \begin{align*}
    \half
    &+
    \frac{c}{2} - \delta + \left((1-c)\left(\frac{1}{2} + 2\hat\beta\right) - 1\right)\hat{\gamma}
    \\
    &=
    \half + \frac{c}{2} - \delta + \left((1-c)\left(\frac{1}{2} + 2\frac{2 - 3(\delta - \frac{5}{6})}{24 + 36(\delta - \frac{5}{6})}\right) - 1\right)\left(\half + \frac{3}{4}\left(\delta - \frac{5}{6}\right)\right)
    \\
    &=
    \frac{-2c\delta + 3c - 12\delta + 6}{8}
    .
  \end{align*}
  Since
  \begin{align*}
    &\frac{\dee}{\dee c} \Big(-2c\delta + 3c - 12\delta + 6\Big) = 3 - 2\delta > 0 \text{ for $\tfrac{5}{6} < \delta < 1$, and}
    \\
    &\frac{\dee}{\dee \delta} \Big(-2c\delta + 3c - 12\delta + 6\Big) = -2c - 12 < 0 \text{ for $0 < c < \thalf$,}
  \end{align*}
  the supremum, as a function of $0 < c \leqslant \thalf$ and $\tfrac{5}{6} < \delta < 1$ is
  \begin{align*}
    \frac{-2c\delta + 3c - 12\delta + 6}{8}\Bigg|_{c,\delta = \half,\frac{5}{6}} = -\frac{5}{12}
    .
  \end{align*}
  Since $\hat\rho \geqslant -\tfrac{1}{14} \geqslant -\tfrac{5}{12}$, we have verified that the horizontal integrals \eqref{eq:t1ns_contrib_step1_et} contribute an amount bounded by $D^{\hat\rho}$ in the $D$ aspect.

  We now turn to the matter of bounding \eqref{eq:t1ns_contrib_step1_mt}. Applying \cref{lemma:t1ns_size} once again,
  \begin{align}
    \nonumber
    \eqref{eq:t1ns_contrib_step1_mt}
    &=
    \frac{1}{2\pi i}
    \int_{\eps - iT}^{\eps + iT}
    \Favg \nsqsum \frac{a_\rep(n)\chi_d(n)}{n^s} \Vp x^{s-\half}\,\frac{ds}{s}
    \\
    \nonumber
    &\ll \int_{\eps - iT}^{\eps + iT} \frac{\qrep^\half D^\half}{\#\cF} \mup^\half \left(1 + \mup^{\half - \sigma + \theta}\right) \epsfac x^{\sigma - \half} \frac{ds}{s}
    \\
    \nonumber
    &\ll D^{\half-\delta} (qD)^\eps \int_{\eps - iT}^{\eps + iT}  \xi Q^\half  \frac{ds}{s}    
    \\
    \nonumber
    &\ll \qrep^\half D^{\half-\delta} (qD)^\eps \int_{\eps - iT}^{\eps + iT}  |s - s_0|^{2\beta} |s - \kappa|^\half  \frac{ds}{s}
    \\
    \label{eq:t1ns_contrib_step2_mt}
    &\ll \qrep^\half D^{\half-\delta} \left(T^{\half + 2\beta} + \tau^\half |s_0|^{2\beta}\right) (qDT)^\eps 
    .
  \end{align}
  Our assumptions ensure that $\tau^\half |s_0|^{2\beta} \ll T^{\half + 2\beta}$, and it is the case that $\hat{\beta} > 0$, so we may drop the second term in \eqref{eq:t1ns_contrib_step2_mt}. Substituting $T = D^{\hat{\gamma}}$ into the remaining term, the exponent of $D$ becomes
  \begin{align}
    \label{eq:t1ns_contrib_step2_mt_eval}
    \half - \delta + \left(\half + 2\hat\beta\right)\hat\gamma
    &=
    -\frac{3}{4}\left(\delta - \frac{5}{6}\right)
    \leqslant \hat{\rho}
    ,
  \end{align}
  which conforms to the bound of \cref{lemma:t1ns_contrib_r1} in the $D$ aspect. By inspection \eqref{eq:t1ns_contrib_step2_mt} also conforms in the $\qrep$ aspect.
\end{proof}

\begin{lemma}[Term 2 non-square contribution]
  \label{lemma:t2ns_contrib_r1}
  If
  $0 < c \leqslant \thalf$,
  $r = 1$,
  $x \asymp (\qrep D)^{1 \pm \eps}$,
  $T = D^{\hat{\gamma}}$,
  $\#\cF = D^{\delta}$,
  and
  $\beta = \hat{\beta}$,
  \begin{align*}
    &\frac{1}{2\pi i}\int_{c - iT}^{c + iT} \Gfac \Favg \nsqsum \frac{a_{\bar\rep}(n)\chi_d(n)}{n^{1-s}} \Vm \xd^{\!s-\half} \,\frac{ds}{s}
    \,\ll\, D^{\hat{\rho}} \qrep^{\frac{3}{4}} \epsfac 
    .
  \end{align*}
\end{lemma}

\hypertarget{proof:t2ns_contrib_r1}{}
\begin{proof}
  This proof is similar to the \hyperlink{proof:t1ns_contrib_r1}{proof of \cref*{lemma:t1ns_contrib_r1}}, and will be a bit more brief. 
  Shift the line of integration to $\sigma = \thalf$, i.e.\ write
  \begin{align}
    \nonumber
    \frac{1}{2\pi i}
    &\int_{c - iT}^{c + iT} \Gfac \Favg \nsqsum \frac{a_{\bar\rep}(n)\chi_d(n)}{n^{1-s}} \Vm \xd^{\!s-\half} \,\frac{ds}{s}
    \\
    \label{eq:t2ns_contrib_step1_mt}
    &=
    \frac{1}{2\pi i}\int_{\half - iT}^{\half + iT} \Gfac \Favg \nsqsum \frac{a_{\bar\rep}(n)\chi_d(n)}{n^{1-s}} \Vm \xd^{\!s-\half} \,\frac{ds}{s}
    \\
    \label{eq:t2ns_contrib_step1_et}
    &+
    \frac{1}{2\pi i}\left(\int_{c - iT}^{\half - iT} + \int_{\half + iT}^{c + iT}\right) \Gfac \Favg \nsqsum \frac{a_{\bar\rep}(n)\chi_d(n)}{n^{1-s}} \Vm \xd^{\!s-\half} \,\frac{ds}{s}
    .
  \end{align}
  Bounding the horizontal segments using \cref{lemma:t2ns_size},
  \begin{align}
    \nonumber
    \eqref{eq:t2ns_contrib_step1_et}
    &\ll
    \int_{\half + iT}^{c + iT}
    \frac{\qrep^\half D^\half}{\#\cF} (QD^r)^{\half - \sigma} \mum^\half \left(1 + \mum^{\sigma - \half + \theta}\right) \epsfac x^{\sigma - \half}\,\frac{ds}{s}
    \\
    \label{eq:t2ns_contrib_step1_et_eval1_redundant}
    &\ll \frac{\qrep^\half D^\half}{\#\cF} (QD^r)^{\half - \sigma} \mum^\half \left(1 + \mum^{\sigma - \half + \theta}\right) \epsfac x^{\sigma - \half} s^{-1}\bigg\vert_{s = \half + iT}
    \\
    \label{eq:t2ns_contrib_step1_et_eval1}
    &+ \frac{\qrep^\half D^\half}{\#\cF} (QD^r)^{\half - \sigma} \mum^\half \left(1 + \mum^{\sigma - \half + \theta}\right) \epsfac x^{\sigma - \half} s^{-1}\bigg\vert_{s = c + iT}
    .
  \end{align}
  The contribution of \eqref{eq:t2ns_contrib_step1_et_eval1_redundant} will be subsumed by \eqref{eq:t2ns_contrib_step1_mt}'s.

  When $c \leqslant \thalf$,
  \begin{align}
    \nonumber
    \eqref{eq:t2ns_contrib_step1_et_eval1}
    &\ll
    q^{\frac{5}{4} - c} D^{\frac{5}{4} - c - \delta} T^{-\frac{1}{4} - c} (qD)^{c - \half} (qDT)^\eps
    \\
    \label{eq:t2ns_contrib_step1_et_eval2}
    &\ll
    q^{\frac{3}{4}} D^{\frac{3}{4} - \delta - (\frac{1}{4} + c)\hat\gamma} (qD)^\eps
    .
  \end{align}
  The horizontal segments' contribution \eqref{eq:t2ns_contrib_step1_et_eval2} conforms in the $\qrep$ aspect to the error term of \cref{lemma:t2ns_contrib_r1}.
  The exponent of $D$ is
  \begin{align}
    \label{eq:t2ns_contrib_step1_et_eval3}
    \frac{3}{4} - \delta - \left(\frac{1}{4} + c\right)\hat\gamma
    &=
    \frac{-24c\delta + 4c - 38\delta + 25}{32}
    .
  \end{align}
  The supremum of \eqref{eq:t2ns_contrib_step1_et_eval3} on $0 < c \leqslant \thalf$ and $\tfrac{5}{6} < \delta < 1$ is $-\tfrac{5}{24} < -\tfrac{1}{14} \leqslant \hat\rho$ (at $c = 0, \delta = \tfrac{5}{6}$). This concludes the task of bounding the horizontal segments \eqref{eq:t2ns_contrib_step1_et}.

  Applying \cref{lemma:t2ns_size} to \eqref{eq:t2ns_contrib_step1_mt},
  \begin{align}
    \nonumber
    \eqref{eq:t2ns_contrib_step1_mt}
    &=
    \frac{1}{2\pi i}\int_{\half - iT}^{\half + iT} \Gfac \Favg \nsqsum \frac{a_{\bar\rep}(n)\chi_d(n)}{n^{1-s}} \Vm \xd^{\!s-\half} \,\frac{ds}{s}
    \\
    \nonumber
    &\ll \int_{\half - iT}^{\half + iT} \frac{\qrep^\half D^\half}{\#\cF} (QD^r)^{\half - \sigma} \mum^\half \left(1 + \mum^{\sigma - \half + \theta}\right) \epsfac x^{\sigma - \half}\,\frac{ds}{s}
    \\
    \nonumber
    &\ll \qrep^\half D^{\frac{3}{4}-\delta} (qD)^\eps \int_{\half - iT}^{\half + iT}  \xi^{-\half} Q^{\frac{1}{4}}  \frac{ds}{s}    
    \\
    \nonumber
    &\ll q^{\frac{3}{4}} D^{\frac{3}{4}-\delta} (qD)^\eps \int_{\half - iT}^{\half + iT}  |s - s_0|^{-\beta} |s - \kappa|^{\frac{1}{4}}  \frac{ds}{s}
    \\
    \label{eq:t2ns_contrib_step2_mt}
    &\ll q^{\frac{3}{4}} D^{\frac{3}{4}-\delta} \left(T^{\frac{1}{4} - \beta} + \tau^{\frac{1}{4}} |s_0|^{-\beta}\right) (qDT)^\eps 
    .
  \end{align}
  The first term of \eqref{eq:t2ns_contrib_step2_mt} is $\qrep^{\frac{3}{4}}$ in the $\qrep$ aspect, which is admissible, and
  \begin{align}
    \nonumber
    \ll
    D^{\frac{3}{4}-\delta + (\frac{1}{4} - \beta)\hat\gamma} \asymp D^{-\frac{3}{4}(\delta - \frac{5}{6})} \ll D^{\hat\rho}
  \end{align}
  in the $D$ aspect. The second term of \eqref{eq:t2ns_contrib_step2_mt} in the $D$ aspect is
  \begin{align}
    \label{eq:t2ns_contrib_step2_mt_eval}
    &\ll
    D^{\frac{3}{4}-\delta} \tau^{\frac{1}{4}}
    \ll
    D^{\frac{3}{4}-\delta + \frac{1}{4}(1 - \delta)}
    \ll
    D^{-\frac{3}{4}(\delta - \frac{5}{6})}
    \ll
    D^{\hat\rho}
  \end{align}
  for $\tfrac{5}{6} < \delta < 1$.
\end{proof}

\begin{lemma}[Term 1 square contribution]
  \label{lemma:t1s_contrib_r1}
  If
  $0 < c \leqslant \thalf$,
  $r = 1$,
  $x \asymp (\qrep D)^{1 \pm \eps}$,
  $T = D^{\hat{\gamma}}$,
  $\#\cF = D^{\delta}$,
  and
  $\beta = \hat{\beta}$,
  \begin{align*}
    &\frac{1}{2\pi i}\int_{c - iT}^{c + iT} \Favg \sum_{n=1}^\infty \frac{a_{\rep}(n^2)\chi_d(n^2)}{n^{2s}} \Vpsq \x^{s-\half} \frac{ds}{s}
    \,\ll\, D^{\hat{\rho}} \qrep^{\frac{3}{4}} \epsfac 
    .
  \end{align*}
\end{lemma}

\hypertarget{proof:t1s_contrib_r1}{}
\begin{proof}
  Shift the line of integration to $\sigma = \eps$. The horizontal segments are
  \begin{align*}
    \frac{1}{2\pi i}
    &\left(\int_{c - iT}^{\eps - iT} + \int_{\eps + iT}^{c + iT}\right)\Favg \sum_{n=1}^\infty \frac{a_{\rep}(n^2)\chi_d(n^2)}{n^{2s}} \Vpsq \x^{s-\half} \frac{ds}{s}
    \\
    &\ll
    \int_{\eps + iT}^{c + iT}  \left(1 + \mup^{\half - \sigma + \theta}\right) \epsfac \x^{\sigma-\half} \frac{ds}{s} \qquad\qquad\text{(by \cref{lemma:t1s_size})}
    \\
    &\ll
    (\xi\sqrt{QD})^\half (qD)^{-\half} T^{-1} \epsfac + \epsfac T^{-1}
    \\
    &\ll \xi^\half (qD)^{-\frac{1}{4}} T^{-\frac{3}{4}} \epsfac + T^{-1}\epsfac
    .
  \end{align*}
  This satisfies the bound of \cref{lemma:t1s_contrib_r1} in the $\qrep$ aspect. The $T^{-1}$ term satisfies the bound in the $D$ aspect. The first term is smaller than the integral along the vertical segment, which we'll bound momentarily, by at least a factor of $T$, so we may ignore it for our purposes.

  Using \cref{lemma:t1s_size} for the integral along the vertical segment,
  \begin{align*}
    \frac{1}{2\pi i}
    \int_{\eps - iT}^{\eps + iT}
    &\Favg \sum_{n=1}^\infty \frac{a_{\rep}(n^2)\chi_d(n^2)}{n^{2s}} \Vpsq \x^{s-\half} \frac{ds}{s}
    \\
    &\ll
    \frac{1}{2\pi i}
    \int_{\eps - iT}^{\eps + iT} \left(1 + \mup^{\half - \sigma + \theta}\right) \epsfac \x^{\sigma-\half} \frac{ds}{s}
    \\
    &\ll (qD)^{-\frac{1}{4}} (T^{\frac{1}{4} + \hat\beta} + |s_0|^{\hat\beta} \tau^{\frac{1}{4}}) \epsfac
  \end{align*}

  The exponent of $D$ in the first term is
  \begin{align}
    \label{eq:t1s_contrib_mt_eval}
    -\frac{1}{4} + \left(\frac{1}{4} + \hat\beta\right)\hat\gamma
    &=
    \frac{1}{8}\delta - \frac{3}{16}
    \leqslant
    \hat\rho
  \end{align}
  (with equality iff $\delta = \tfrac{13}{14}$).
\end{proof}

We write
\begin{align}
  \label{eq:sstardef}
  |s_*| \coloneqq \max_j\!\left\{|s + \kappa_j|, |1 - s + \bar\kappa_j|, 1\right\}
  ,
\end{align}
useful for abbreviating error terms when $|s| \asymp 1$.

\begin{lemma}[Term 2 square contribution]
  \label{lemma:t2s_contrib_r1}
  If
  $0 < c \leqslant \thalf$,
  $r = 1$,
  $x \asymp (\qrep D)^{1 \pm \eps}$,
  $T = D^{\hat{\gamma}}$,
  $\#\cF = D^{\delta}$,
  and
  $\beta = \hat{\beta}$,
  \begin{align*}
    \frac{1}{2\pi i}
    &\int_{c - iT}^{c + iT}  \Gfac \Favg \sum_{n=1}^\infty \frac{\rootnum_{\rep\otimes\chi_d} a_{\bar\rep}(n^2)\chi_d(n^2)}{n^{2-2s}} \Vmsq \xd^{\!s-\half} \frac{ds}{s}
    \\
    &= \frac{\rootnum_{\cF}}{2\pi i}\int_{\eps - iT}^{\eps + iT}
    \frac{\halfGamma{1-s+\bar\kappa}}{\halfGamma{s+\kappa}}
    \frac{L(2-2s, \barrepsq)}{L^{(2\qrep)}(3-2s, \barrepsq)}
    \mspace{1mu}
    \prod_{p \nmid 2\qrep}
    \left(1 - \frac{1}{(p+1)(1 - \bar\alpha_1(p)^{-2}p^{3-2s})}\right)
    \\
    &\hspace{8cm}\cdot
    \frac{1 - \left(1 - \frac{\Delta D}{D}\right)^{1 - (s-\half)}}{(1 - (s-\half))\frac{\Delta D}{D}}
    \xD^{\!s-\half} \frac{ds}{s}
    \\
    &+ O(D^{\hat{\rho}} \qrep^{\frac{3}{4}} + |\tau|D^{\delta - 1})\epsfac
    .
  \end{align*}
\end{lemma}
\hypertarget{proof:t2s_contrib_r1}{}
\begin{proof}
  
  Let $I_0, \dots, I_6$ be the line segments
  \begin{alignat*}{11}
    I_0:&\;\; &c\mspace{2mu} - iT &\longrightarrow& \,c\mspace{2mu} + iT&
    &\qquad\qquad
    I_4:&\;\; &\thalf - iT &\longrightarrow& \,\eps - iT&
    \\
    I_1:&\;\; &c\mspace{2mu} - iT &\longrightarrow& \,\thalf - iT&
    &
    I_5:&\;\; &\eps - iT &\longrightarrow& \,\eps + iT&
    \\
    I_2:&\;\; &\thalf - iT &\longrightarrow& \,\thalf + iT&
    &
    I_6:&\;\; &\eps + iT &\longrightarrow& \,c\mspace{2mu} + iT&
    \\
    I_3:&\;\; &\thalf + iT &\longrightarrow& \,c\mspace{2mu} + iT&
    &
    &&&&&
    .
  \end{alignat*}
  Let
  \begin{align*}
    F_0 &\coloneqq \frac{\halfGamma{1-s+\bar\kappa}}{\halfGamma{s+\kappa}} \Favg \sum_{n=1}^\infty \frac{\rootnum_{\rep\otimes\chi_d} a_{\bar\rep}(n^2)\chi_d(n^2)}{n^{2-2s}} V_{1-s}^*\!\left(\frac{n^2\xi}{\sqrt{\qrep d}}\right) \left(\frac{\pi\x}{\qrep d}\right)^{\!s-\half}
    \\
    F_1 &\coloneqq \frac{\halfGamma{1-s+\bar\kappa}}{\halfGamma{s+\kappa}} \Favgrn \sum_{n=1}^\infty \frac{a_{\bar\rep}(n^2)\chi_d(n^2)}{n^{2-2s}} V_{1-s}^*\!\left(\frac{n^2\xi}{\sqrt{\qrep D}}\right) \left(\frac{\pi\x}{\qrep d}\right)^{\!s-\half}
    \\
    F_2 &\coloneqq \frac{\halfGamma{1-s+\bar\kappa}}{\halfGamma{s+\kappa}} \sum_{n=1}^\infty \frac{a_{\bar\rep}(n^2)}{n^{2-2s}} V_{1-s}^*\!\left(\frac{n^2\xi}{\sqrt{\qrep D}}\right) \prod_{p\mid 2\qrep n} \frac{p}{p+1} \frac{6\rootnum_{\cF} \Eta_{\cF} \Delta D}{\pi^2 \phi(\qrep) \#\cF} \frac{1 - \left(1 - \frac{\Delta D}{D}\right)^{1 - (s-\half)}}{(1 - (s-\half))\frac{\Delta D}{D}} \left(\frac{\pi\x}{\qrep D}\right)^{\!s-\half}
    \\
    F_3 &\coloneqq \rootnum_{\cF}\frac{\halfGamma{1-s+\bar\kappa}}{\halfGamma{s+\kappa}} \sum_{n=1}^\infty \frac{a_{\bar\rep}(n^2)}{n^{2-2s}} V_{1-s}^*\!\left(\frac{n^2\xi}{\sqrt{\qrep D}}\right) \prod_{p\mid \frac{n}{(n,2q)}} \frac{p}{p+1} \frac{1 - \left(1 - \frac{\Delta D}{D}\right)^{1 - (s-\half)}}{(1 - (s-\half))\frac{\Delta D}{D}} \left(\frac{\pi\x}{\qrep D}\right)^{\!s-\half}
    \\
    F_4 &\coloneqq
    \rootnum_{\cF}
    \frac{\halfGamma{1-s+\bar\kappa}}{\halfGamma{s+\kappa}}
    \frac{L(2-2s, \barrepsq)}{L^{(2\qrep)}(3-2s, \barrepsq)}
    \mspace{1mu}
    \prod_{p \nmid 2\qrep} \left(1 - \frac{1}{(p+1)(1 - \bar\alpha_j(p)^{-2}p^{3-2s})}\right)
    \frac{1 - \left(1 - \frac{\Delta D}{D}\right)^{1 - (s-\half)}}{(1 - (s-\half))\frac{\Delta D}{D}} \left(\frac{\pi\x}{\qrep D}\right)^{\!s-\half}
    \shortintertext{and}
    R_0 &\coloneqq |s_*|^{\half - \sigma} \left(1 + \mumd^{\sigma - \half}\right) \epsfac
    \\
    R_1 &\coloneqq |s_*|^{\half - \sigma} \frac{\#\cF}{D} \left(1 + \mumd^{\sigma - \half}\right) \epsfac \max_{d \in \cF} \left(\frac{\pi\x}{\qrep d}\right)^{\!\sigma-\half}
    \\
    R_2 &\coloneqq
    |s_*|^{\half - \sigma} \frac{\qrep^{\frac{1}{6} }D^\half}{\#\cF} \left(1 + \mumd^{\sigma - \half}\right)
    \left(D^{\frac{1}{14}} + |s_*|^{\frac{1}{6}}\right)
    \\
    R_3 &\coloneqq
    |s_*|^{\half - \sigma} \frac{D^{\half}}{\#\cF}\qet \left(1 + \mumd^{\sigma - \half}\right) \min\!\left\{ 1, \frac{D}{|s|\Delta D} \right\} \epsfac \max_{d \in \cF} \left(\frac{\pi\x}{\qrep d}\right)^{\!\sigma-\half}
    \\
    R_4 &\coloneqq |s_*|^{\half - \sigma} \mumd^{\sigma - \half} \min\!\left\{ 1, \frac{D}{|s|\Delta D} \right\} \epsfac
    .
  \end{align*}
  The left hand side of \cref{lemma:t2s_contrib_r1} (times $2\pi i$) can be written
  \begin{align*}
    &\int_{I_0} F_0\,\frac{ds}{s}
    \\
    &= \int_{I_1 + I_2 + I_3} F_0\,\frac{ds}{s}
    \\
    &=
    \int_{I_1 + I_2 + I_3} F_1 + O(R_1)\,\frac{ds}{s} \qquad\text{(by \cref{lemma:t2s_approx_step1})}
    \\
    &= \int_{I_1 + I_2 + I_3} F_2 + O(R_1 + R_2)\,\frac{ds}{s} \qquad\text{(by \cref{lemma:t2s_approx_step2})}
    \\
    &=
    \int_{I_1 + I_2 + I_3} O(R_1 + R_2)\,\frac{ds}{s}
    \,+\,
    \int_{I_1 + I_4 + I_5 + I_6 + I_3} F_2\,\frac{ds}{s}
    \\
    &=
    \int_{I_1 + I_2 + I_3} O(R_1 + R_2)\,\frac{ds}{s}
    \,+\,
    \int_{I_1 + I_4 + I_6 + I_3} O(R_0)\,\frac{ds}{s}
    \,+\,
    \int_{I_5} F_2\,\frac{ds}{s} \qquad\text{(by \cref{lemma:t2s_size})}
    \\
    &=
    \int_{I_1 + I_2 + I_3} O(R_1 + R_2)\,\frac{ds}{s}
    \,+\,
    \int_{I_1 + I_4 + I_6 + I_3} O(R_0)\,\frac{ds}{s}
    \,+\,
    \int_{I_5} F_3 + O(R_3) \,\frac{ds}{s} \qquad\text{(by \cref{lemma:t2s_approx_step3})}
    \\
    &=
    \int_{I_1 + I_2 + I_3} O(R_1 + R_2)\,\frac{ds}{s}
    \,+\,
    \int_{I_1 + I_4 + I_6 + I_3} O(R_0)\,\frac{ds}{s}
    \,+\,
    \int_{I_5} F_4 + O(R_3 + R_4) \,\frac{ds}{s} \qquad\text{(by \cref{lemma:t2s_approx_step4})}
    .
  \end{align*}
  \Cref{lemma:t2s_contrib_r1} is then straightforward to verify.
\end{proof}

We introduce \cref{lemma:eliminate_derivative_factor} so that we may remove the factor of
$$\frac{1 - \left(1 - \frac{\Delta D}{D}\right)^{1 - (s-\half)}}{(1 - (s-\half))}$$
appearing in the integrand of the right hand side of \cref{lemma:t2s_contrib_r1}, and so that we may extend the line of integration to run from $c - i\infty$ to $c + i\infty$, as opposed to being truncated at $|\Im(s)| < T$.

\begin{lemma}
  \label{lemma:eliminate_derivative_factor}
  Suppose that
  $\#\cF = D^{\delta}$
  and
  $T^{1-\eps} \gg D^{1-\delta} \gg \tau^{1+\eps}$.
  Let $\subconvexityparam$ be such that
  $L(\thalf + it, \bar\chi^2) \ll (\qrep |t|)^{\subconvexityparam + \eps}$.
  Then
  \begin{align*}
    &\frac{1}{2\pi i}\int_{\frac{3}{4} - iT}^{\frac{3}{4} + iT}
    \frac{\halfGamma{1-s+\bar\kappa}}{\halfGamma{s+\kappa}}
    \frac{L(2-2s, \barrepsq)}{L^{(2\qrep)}(3-2s, \barrepsq)}
    \mspace{1mu}
    \prod_{p \nmid 2\qrep}
    \left(1 - \frac{1}{(p+1)(1 - \bar\alpha_1(p)^{-2}p^{3-2s})}\right)
    \\
    &\hspace{8cm}\cdot
    \frac{1 - \left(1 - \frac{\Delta D}{D}\right)^{1 - (s-\half)}}{(1 - (s-\half))\frac{\Delta D}{D}}
    \xD^{\!s-\half} \frac{ds}{s}
    \\
    &=
    \frac{1}{2\pi i}\int_{\frac{3}{4} - i\infty}^{\frac{3}{4} + i\infty}
    \frac{\halfGamma{1-s+\bar\kappa}}{\halfGamma{s+\kappa}}
    \frac{L(2-2s, \barrepsq)}{L^{(2\qrep)}(3-2s, \barrepsq)}
    \mspace{1mu}
    \prod_{p \nmid 2\qrep}
    \left(1 - \frac{1}{(p+1)(1 - \bar\alpha_1(p)^{-2}p^{3-2s})}\right)
    \xD^{\!s-\half} \frac{ds}{s}
    \\&
    \quad
    + O(\qrep^{\subconvexityparam+\eps} D^{(\subconvexityparam - \frac{1}{4})(1-\delta) + \eps})
    .
  \end{align*}
\end{lemma}
\hyperlink{proof:eliminate_derivative_factor}{Click here to jump to the proof of \cref*{lemma:eliminate_derivative_factor}.} The approach is to truncate the integral to some intermediate height, either approximate the spurious factor $(1 - \left(1 - \tfrac{\Delta D}{D}\right)^{1 - (s-\half)})/((1 - (s-\half))\tfrac{\Delta D}{D})$ by $1$ when $|s|$ is small or bound it when $|s|$ is big, and then extend the line of integration to infinity after the spurious factor is removed. Three error terms are created by these three steps, and they are found to balance when the point of truncation is taken to be on the order of $D^{1-\delta}$.

  

We are now prepared to prove \cref{thm:quadratic_murmurations}.

\hypertarget{proof:quadratic_murmurations}{}
\begin{proof}[Proof of \cref{thm:quadratic_murmurations}]
  
  By Perron's formula \eqref{eq:perron},
  \begin{align}
    \nonumber
    \frac{1}{\#\cF} \sum_{d \in \cF} \frac{1}{\sqrt{\x}} \sum_{n < \x} a_\rep(n)\chi_d(n)
    &=
    \frac{1}{2\pi i} \int_{1 + \eps - iT}^{1 + \eps + iT} \Favg L(s,\rep\otimes\chi_d)\,\x^{s-\half}\,\frac{ds}{s}
    \\
    \label{eq:thm1proof_perron_et}
    &+
    O(\x^{\half}T^{-1} (\x Tq)^\eps)
    .
  \end{align}
  
  The construction of $\cF$ ensures that each $L(s,\rep\otimes\chi_d)$ is entire. Shifting the line of integration to $\Re(s) = \thalf$,
  \begin{align}
    \nonumber
    &\frac{1}{2\pi i} \int_{1 + \eps - iT}^{1 + \eps + iT} \Favg L(s,\rep\otimes\chi_d)\,\x^{s-\half}\,\frac{ds}{s}
    \\
    \label{eq:thm1proof_shift_mt}
    &=
    \frac{1}{2\pi i} \int_{\half - iT}^{\half + iT} \Favg L(s,\rep\otimes\chi_d)\,\x^{s-\half}\,\frac{ds}{s}
    \\
    \label{eq:thm1proof_shift_et}
    &+
    \frac{1}{2\pi i} \left(\int_{1 + \eps - iT}^{\half - iT} + \int_{\half + iT}^{1 + \eps + iT}\right) \Favg L(s,\rep\otimes\chi_d)\,\x^{s-\half}\,\frac{ds}{s}
    .
  \end{align}

  We now bound the contribution from the horizontal segments \eqref{eq:thm1proof_shift_et}. Using the subconvexity bound \cite{petrow_young:2023} and the Phragm\'en--Lindel\"of principle \cite[Thm.\ 8.2.1]{goldfeld:book}, 
  \begin{align*}
    \frac{1}{2\pi i}\left(\int_{1 + \eps - iT}^{\half - iT} + \int_{\half + iT}^{1 + \eps + iT} \right) \frac{1}{\#\cF} \sum_{d \in \cF} L(s,\rep \otimes \chi_d) \,\x^{s-\half} \,\frac{ds}{s}
    &\ll
    \int_\half^{1+\eps} \left(\qrep DT\right)^{\frac{1}{3}(1 - \sigma) + \eps} \x^{\sigma - \half} T^{-1}\,d\sigma 
    \\
    &\ll
    \left(\qrep DT\right)^{\frac{1}{6} + \eps} T^{-1} + \x^{\half+\eps} T^{-1}
    .
  \end{align*}
  When $x \asymp (\qrep D)^{1 \pm \eps}$ and $T = D^{\hat{\gamma}}$, 
  \begin{align*}
    \left(\qrep DT\right)^{\!\frac{1}{6} + \eps} T^{-1} + \x^{\half+\eps} T^{-1}
    \ll
    \qrep^{\half + \eps} D^{\hat{\rho} + \eps}
    .
  \end{align*}

  The other error term \eqref{eq:thm1proof_perron_et},
  \begin{align*}
    O\big(\x^{\half}T^{-1}(\x T q)^\eps\big)
    ,
  \end{align*}
  which also comes from Perron's formula \eqref{eq:perron}, is asymptotically bounded by $\qrep^{\half + \eps} D^{\hat{\rho} + \eps}$.
  
  Having bounded all error terms which have arisen thus far, we continue working on the main term \eqref{eq:thm1proof_shift_mt}.
  Substituting the approximate functional equation into \eqref{eq:thm1proof_shift_mt}, we obtain the decomposition
  \begin{align}
    \nonumber
    \frac{1}{2\pi i} \int_{\half - iT}^{\half + iT} \Favg L(s,\rep\otimes\chi_d)\,\x^{s-\half}\,\frac{ds}{s}
    &=
    \eqref{eq:afe_unweighted_t1ns_re} + \eqref{eq:afe_unweighted_t2ns_re} + \eqref{eq:afe_unweighted_t1s_re} + \eqref{eq:afe_unweighted_t2s_re}
    ,
  \end{align}
  where we have further split each of the two terms in the approximate functional equation into sums over squares \eqref{eq:afe_unweighted_t1s_re} \eqref{eq:afe_unweighted_t2s_re} and sums over nonsquares \eqref{eq:afe_unweighted_t1ns_re} \eqref{eq:afe_unweighted_t2ns_re}.

  The three truncated inverse Mellin transforms \eqref{eq:afe_unweighted_t1ns_re}, \eqref{eq:afe_unweighted_t2ns_re}, and \eqref{eq:afe_unweighted_t1s_re} are bounded by \cref{lemma:t1ns_contrib_r1,lemma:t2ns_contrib_r1,lemma:t1s_contrib_r1} for $x \asymp (qD)^{1\pm\eps}$, $T = D^{\hat{\gamma}}$, $\#\cF = D^{\delta}$, $\beta = \hat{\beta}$. In each case they conform to the error term $O(D^{\frac{1}{12}\hat\rho+\eps}q^{\frac{1}{6} + \eps} + D^{\hat{\rho} + \eps} \qrep^{\frac{3}{4}}$ of \cref{thm:quadratic_murmurations}.
  
  The integral \eqref{eq:afe_unweighted_t2s_re} from the squares in the second term of the approximate functional equation ultimately contributes the main term of \cref{thm:quadratic_murmurations}.
  To analyze \eqref{eq:afe_unweighted_t2s_re}, we apply \cref{lemma:t2s_contrib_r1}, yielding the estimate
  \begin{align}
    \label{eq:thm1proof_L2approx}
    \eqref{eq:afe_unweighted_t2s_re}
    &=
    \frac{\rootnum_{\cF}}{2\pi i}\int_{\eps - iT}^{\eps + iT}
    \frac{\halfGamma{1-s+\bar\kappa}}{\halfGamma{s+\kappa}}
    \frac{L(2-2s, \barrepsq)}{L^{(2\qrep)}(3-2s, \barrepsq)}
    \mspace{1mu}
    \prod_{p \nmid 2\qrep} \left(1 - \frac{1}{(p+1)(1 - \bar\alpha_1(p)^{-2}p^{3-2s})}\right)
    \\
    \nonumber
    &\hspace{8cm}\cdot
    \frac{1 - \left(1 - \frac{\Delta D}{D}\right)^{1 - (s-\half)}}{(1 - (s-\half))\frac{\Delta D}{D}}
    \left(\frac{\pi\x}{\qrep D}\right)^{\!s-\half}
    \frac{ds}{s}
    \\
    \nonumber
    &+ O(D^{\hat{\rho}} \qrep^{\frac{3}{4}} \epsfac )
    .
  \end{align}

  We now shift the contour of integration of the right hand side of \eqref{eq:thm1proof_L2approx} to $\Re(s) = 3/4$, so that we may apply \cref{lemma:eliminate_derivative_factor}. Before continuing with the main term, let's assess the consequences of this shift.

  Recall that we've written $\rep \eqqcolon |\mspace{0mu}\cdot\mspace{0mu}|^{i\tau}\chi$. If $\bar\chi^2$ is nontrivial, then
  the integrand \eqref{eq:thm1proof_L2approx} has no poles in the region $\eps - iT \longrightarrow \tfrac{3}{4} - iT \longrightarrow \tfrac{3}{4} + iT  \longrightarrow \eps + iT \longrightarrow \eps - iT$.
  If $\bar\chi^2$ is trivial, then, using our assumption that $|\tau| < T$, the integrand \eqref{eq:afe_unweighted_t2s_re} has a pole at $s = \thalf + i\tau$ with residue
  \begin{align}
    \label{eq:thm1proof_residue}
    -\rootnum_{\cF}
    \frac{6}{\pi^2}\frac{1 + \frac{1}{5}\one{2\nmid q}}{1 + 2i\tau}
    \left(\frac{\pi \x}{q D}\right)^{\!i\tau}
    \prod_{p\mid q}\frac{p}{p+1}
    \prod_{\mspace{1.75mu}p \nmid q} \!\left(1 + \frac{1}{(p + 1)(p^2 - 1)}\right)
    \!.
  \end{align}

  We now bound the integral of \eqref{eq:thm1proof_L2approx} along the horizontal segments $\eps - iT \longrightarrow \tfrac{3}{4} - iT$ and $\tfrac{3}{4} + iT  \longrightarrow \eps + iT$
  For $\sigma < \thalf$, we have $L(2-2s,\barrepsq) \ll 1$. Combining with
  \begin{align}
    \nonumber
    \frac{\halfGamma{1-s+\bar\kappa}}{\halfGamma{s+\kappa}} \ll |s|^{\half - \sigma}
  \end{align}
  and
  \begin{align}
    \nonumber
    \frac{1 - \left(1 - \frac{\Delta D}{D}\right)^{1 - (s-\half)}}{(1 - (s-\half))\frac{\Delta D}{D}} \left(\frac{\pi\x}{\qrep D}\right)^{\!s - \half} \ll \frac{D}{|s|\Delta D}
  \end{align}
  we find that, for $x \asymp (\qrep D)^{1 \pm \eps}$, $T = D^{\hat{\gamma}}$, $\#\cF = D^\delta$, and $\beta = \hat{\beta}$, the contribution from the horizontal segments $\eps \pm iT \longrightarrow \thalf - \eps \pm iT$ is $\ll D^{1 - \delta - \frac{3}{2}\hat{\gamma} + \eps} \ll D^{\hat\rho}$.

  For $\thalf < \sigma < \tfrac{3}{4}$, we use e.g.\ the convexity bound \cite[Thm. 5.23]{IK} to bound the contribution of the horizontal integral by $D^{1 - \delta - \hat{\gamma} + \eps} q^{\frac{1}{4} + \eps}$; the product appearing in the integrand of \eqref{eq:afe_unweighted_t2s_re} converges absolutely (\cref{lemma:prod_convergence}).
  

  We continue now with the main term 
  \begin{align*}
    &\frac{\rootnum_{\cF}}{2\pi i}\int_{\frac{3}{4} - iT}^{\frac{3}{4} + iT}
    \frac{\halfGamma{1-s+\bar\kappa}}{\halfGamma{s+\kappa}}
    \frac{L(2-2s, \barrepsq)}{L^{(2\qrep)}(3-2s, \barrepsq)}
    \mspace{1mu}
    \prod_{p \nmid 2\qrep} \left(1 - \frac{1}{(p+1)(1 - \bar\alpha_1(p)^{-2}p^{3-2s})}\right)
    \\
    &\hspace{8cm}\cdot
    \frac{1 - \left(1 - \frac{\Delta D}{D}\right)^{1 - (s-\half)}}{(1 - (s-\half))\frac{\Delta D}{D}}
    \left(\frac{\pi\x}{\qrep D}\right)^{\!s - \half} \frac{ds}{s}
    .
  \end{align*}
  Applying \cref{lemma:eliminate_derivative_factor} with $\lambda = 1/6$ \cite{petrow_young:2023} and collecting the residue \eqref{eq:thm1proof_residue} and the error terms, all of which we have noted to be admissible over the course of the proof, proves \cref{thm:quadratic_murmurations}.
\end{proof}

\subsection{Proof of \cref{thm:gamma_murmurations}}
\label{sec:contrib_weighted}

The proof of \cref{thm:gamma_murmurations} proceeds very much along the same lines of the proof of \cref{thm:quadratic_murmurations} from \cref{sec:contrib_unweighted}. It is, however, much simpler, thanks to the weight $f$ featuring in \cref{thm:gamma_murmurations}.

We ask that $f(\sigma + it) \ll |t|^{-2 - \eps}$ for two reasons. The first of these reasons is to ensure that tails of integrals which appear are negligible, in turn allowing us to extend the vertical line segments we integrate along to infinity. The largest any of the integrands appearing over the course of the proof can be in the $t$ aspect is $|t|^{\frac{2}{3} + \eps}$: the error term $R_2$ in the \hyperlink{proof:t2s_contrib_r1}{proof of \cref*{lemma:t2s_contrib_r1}} at $\sigma = \eps$ is of this size (in what follows we'll take $\beta = 0$). Imposing $f(\sigma + it) \ll |t|^{-2 - \eps}$ ensures, with room to spare, that every integrand decays quickly enough for all of the contribution to come from regions $|s| \ll 1$.

Second, we require $f(\sigma + it) \ll |t|^{-2 - \eps}$ to obviate the need for an equivalent of \cref{lemma:eliminate_derivative_factor}. The more quickly $f$ decays, the greater the extent to which the mass of the integral is concentrated near the real axis. When $|s|$ is sufficiently small, we'll have
\begin{align}
  \label{eq:deriv_approx}
  \frac{1 - \left(1 - \frac{\Delta D}{D}\right)^{1 - (s-\half)}}{(s - \half)\frac{\Delta D}{D}} \approx 1
\end{align}
by a first order series expansion. If the contribution to our integrals from outside regions where \eqref{eq:deriv_approx} is valid is appropriately small, then we have effectively removed the factor of the left hand side of \eqref{eq:deriv_approx} from our integrands. This procedure happens in \eqref{eq:t2s_contrib_smoothed_step2}. In many applications, requiring that $f$ decay at this rather modest rate is somewhat of a non-issue, whereas incurring penalty to the error term from an equivalent of \cref{lemma:eliminate_derivative_factor} would be regrettable.

We introduce a new target for the $D$ aspect of the error term:
\begin{align}
  \label{eq:rhofhatdef}
  \hat{\rho}_f \coloneqq -\frac{1}{8} + \left|\mspace{1mu}\delta - \frac{7}{8}\right| = \max\!\left\{\frac{3}{4} - \delta,\, \delta - 1\right\}.
\end{align}
Our target in the $\qrep$ aspect, $\qrep^{\frac{3}{4}}$,
remains unchanged.

\begin{lemma}[Smoothed term 1 non-square contribution]
  \label{lemma:t1ns_contrib_r1_smoothed}
  If
  $0 < c < 1$,
  $r = 1$,
  $x \asymp (\qrep D)^{1 \pm \eps}$,
  and
  $\#\cF = D^{\delta}$,
  \begin{align*}
    &\frac{1}{2\pi i}\int_{c - i\infty}^{c + i\infty} \Favg \nsqsum \frac{a_\rep(n)\chi_d(n)}{n^s} \Vp \x^{s-\half} f(s)\mspace{1.5mu}ds
    \,\ll\, D^{\hat{\rho}_f} \qrep^{\frac{3}{4}} \epsfac
    .
  \end{align*}
\end{lemma}

\hypertarget{proof:t1ns_contrib_r1_smoothed}{}
\begin{proof}
  The integrand is holomorphic for $0 < \sigma < 1$, so we may shift the line of integration to $\sigma = \eps$. Invoking \cref{lemma:t1ns_size},
  \begin{align}
    \nonumber
    \frac{1}{2\pi i}
    \int_{(\eps)}
    &\Favg \nsqsum \frac{a_\rep(n)\chi_d(n)}{n^s} \Vp x^{s-\half} f(s)\mspace{1.5mu}ds
    \\
    \nonumber
    &\ll \int_{(\eps)} \frac{\qrep^\half D^\half}{\#\cF} \big(\xi\sqrt{QD}\big)^{1 - \sigma} \epsfac x^{\sigma - \half} f(s)\mspace{1.5mu}ds
    \\
    \nonumber
    &\ll \frac{\tau^\half \qrep^\half D^\half}{\#\cF} (\tau qD)^\eps
    ,
  \end{align}
  where we have made use our earlier observations about the consequences of $f$'s rate of decay.
\end{proof}

\begin{lemma}[Smoothed term 2 non-square contribution]
  \label{lemma:t2ns_contrib_r1_smoothed}
  If
  $0 < c < 1$,
  $r = 1$,
  $x \asymp (\qrep D)^{1 \pm \eps}$,
  and
  $\#\cF = D^{\delta}$,
  \begin{align*}
    &\frac{1}{2\pi i}\int_{c - i\infty}^{c + i\infty} \Gfac \Favg \nsqsum \frac{\rootnum_{\rep\otimes\chi_d}a_{\bar\rep}(n)\chi_d(n)}{n^{1-s}} \Vm \xd^{\!s-\half} f(s)\mspace{1.5mu}ds
    \,\ll\, D^{\hat{\rho}_f} \qrep^{\frac{3}{4}} \epsfac
    .
  \end{align*}
\end{lemma}

\hypertarget{proof:t2ns_contrib_r1_smoothed}{}
\begin{proof}
  Shifting the line of integration to $\sigma = \thalf$ and applying \cref{lemma:t2ns_size},
  \begin{align*}
    &\frac{1}{2\pi i}\int_{(c)} \Gfac \Favg \nsqsum \frac{\rootnum_{\rep\otimes\chi_d} a_{\bar\rep}(n)\chi_d(n)}{n^{1-s}} \Vm \xd^{\!s-\half} f(s)\mspace{1.5mu}ds
    \\
    &\ll
    \int_{(\half)}
    \frac{\qrep^\half D^\half}{\#\cF} (QD^r)^{\half - \sigma} \mum^\half \left(1 + \mum^{\sigma - \half + \theta}\right) \epsfac x^{\sigma - \half} f(s)\mspace{1.5mu}ds
    \\
    &\ll
    \frac{\tau^{\frac{1}{4}} \qrep^{\frac{3}{4}} D^{\frac{3}{4}}}{\#\cF} (\tau \qrep D)^\eps
    .
  \end{align*}
  This is the term which constrains the $\tau$ aspect of the error term in \cref{thm:gamma_murmurations} when $\cF \ll D^{\frac{7}{8}}$.
\end{proof}

\begin{lemma}[Smoothed term 1 square contribution]
  \label{lemma:t1s_contrib_r1_smoothed}
  If
  $0 < c < 1$,
  $r = 1$,
  $x \asymp (\qrep D)^{1 \pm \eps}$,
  and
  $\#\cF = D^{\delta}$,
  \begin{align*}
    &\frac{1}{2\pi i}\int_{c - i\infty}^{c + i\infty} \Favg \sum_{n=1}^\infty \frac{a_{\rep}(n^2)\chi_d(n^2)}{n^{2s}} \Vpsq \x^{s-\half} f(s)\mspace{1.5mu}ds
    \,\ll\, D^{\hat{\rho}_f} \qrep^{\frac{3}{4}} \epsfac
    .
  \end{align*}
\end{lemma}

\hypertarget{proof:t1s_contrib_r1_smoothed}{}
\begin{proof}
  Shifting the line of integration to $\sigma = \eps$ and applying \cref{lemma:t1s_size},
  \begin{align*}
    \frac{1}{2\pi i}
    &\int_{(c)} \Favg \sum_{n=1}^\infty \frac{a_{\rep}(n^2)\chi_d(n^2)}{n^{2s}} \Vpsq \x^{s-\half} f(s)\mspace{1.5mu}ds
    \\
    &\ll
    \int_{(\eps)}  \left(1 + \mup^{\half - \sigma + \theta}\right) \epsfac \x^{\sigma-\half} f(s)\mspace{1.5mu}ds
    \\
    &\ll
    (\xi\sqrt{QD})^\half (\qrep D)^{-\half} \epsfac
    \\
    &\ll
    \tau^{\frac{1}{4}} (\qrep D)^{-\frac{1}{4}} (\tau\qrep D)^\eps
    .
    \qedhere
  \end{align*}
\end{proof}

\begin{lemma}[Smoothed term 2 square contribution]
  \label{lemma:t2s_contrib_r1_smoothed}
  If
  $0 < c < 1$,
  $r = 1$,
  $x \asymp (\qrep D)^{1 \pm \eps}$,
  and
  $\#\cF = D^{\delta}$,
  \begin{align*}
    &\frac{1}{2\pi i}\int_{c - i\infty}^{c + i\infty}  \Gfac \Favg \sum_{n=1}^\infty \frac{\rootnum_{\rep\otimes\chi_d} a_{\bar\rep}(n^2)\chi_d(n^2)}{n^{2-2s}} \Vmsq \xd^{\!s-\half} \frac{ds}{s}
    \\
    &= \frac{\rootnum_{\cF}}{2\pi i}\int_{\eps - i\infty}^{\eps + i\infty}
    \frac{\halfGamma{1-s+\bar\kappa}}{\halfGamma{s+\kappa}}
    \frac{L(2-2s, \barrepsq)}{L^{(2\qrep)}(3-2s, \barrepsq)}
    \mspace{1mu}
    \prod_{p \nmid 2\qrep} \left(1 - \frac{1}{(p+1)(1 - \bar\alpha_1(p)^{-2}p^{3-2s})}\right)
    \left(\frac{\pi\x}{\qrep D}\right)^{\!s-\half} f(s)\mspace{1.5mu}ds
    \\
    &+ O(D^{\hat{\rho}_f} \qrep^{\frac{3}{4}} + |\tau|D^{\delta - 1} )\epsfac
    .
  \end{align*}
\end{lemma}

\hypertarget{proof:t2s_contrib_r1_smoothed}{}
\begin{proof} 
  Let $F_0$ through $F_4$ and $R_0$ through $R_4$ be as in the \hyperlink{proof:t2s_contrib_r1}{proof of \cref*{lemma:t2s_contrib_r1}}. The left hand side of \cref{lemma:t2s_contrib_r1} (times $2\pi i$) can be written
  \begin{align}
    \nonumber
    \int_{(c)}
    F_0 \,f(s)\mspace{1.5mu}ds
    &= \int_{(\half)} F_0 \,f(s)\mspace{1.5mu}ds
    \\
    \nonumber
    &= \int_{(\half)} \big(F_1 + O(R_1)\big)\,f(s)\mspace{1.5mu}ds \qquad\text{(by \cref{lemma:t2s_approx_step1})}
    \\
    \nonumber
    &= \int_{(\half)} \big(F_2 + O(R_1 + R_2)\big)\,f(s)\mspace{1.5mu}ds \qquad\text{(by \cref{lemma:t2s_approx_step2})}
    \\
    \nonumber
    &= \int_{(\half)} \big(F_3 + O(R_1 + R_2 + R_3)\big)\,f(s)\mspace{1.5mu}ds \qquad\text{(by \cref{lemma:t2s_approx_step3})}
    \\
    \nonumber
    &= \int_{(\half)} O(R_1 + R_2 + R_3)\,f(s)\mspace{1.5mu}ds
    \,+\,
    \int_{(\eps)} F_3 \,f(s)\mspace{1.5mu}ds
    \\
    \label{eq:t2s_contrib_smoothed_step1}
    &= \int_{(\half)} O(R_1 + R_2 + R_3)\,f(s)\mspace{1.5mu}ds
    \,+\,
    \int_{(\eps)} O(R_4) \,f(s)\mspace{1.5mu}ds 
    \,+\,
    \int_{(\eps)} F_4 \,f(s)\mspace{1.5mu}ds \qquad\text{(by \cref{lemma:t2s_approx_step4})}
    .
  \end{align}
  
  Analyzing the third term on the right hand side of \eqref{eq:t2s_contrib_smoothed_step1},
  \begin{align}
    \nonumber
    &\int_{(\eps)}
    \frac{\halfGamma{1-s+\bar\kappa}}{\halfGamma{s+\kappa}}
    \frac{L(2-2s, \barrepsq)}{L^{(2\qrep)}(3-2s, \barrepsq)}
    \mspace{1mu}
    \prod_{p \nmid 2\qrep} \left(1 - \frac{1}{(p+1)(1 - \bar\alpha_j(p)^{-2}p^{3-2s})}\right)
    \\
    \nonumber
    &\hspace{8cm}\cdot
    \frac{1 - \left(1 - \frac{\Delta D}{D}\right)^{1 - (s-\half)}}{(1 - (s-\half))\frac{\Delta D}{D}} \left(\frac{\pi\x}{\qrep D}\right)^{\!s-\half}
    f(s)\mspace{1.5mu}ds
    \\
    \label{eq:t2s_contrib_smoothed_mt}
    &=
    \int_{(\eps)}
    \frac{\halfGamma{1-s+\bar\kappa}}{\halfGamma{s+\kappa}}
    \frac{L(2-2s, \barrepsq)}{L^{(2\qrep)}(3-2s, \barrepsq)}
    \mspace{1mu}
    \prod_{p \nmid 2\qrep} \left(1 - \frac{1}{(p+1)(1 - \bar\alpha_j(p)^{-2}p^{3-2s})}\right)
    \left(\frac{\pi\x}{\qrep D}\right)^{\!s-\half}
    f(s)\mspace{1.5mu}ds
    \\
    \label{eq:t2s_contrib_smoothed_step2}
    &\begin{aligned}
    &+
    \int_{(\eps)}
    \frac{\halfGamma{1-s+\bar\kappa}}{\halfGamma{s+\kappa}}
    \frac{L(2-2s, \barrepsq)}{L^{(2\qrep)}(3-2s, \barrepsq)}
    \mspace{1mu}
    \prod_{p \nmid 2\qrep} \left(1 - \frac{1}{(p+1)(1 - \bar\alpha_j(p)^{-2}p^{3-2s})}\right)
    \\
    &\hspace{8cm}
    \cdot
    \left[\frac{1 - \left(1 - \frac{\Delta D}{D}\right)^{1 - (s-\half)}}{(1 - (s-\half))\frac{\Delta D}{D}} - 1\right]
    \left(\frac{\pi\x}{\qrep D}\right)^{\!s-\half}
    f(s)\mspace{1.5mu}ds
    .
    \end{aligned}
  \end{align}
  If $s\Delta D \ll D$, then a linear approximation gives
  \begin{align}
    \label{eq:t2s_smoothed_approx1}
    \frac{1 - \left(1 - \frac{\Delta D}{D}\right)^{1 - (s-\half)}}{(1 - (s-\half))\frac{\Delta D}{D}}
    &=
    1 + O\!\left(\frac{|s_*|\Delta D}{D}\right)
    .
  \end{align}
  If $s\Delta D \gg D$, then instead bound the numerator as $\ll_\sigma 1$, so
  \begin{align}
    \label{eq:t2s_smoothed_approx2}
    \frac{1 - \left(1 - \frac{\Delta D}{D}\right)^{1 - (s-\half)}}{(1 - (s-\half))\frac{\Delta D}{D}}
    &\ll
    O\!\left(\frac{D}{|s_*|\Delta D}\right)
    .
  \end{align}
  Shifting \eqref{eq:t2s_contrib_smoothed_step2}'s line of integration to $\sigma = \thalf - \eps$ --- the integrand is holomorphic in the strip $0 < \sigma < \thalf$ --- and substituting \eqref{eq:t2s_smoothed_approx1} \eqref{eq:t2s_smoothed_approx2},
  \begin{align}
    \nonumber
    \eqref{eq:t2s_contrib_smoothed_step2}
    &= \int_{(\half - \eps)}
    \frac{\halfGamma{1-s+\bar\kappa}}{\halfGamma{s+\kappa}}
    \frac{L(2-2s, \barrepsq)}{L^{(2\qrep)}(3-2s, \barrepsq)}
    \mspace{1mu}
    \prod_{p \nmid 2\qrep} \left(1 - \frac{1}{(p+1)(1 - \bar\alpha_j(p)^{-2}p^{3-2s})}\right)
    \\
    \nonumber
    &\hspace{7cm}
    \cdot
    \left[\frac{1 - \left(1 - \frac{\Delta D}{D}\right)^{1 - (s-\half)}}{(1 - (s-\half))\frac{\Delta D}{D}} - 1\right]
    \left(\frac{\pi\x}{\qrep D}\right)^{\!s-\half}
    f(s)\mspace{1.5mu}ds
    \\
    \nonumber
    &\begin{aligned}
    \ll
    &\int_{\half - \eps - i\frac{D}{\Delta D}}^{\half - \eps + i\frac{D}{\Delta D}}
    \frac{\halfGamma{1-s+\bar\kappa}}{\halfGamma{s+\kappa}}
    \frac{L(2-2s, \barrepsq)}{L^{(2\qrep)}(3-2s, \barrepsq)}
    \frac{|s_*|\Delta D}{D}
    \left(\frac{\pi\x}{\qrep D}\right)^{\!s-\half}
    f(s)\mspace{1.5mu}ds
    \\
    &+
    \int_{\substack{\Re(s) = \half - \eps \\ |\Im(s)| > \frac{D}{\Delta D}}}
    \frac{\halfGamma{1-s+\bar\kappa}}{\halfGamma{s+\kappa}}
    \frac{L(2-2s, \barrepsq)}{L^{(2\qrep)}(3-2s, \barrepsq)}
    \frac{D}{|s_*|\Delta D}
    \left(\frac{\pi\x}{\qrep D}\right)^{\!s-\half}
    f(s)\mspace{1.5mu}ds
    \end{aligned}
    \qquad\qquad\text{(using \cref{lemma:prod_convergence})}
    \\
    \nonumber
    &\ll
    \frac{\Delta D}{D} \int_{\half - \eps - i\frac{D}{\Delta D}}^{\half - \eps + i\frac{D}{\Delta D}} (|s| + |\tau|)^{\half - \sigma}\,f(s)\mspace{1.5mu}ds
    \,+\, \frac{D}{\Delta D} \int_{\half - \eps + i\frac{D}{\Delta D}}^{\half - \eps + i\infty} (|s| + |\tau|)^{\half - \sigma}|s|^{-1}\,f(s)\mspace{1.5mu}ds
    \\
    \label{eq:t2s_contrib_smoothed_step3}
    &\ll
    |\tau|^\eps\frac{\Delta D}{D} \,+\, |\tau|^\eps\mspace{-2mu}\left(\frac{D}{\Delta D}\right)^{\!-1 + \eps}
    ,
  \end{align}
  using the facts $\tau \ll \big(\frac{D}{\#\cF}\big)^{1-\eps}$ and $f(s) \ll |s_*|^{-2 - \eps}$. Since $\frac{\Delta D}{D} \asymp D^{\delta - 1} \ll D^{\hat{\rho}_f}$, the contribution from the error term \eqref{eq:t2s_contrib_smoothed_step2} is admissible. It is straightforward to verify, proceeding along the same lines as the proofs of \cref{lemma:t1ns_contrib_r1_smoothed,lemma:t2ns_contrib_r1_smoothed,lemma:t1s_contrib_r1_smoothed}, that the integrals in the first and second terms of \eqref{eq:t2s_contrib_smoothed_step1} also constitute admissible error terms. The main term of \cref{lemma:t2s_contrib_r1_smoothed} is \eqref{eq:t2s_contrib_smoothed_mt} (divided by $2\pi i$).
\end{proof}

\begin{proof}[Proof of \cref{thm:gamma_murmurations}]
  By Mellin inversion applied to each Dirichlet series,
  \begin{align*}
    \frac{1}{2\pi i} \int_{(1+\theta)} \Favg L(s,\rep\otimes\chi_d)\,\x^{s-\half}\,f(s)\mspace{1.5mu}ds
    &=
    \Favg \frac{1}{\sqrt{\x}} \sum_{n=1}^\infty \cM^{-1}\mspace{-2mu}\{f\}\big(\tfrac{n}{x}\big) \,a_\rep(n)\chi_d(n)
    .
  \end{align*}

  We now evaluate the left hand side in a different way.
  The integrand is holomorphic, so we may shift the line of integration to $\Re(s) = c$ with $0 < c < 1$. Convergence is ensured by the convexity bound \cite[Thm.\ 8.2.3]{goldfeld:book}. Then, by the approximate functional equation \eqref{eq:afe_rep},
  \begin{align*}
    \frac{1}{2\pi i}
    &
    \int_{c - iT}^{c + iT} \Favg L(s,\rep \otimes \chi_d)\,\x^{s-\half}\,f(s)\mspace{1.5mu}ds
    \\
    &=
    \frac{1}{2\pi i} \int_{c - iT}^{c + iT} \Favg \sum_{\substack{n=1 \\ n \neq \square}}^\infty \frac{a_\rep(n)\chi_d(n)}{n^s}V_s\!\left(\frac{n}{\sqrt{\qrep d^\repdim}}\right)\x^{s-\half}\,f(s)\mspace{1.5mu}ds
    \\
    &+
    \frac{1}{2\pi i} \int_{c - iT}^{c + iT} \prod_{j=1}^\repdim \frac{\halfGamma{1-s+\bar\kappa_j}}{\halfGamma{s+\kappa_j}} \Favg \sum_{\substack{n=1 \\ n \neq \square}}^\infty \frac{\rootnum_{\rep\otimes\chi_d} a_{\bar\rep}(n)\chi_d(n)}{n^{1-s}}V_{1-s}^*\!\left(\frac{n}{\sqrt{\qrep d^\repdim}}\right)\xd^{\!s-\half}f(s)\mspace{1.5mu}ds
    \\
    &+
    \frac{1}{2\pi i} \int_{c - iT}^{c + iT} \Favg \sum_{n=1}^\infty \frac{a_\rep(n^2)\chi_d(n^2)}{n^{2s}}V_s\!\left(\frac{n^2}{\sqrt{\qrep d^\repdim}}\right)\x^{s-\half}\,f(s)\mspace{1.5mu}ds
    \\
    &+
    \frac{1}{2\pi i} \int_{c - iT}^{c + iT} \prod_{j=1}^\repdim \frac{\halfGamma{1-s+\bar\kappa_j}}{\halfGamma{s+\kappa_j}} \Favg \sum_{n=1}^\infty \frac{\rootnum_{\rep\otimes\chi_d} a_{\bar\rep}(n^2)\chi_d(n^2)}{n^{2-2s}}V_{1-s}^*\!\left(\frac{n^2}{\sqrt{\qrep d^\repdim}}\right)\xd^{\!s-\half}f(s)\mspace{1.5mu}ds
    .
  \end{align*}
  Applying \cref{lemma:t1ns_contrib_r1_smoothed,lemma:t2ns_contrib_r1_smoothed,lemma:t1s_contrib_r1_smoothed,lemma:t2s_contrib_r1_smoothed} to each of the four individual terms above proves \cref{thm:gamma_murmurations}.
\end{proof}

\appendix

\section{Auxiliary proofs}
\label{sec:lemma_proofs}

{\allowdisplaybreaks 

\subsection{The approximate functional equation}
\label{subsec:proofs_afe}

\hypertarget{proof:afe_psi}{}
\begin{proof}[Proof of \cref{lemma:afe_psi}]
\begin{align*}
  &
  \frac{1}{2\pi i} \int_{c_V - i\infty}^{c_V + i\infty}  \Lambda(s+w,\rep) \psi(s,w)\,\frac{dw}{w}
  \\
  &=
  \frac{1}{2\pi i} \int_{-c_V - i\infty}^{-c_V + i\infty}  \Lambda(s+w,\rep) \psi(s,w)\,\frac{dw}{w}
  + \sum_{|\Re(w')| < c_V} \underset{w'=w}{\mathrm{Res}} \,\Lambda(s+w,\rep) \frac{\psi(s,w)}{w}
  \\
  &=
  \frac{1}{2\pi i} \int_{-c_V - i\infty}^{-c_V + i\infty}  \rootnum_\rep \Lambda(1-s-w,\bar\rep) \psi(s,w)\,\frac{dw}{w}
  + \sum_{|\Re(w')| < c_V} \underset{w'=w}{\mathrm{Res}} \,\Lambda(s+w,\rep) \frac{\psi(s,w)}{w}
  \\
  &=
  -\frac{1}{2\pi i} \int_{c_V - i\infty}^{c_V + i\infty}  \rootnum_\rep \Lambda(1-s+w,\bar\rep) \psi(s,-w)\,\frac{dw}{w}
  + \sum_{|\Re(w')| < c_V} \underset{w'=w}{\mathrm{Res}} \,\Lambda(s+w,\rep) \frac{\psi(s,w)}{w}
  .
\end{align*}

As $\Lambda(s+w,\rep)\psi(s,w)$ is holomorphic at $w = 0$, 
\begin{align*}
  &\Lambda(s,\rep)\psi(s,0)
  \;+\;
  \sum_{\substack{|\Re(w')| < c_V \\ w' \neq 0}} \underset{w'=w}{\mathrm{Res}} \,\Lambda(s+w,\rep) \frac{\psi(s,w)}{w}
  \\
  &\hspace{1cm}
  =
  \frac{1}{2\pi i} \int_{c_V - i\infty}^{c_V + i\infty}  \Lambda(s+w,\rep) \psi(s,w)\,\frac{dw}{w}
  \;+\;
  \frac{\rootnum_\rep}{2\pi i} \int_{c_V - i\infty}^{c_V + i\infty}  \Lambda(1-s+w,\bar\rep) \psi(s,-w)\,\frac{dw}{w}
\end{align*}

Replacing $L(s+w,\rep)$ and $L(1-s+w,\rep)$ with their series definitions --- recall that by assumption $L(s+w,\rep)$ and $L(1-s+w,\bar\rep)$ converge absolutely whenever $\Re(w) = c_V$,

\begin{align*}
  \left(\frac{\qrep}{\pi^\repdim}\right)^{\!\frac{s}{2}}
  &
  \prod_{j=1}^\repdim \halfGamma{s+\kappa_j} L(s,\rep) \psi(s,0)
  \;+\;
  \sum_{\substack{|\Re(w')| < c_V \\ w' \neq 0}} \underset{w'=w}{\mathrm{Res}} \,\Lambda(s+w,\rep) \frac{\psi(s,w)}{w}
  \\
  &
  =
  \frac{1}{2\pi i} \int_{c_V - i\infty}^{c_V + i\infty} \left(\frac{\qrep}{\pi^\repdim}\right)^{\!\frac{s+w}{2}} \prod_{j=1}^\repdim \halfGamma{s+w+\kappa_j} \sum_{n=1}^\infty \frac{a_\rep(n)}{n^{s+w}} \psi(s,w)\,\frac{dw}{w}
  \\
  &
  +
  \frac{\rootnum_\rep}{2\pi i} \int_{c_V - i\infty}^{c_V + i\infty}
  \left(\frac{\qrep}{\pi^\repdim}\right)^{\!\frac{1-s+w}{2}} \prod_{j=1}^\repdim \halfGamma{1-s+w+\bar\kappa_j} \sum_{n=1}^\infty \frac{a_{\bar\rep}(n)}{n^{1-s+w}} \psi(s,-w)\,\frac{dw}{w}
  ,
\end{align*}
equivalently
\begin{align*}
  \prod_{j=1}^\repdim
  \halfGamma{s+\kappa_j}
  &L(s,\rep) \psi(s,0)
  \\
  ={}
  &
  \sum_{n=1}^\infty \frac{a_\rep(n)}{n^{s}}
  \frac{1}{2\pi i} \int_{c_V - i\infty}^{c_V + i\infty} \prod_{j=1}^\repdim \halfGamma{s+w+\kappa_j} \psi(s,w) \left(\frac{\qrep^\half}{\pi^{\frac{\repdim}{2}} n}\right)^{\!w} \,\frac{dw}{w}
  \\
  {}+{}
  \rootnum_\rep\left(\frac{\pi^\repdim}{\qrep}\right)^{\!s - \half}
  &
  \sum_{n=1}^\infty \frac{a_{\bar\rep}(n)}{n^{1-s}}
  \frac{1}{2\pi i} \int_{c_V - i\infty}^{c_V + i\infty} \prod_{j=1}^\repdim \halfGamma{1-s+w+\bar\kappa_j} \psi(s,-w) \left(\frac{\qrep^\half}{\pi^{\frac{\repdim}{2}} n}\right)^{\!w} \,\frac{dw}{w}
  \\
  -{}
  &\sum_{\substack{|\Re(w')| < c_V \\ w' \neq 0}} \underset{w'=w}{\mathrm{Res}} \,\Lambda(s+w,\rep) \frac{\psi(s,w)}{w}
  .
  \qedhere
\end{align*}
\end{proof}

\hypertarget{proof:V_approx}{}
\begin{proof}[Proof of \cref{lemma:V_approx}]
  Recall \eqref{eq:Gdef}:
  \begin{align*}
    G(s) &\coloneqq \prod_{j=1}^\repdim \halfGamma{s + \kappa_j}
  \end{align*}

  For brevity, in addition to the usual $s \eqqcolon \sigma + it$, write
  \begin{align}
    \label{eq:sigmawdef}
    w \eqqcolon \sigma_w + it_w,\qquad\qquad \kappa_j \eqqcolon \sigma_j + it_j
    .
  \end{align}
  Moreover, write
  \begin{align}
    \nonumber
    |s_*| \coloneqq \max_j\!\left\{ |s + \kappa_j|, |1 - s + \bar\kappa_j|, 1 \right\}
    .
  \end{align}

  To prove \cref{lemma:V_approx}, we will
  \begin{enumerate}[label=\arabic*.]
  \item
    Estimate the sizes of the individual factors in $V_s$'s integrand,
  \item
    Combine these estimates and regroup terms,
  \item
    Bound each of the new groups of terms individually,
  \item
    Recombine these groups so that we may bound the integral along the vertical line $\Re(w) = \sigma_w$,
  \item
    Observe that, for our bound to be small, we must take $\sigma_w < 0$ if $n < Q^\half\xi^{\pm 1}$, thereby picking up a residue.
  \end{enumerate}
  We look only at $V_s(\pi^{\frac{\repdim}{2}}n / \qrep^\half\xi)$. Handling $V_{1-s}^*(\pi^{\frac{\repdim}{2}} \xi n / \qrep^\half)$ afterwards is easy.

  For the first part of the proof we estimate
  \begin{align*}
    \frac{G(s+w)}{G(s)},
    \;
    \left(e^{i\frac{w}{A}} + e^{-i\frac{w}{A}} - 1\right)^{-B}, \text{ and}
    \;
    \left(\xi^{\pm 1}y\right)^{-w}
  \end{align*}
  in \eqref{eq:V_G_approx}, \eqref{eq:V_psi_approx}, and \eqref{eq:V_xi_approx} respectively.
    
  We begin with the gamma factors $G(s+w)/G(s)$. By Stirling's formula \eqref{eq:stirling} (see also \cref{rem:stirling}), we have the approximation
  \begin{align}
    \nonumber
    \left|\frac{G(s+w)}{G(s)}\right|
    &=
    \prodr \left|\frac{s+w+\kappa_j}{2}\right|^{\half(\sigma + \sigma_w + \sigma_j - 1)}
    \exp\!\left(-\half \sumr \sigma + \sigma_w + \sigma_j\right)
    \\
    \nonumber
    &\hspace{3cm}\cdot
    \exp\!\left(-\half \sumr (t + t_w + t_j)\arg(s + w + \kappa_j)\right)
    \\
    \nonumber
    &\quad\cdot
    \Bigg[
    \prodr \left|\frac{s+\kappa_j}{2}\right|^{\half(\sigma + \sigma_j - 1)}
    \exp\!\left(-\half \sumr \sigma + \sigma_j\right)
    \\
    \nonumber
    &\hspace{3cm}\cdot
    \exp\!\left(-\half \sumr (t + t_j)\arg(s + \kappa_j)\right)
    \Bigg]^{\!-1}
    \left(1 + O(\repdim |s_*|^{-1})\right)
    \\
    \label{eq:V_G_approx}
    &
    \begin{aligned}
    ={}&\exp\!\left(-\half\sumr t_w \arg(s + w + \kappa_j)  + \half \sumr (t + t_j)\big(\arg(s + \kappa_j) - \arg(s + w + \kappa_j)\big)\right)
    \\
    &\cdot
    \left(\frac{1}{(2e)^{\repdim}} \prodr |s + \kappa_j|\right)^{\frac{\sigma_w}{2}}
    \prodr \left|1+\frac{w}{s+\kappa_j}\right|^{\half(\sigma + \sigma_w + \sigma_j - 1)}
    \left(1 + O(\repdim |s_*|^{-1})\right)
    .
    \end{aligned}
  \end{align}  
  
  Continuing with the first step of the proof, we estimate the second factor of the integrand,
  \begin{align*}
    \left(e^{i\frac{w}{A}} + e^{-i\frac{w}{A}} - 1\right)^{-B}
    &=
    \left(e^{i\frac{\sigma_w}{A}} e^{-\frac{t_w}{A}} + e^{-i\frac{\sigma_w}{A}} e^{\frac{t_w}{A}} - 1\right)^{-B}
    \\
    &=
    \left(e^{\mp i\frac{\sigma_w}{A}} e^{\pm \frac{t_w}{A}}
    \left(e^{\pm i\frac{2\sigma_w}{A}} e^{\mp \frac{2t_w}{A}} - e^{\pm i\frac{\sigma_w}{A}} e^{\mp \frac{t_w}{A}} + 1\right)
    \right)^{-B}
    .
  \end{align*}
  Having written the factor in this form, we obtain the following estimate of its size:
  \begin{align}
    \nonumber
    \left|\left(e^{i\frac{w}{A}} + e^{-i\frac{w}{A}} - 1\right)^{-B}\right|
    &=
    e^{-\frac{B}{A}|t_w|}
    \left|1 - e^{\pm i\frac{\sigma_w}{A}} e^{-\frac{|t_w|}{A}} + e^{\pm i\frac{2\sigma_w}{A}} e^{-\frac{2|t_w|}{A}}\right|^{-B}
    \\
    \label{eq:V_psi_approx}
    &=
    e^{-\frac{B}{A}|t_w|}\left(1 + O\!\left(B e^{-\frac{|t_w|}{A}}\right)\right)
    .
  \end{align}

  The size of third factor of $V_s$'s integrand needn't be estimated at all in fact:
  \begin{align}
    \label{eq:V_xi_approx}
    \left|\left(\xi^{\pm 1}y\right)^{-w}\right|
    &=
    y^{-\sigma_w} |\xi|^{\mp \sigma_w} \exp\!\left(\pm t_w \arg\xi \right)
    .
  \end{align}

  We are now on to step 2 of the proof.
  Combining \eqref{eq:V_G_approx}, \eqref{eq:V_psi_approx}, and \eqref{eq:V_xi_approx},
  \begin{align}
    \nonumber
    &\left|\frac{G(s+w)}{G(s)} \left(e^{i\frac{w}{A}} + e^{-i\frac{w}{A}} - 1\right)^{-B}\left(\frac{\xi \qrep^\half}{\pi^{\frac{\repdim}{2}} n}\right)^{w}\frac{1}{w}\right|
    \\
    \nonumber
    ={}&
    \exp\!\left(-\frac{B}{A}|t_w| - t_w \arg\xi - \frac{t_w}{2} \sumr \arg(s + w + \kappa_j)\right)
    \left(\frac{\qrep^\half |\xi| }{(2\pi e)^{\frac{\repdim}{2}} n} \prodr |s + \kappa_j|^\half \right)^{\sigma_w}
    \\
    \label{eq:V_integrand_size}
    \cdot
    &\exp\!\left(\half \sumr (t + t_j)\big(\arg(s + \kappa_j) - \arg(s + w + \kappa_j)\big)\right)
    \frac{1}{|w|}\prodr \left|1+\frac{w}{s+\kappa_j}\right|^{\half(\sigma + \sigma_w + \sigma_j - 1)}
    \\
    \nonumber
    \cdot
    &\left(1 + O\!\left(\repdim |s_*|^{-1} + B e^{-\frac{|t_w|}{A}}\right)\right)
    .
  \end{align}
  This step is quite short.
  
  We now proceed with step 3 of the proof, where we regroup the factors in \eqref{eq:V_integrand_size} judiciously, and then bound these individual groups.

  The first factor of \eqref{eq:V_integrand_size} we look at is $|\xi^{it_w}| = \exp(-t_w \arg\xi)$.
  When $\sigma > \Re(s_0)$, which is always the case under our assumptions, we have
  \begin{align}
    \nonumber
    |\arg\xi| < \pi|\beta|,
  \end{align}
  where we have used the assumption that $|\beta| < 1$. So
  \begin{align}
    \label{eq:V_argxi}
    \exp\!\left(- t_w \arg\xi\right) \leqslant \exp(\pi|\beta||t_w|).
  \end{align}

  Next we look at the collection of factors of the form $\exp(-t\arg(s))$ coming from Stirling's formula \eqref{eq:stirling}.
  With some case-by-case analysis, we arrive at the bound, coarse for the sake of simplicity,
  \begin{align}
    \label{eq:V_args}
    \half \sumr (t + t_j)\arg(s + \kappa_j) - (t + t_j + t_w) \arg(s + w + \kappa_j)
    \,\leqslant\,
    \frac{\pi\repdim}{2}|t_w| + \sumr |\sigma| + |\sigma_j| + |\sigma_w|
    .
  \end{align}
  
  We bound the last piece of \eqref{eq:V_integrand_size} by bounding its logarithm, so that it may be more easily compared to the exponential factors.
  \begin{align}
    \nonumber
    &\log\!\left(\prodr \left|1+\frac{w}{s+\kappa_j}\right|^{\half(\sigma + \sigma_w + \sigma_j - 1)}\right)
    \\
    \nonumber
    &=
    \half\sumr \big(\sigma + \sigma_w + \sigma_j - 1\big)\log\!\left|1 + \frac{w}{s+\kappa_j}\right|
    \\
    \nonumber
    &\leqslant
    \half\sumr \big|\sigma + \sigma_w + \sigma_j - 1\big|\log\!\left(1 + \left|\frac{\sigma_w}{s+\kappa_j}\right| + \left|\frac{t_w}{s+\kappa_j}\right|\right)
    \\
    \nonumber
    &\leqslant
    \half\sumr \big|\sigma + \sigma_w + \sigma_j - 1\big|\left(\left|\frac{\sigma_w}{s+\kappa_j}\right| + \left|\frac{t_w}{s+\kappa_j}\right|\right)
    \\
    \nonumber
    &\leqslant
    \frac{|t_w|}{2}\sumr \frac{|\sigma + \sigma_w + \sigma_j - 1\big|}{|s+\kappa_j|}
    +
    \frac{|\sigma_w|}{2}\sumr \frac{|\sigma + \sigma_w + \sigma_j - 1\big|}{|s+\kappa_j|}
    \\
    \label{eq:V_poly}
    &\leqslant
    \frac{|t_w|}{2}\sumr \left(\left|\frac{\sigma_w - 1}{\sigma + \sigma_j}\right| + 1\right)
    +
    \frac{|\sigma_w|}{2}\sumr \left(\left|\frac{\sigma_w - 1}{\sigma + \sigma_j}\right| + 1\right)
    .
  \end{align}

  We are now on to step 4 of the proof.
  Multiplying the bounds \eqref{eq:V_argxi}, \eqref{eq:V_args}, and \eqref{eq:V_poly}, we bound the size of \eqref{eq:V_integrand_size}, and hence the integrand of $V_s$:
  \begin{align}
    \nonumber
    &\exp\!\left(-\frac{B}{A}|t_w| \pm t_w \arg\xi - \frac{t_w}{2} \sumr \arg(s + w + \kappa_j)\right)
    \left(\frac{\qrep^\half |\xi| }{(2\pi e)^{\frac{\repdim}{2}} n} \prodr |s + \kappa_j|^\half \right)^{\sigma_w}
    \\
    \nonumber
    \cdot
    &\exp\!\left(\half \sumr (t + t_j)\big(\arg(s + \kappa_j) - \arg(s + w + \kappa_j)\big)\right)
    \frac{1}{|w|}\prodr \left|1+\frac{w}{s+\kappa_j}\right|^{\half(\sigma + \sigma_w + \sigma_j - 1)}
    \\
    \label{eq:V_bound_prod}
    &
    \begin{aligned}
    \ll
    &\left(\frac{\qrep^\half |\xi| }{(2\pi e)^{\frac{\repdim}{2}} n} \prodr |s + \kappa_j|^\half \right)^{\sigma_w}
    \exp\!\left(\sumr |\sigma| + |\sigma_j| + |\sigma_w| + \frac{|\sigma_w|}{2} \left(\left|\frac{\sigma_w - 1}{\sigma + \sigma_j}\right| + 1\right)\right)
    \frac{1}{|\sigma_w|}
    \\
    &\cdot
    \exp\!\left(\left[-\frac{B}{A} + \pi|\beta| + \frac{\pi\repdim}{2} + \half\sumr \left(\left|\frac{\sigma_w - 1}{\sigma + \sigma_j}\right| + 1\right)\right]|t_w|\right)
    .
    \end{aligned}
  \end{align}
  By assumption, we have
  \begin{align*}
    \frac{B}{A} > \pi|\beta| + \frac{\pi\repdim}{2} + \half\sumr \left(\left|\frac{\sigma_w - 1}{\sigma + \sigma_j}\right| + 1\right),
  \end{align*}
  so the integrand decays exponentially, and the value of the integrand of $V_s$ along a vertical line is $\ll$ its size at $t_w = 0$:
  \begin{align}
    \nonumber
    &\int_{(\sigma_w)} \left|\frac{G(s+w)}{G(s)} \left(e^{i\frac{w}{A}} + e^{-i\frac{w}{A}} - 1\right)^{-B}\left(\frac{\xi \qrep^\half}{\pi^{\frac{\repdim}{2}} n}\right)^{\!w}\right|\frac{dw}{w}
    \\
    \label{eq:V_bound}
    &\ll
    \left(\frac{\qrep^\half |\xi| }{(2\pi e)^{\frac{\repdim}{2}} n} \prodr |s + \kappa_j|^\half \right)^{\!\sigma_w}
    \exp\!\left(\sumr |\sigma| + |\sigma_j| + \frac{|\sigma_w|}{2} \left(\left|\frac{\sigma_w - 1}{\sigma + \sigma_j}\right| + 3\right)\right)
    \frac{r + B}{|\sigma_w|}
    .
  \end{align}

  Finally we arrive at step 5 of the proof. Note that the integrand of $V_s$ has a pole at $w = 0$, with residue $1$. Moreover there are no other poles in the region $\tfrac{\pi}{3}A > \sigma_w > -\min\!\left\{\frac{\pi}{3}A,\, \sigma + \Re(\kappa_1),\dots,\sigma + \Re(\kappa_\repdim)\right\}$ specified by the hypothesis of \cref{lemma:V_approx}. The result then follows from case-by-case analysis depending on whether the positive quantity
  \begin{align*}
    \frac{\qrep^\half |\xi| }{(2\pi e)^{\frac{\repdim}{2}} n} \prodr |s + \kappa_j|^\half
    &=
    \frac{Q^\half|\xi|}{n}
  \end{align*}
  is greater than $1$ or less than $1$.

  The bound for $V_{1-s}^*(\pi^{\frac{\repdim}{2}}\xi n/\qrep^\half )$ is derived similarly, with the only differences being $\sigma \mapsto 1 - \sigma$, $\kappa_j \mapsto \bar\kappa_j$, and $\xi \mapsto \xi^{-1}$.
\end{proof}

\hypertarget{proof:V_deriv_approx}{}
\begin{proof}[Proof of \cref{lemma:V_deriv_approx}]
  We prove the result for $V_s$. The only difference with $V_s^*$ is $\kappa_j \mapsto \bar\kappa_j$, which is immaterial.
  \begin{align}
    \nonumber
    V_s\!\left(\frac{\xi^{\pm 1} y}{D^{\frac{\repdim}{2}}}\right) - V_s\!\left(\frac{\xi^{\pm 1} y}{d^{\frac{\repdim}{2}}}\right)
    &=
    \frac{1}{2\pi i} \int_{(c_V\!)} \frac{G(s+w)}{G(s)} \left(e^{i\frac{w}{A}} + e^{-i\frac{w}{A}} - 1\right)^{-B} (\xi^{\pm 1}y)^{-w}\left(D^{\frac{\repdim w}{2}} - d^{\frac{\repdim w}{2}}\right)\frac{dw}{w}
    \\
    \label{eq:V_deriv_1}
    &=
    \frac{1}{2\pi i} \int_{(c_V\!)} \frac{G(s+w)}{G(s)} \left(e^{i\frac{w}{A}} + e^{-i\frac{w}{A}} - 1\right)^{-B} \left(\frac{\xi^{\mp 1}D^{\frac{\repdim}{2}}}{y}\right)^{\!w}
    \left(1 - \left(1 - \frac{D - d}{D}\right)^{\!\frac{\repdim w}{2}}\right)
    \frac{dw}{w}
    .
  \end{align}
  The singularity of \eqref{eq:V_deriv_1} at $w = 0$ is removable, and thus the contour can be shifted freely in the same range as \cref{lemma:V_approx}.

  We have
  \begin{align}
    \label{eq:V_deriv_2}
    \left|1 - \left(1 - \frac{D - d}{D}\right)^{\!\frac{\repdim w}{2}}\right| \ll \min\!\left\{1, \frac{r\Delta D}{D}w \right\}
  \end{align}
  if $\Re(w) > 0$, and
  \begin{align}
    \label{eq:V_deriv_3}
    \left|1 - \left(1 - \frac{D - d}{D}\right)^{\!\frac{\repdim w}{2}}\right| \ll \min\!\left\{e^{\frac{r \Delta D}{D} |\Re(w)|}, \frac{r\Delta D}{D}w \right\}
  \end{align}
  if $\Re(w) < 0$.

  Mimicking calculations done during the \hyperlink{proof:V_approx}{proof of \cref*{lemma:V_approx}}, we see that the integrand of \eqref{eq:V_deriv_1} decays exponentially. If
  \begin{align}
    \label{eq:V_deriv_4}
    \frac{B}{A} - \left(\pi|\beta| + \frac{\pi\repdim}{2} + \half\sumr \left(\left|\frac{\sigma_w - 1}{\sigma + \sigma_j}\right| + 1\right)\right)
    \;>\; \frac{r\Delta D}{D}
    ,
  \end{align}
  then the exponential decay is rapid enough to ensure that the value of the integral is $\ll$ the value one would obtain using only the $r\Delta D/D$ upper bound of \eqref{eq:V_deriv_2} and \eqref{eq:V_deriv_3}. A consequence of \eqref{eq:V_deriv_4} is that \eqref{eq:V_deriv_1} is $\ll$ the value of the integrand at $\Im(w) = 0$.

  It then follows from \eqref{eq:V_bound} that
  \begin{align}
    \nonumber
    &\int_{(\sigma_w)} \left|\frac{G(s+w)}{G(s)} \left(e^{i\frac{w}{A}} + e^{-i\frac{w}{A}} - 1\right)^{-B}\left(\frac{\xi^{\mp 1} D^{\frac{\repdim}{2}}}{y}\right)^{\!w}\right|
    \left|1 - \left(1 - \frac{D - d}{D}\right)^{\!\frac{\repdim w}{2}}\right|
    \frac{dw}{w}
    \\
    \nonumber
    \ll
    &\left(\frac{\qrep^\half |\xi|^{\pm 1} }{(2\pi e)^{\frac{\repdim}{2}} n} \prodr |s + \kappa_j|^\half \right)^{\!\sigma_w}
    \exp\!\left(\sumr \sigma + |\Re(\kappa_j)| + \frac{|\sigma_w|}{2} \left(\left|\frac{\sigma_w - 1}{\sigma + \Re(\kappa_j)}\right| + 3\right)\right)
    \frac{\repdim + B}{|\sigma_w|}
    \\
    \nonumber
    &\cdot\frac{\repdim \Delta D}{D}\left(1 + \frac{\repdim \Delta D}{D}|\sigma_w|\right)
    .
  \end{align}
  By assumption the $\frac{\repdim \Delta D}{D}|\sigma_w|$ term does not contribute.
\end{proof}

\subsection{Sums over fundamental discriminants}
\label{subsec:proofs_discriminants}

\hypertarget{proof:sqfree_character_estimate}{}
\begin{proof}[Proof of \cref{lemma:sqfree_character_estimate}]
  Consider the Dirichlet series
  \begin{align*}
    \sum_{n = 1}^\infty n^z \mu(n)^2 \chi(n)\chi_m^0(n) n^{-s}
    &=
    \prod_{p\nmid m} 1 + \frac{\chi(p)}{p^{s-z}}
    \\
    &=
    \prod_{p\mid m} \left(1 + \frac{\chi(p)}{p^{s-z}}\right) \prod_{p} \left(\frac{1 - \frac{\chi(p)}{p^{2s-2z}}}{1 - \frac{\chi(p)}{p^{s-z}}}\right)
    \\
    &=
    \frac{L(s-z,\chi)}{L(2s-2z,\chi^2)} \prod_{p\mid m} \left(1 + \frac{\chi(p)}{p^{s-z}}\right) 
    .
  \end{align*}

  By Perron's formula \eqref{eq:perron},
  \begin{align*}
    \sum_{n < \x} n^z \mu(n)^2 \chi(n)\chi_m^0(n)
    &=
    \frac{1}{2\pi i} \int_{\sigma_0 - iT}^{\sigma_0 + iT} \frac{L(s-z,\chi)}{L(2s-2z,\chi^2)} \prod_{p\mid m} \left(1 + \frac{\chi(p)}{p^{s-z}}\right) \x^s\,\frac{ds}{s}
    \;+\: O\big(\x^{\sigma_0} T^{-1} (\x Tq)^\eps\big)
  \end{align*}
  for $\sigma_0 > 1 + \Re(z)$. 
  If $T > |\Im(z)|$, then, shifting to $\sigma_1 = \thalf + \Re(z)$,
  \begin{align}
    \nonumber
    \sum_{n < \x} n^z \mu(n)^2 \chi(n)\chi_m^0(n)
    &=
    \one{\text{$\chi$ is trivial}}\frac{\x^{z+1}}{z+1}\frac{6}{\pi^2}\prod_{p \mid m\ff}\frac{p}{p+1}
    \\
    \nonumber
    &+
    \frac{1}{2\pi i}\left(\int_{\sigma_0 - iT}^{\sigma_1 - iT} + \int_{\sigma_1 + iT}^{\sigma_0 + iT}\right) \frac{L(s-z,\chi)}{L(2s-2z,\chi^2)} \prod_{p\mid m} \left(1 + \frac{\chi(p)}{p^{s-z}}\right) \x^s\,\frac{ds}{s}
    \\
    \label{eq:LoverL2int}
    &+
    \frac{1}{2\pi i} \int_{\sigma_1 - iT}^{\sigma_1 + iT} \frac{L(s-z,\chi)}{L(2s-2z,\chi^2)} \prod_{p\mid m} \left(1 + \frac{\chi(p)}{p^{s-z}}\right) \x^s\,\frac{ds}{s}
    \\
    \nonumber
    &+
    O\big(\x^{\sigma_0} T^{-1} (\x Tq)^\eps\big)
    .
  \end{align}

  For $\sigma > 0$,
  \begin{align}
    \label{eq:trivprodbound}
    \begin{aligned}
    \prod_{p\mid m} \left|1 + \frac{\chi(p)}{p^{s}}\right|
    &\leqslant
    \exp\!\left(\sum_{p\mid m} \log(1 + p^{-\sigma})\right)
    \\
    &\leqslant
    \exp\!\left(\sum_{p\mid m} p^{-\sigma}\right)
    \\
    &\ll
    \exp\!\left(\sum_{p\ll \log m} p^{-\sigma}\right)
    \\
    &\ll
    \exp\!\left(\frac{(\log m)^{1-\sigma}}{\log\log m}\right)
    \\
    &\ll
    m^\eps
    .
    \end{aligned}
  \end{align}

  Petrow--Young \cite{petrow_young:2023} prove the subconvexity bound
  \begin{align}
    \label{eq:pysubconvexity}
    L(\thalf + it, \chi) \ll (\ff(|t|+1))^{\frac{1}{6} + \eps}.
  \end{align}
  By \cite[Thms.\ 6.7 and 11.4]{MV},
  \begin{align}
    \label{eq:mv1overL}
    \frac{1}{L(1 + it,\chi)} \ll (\ff (|t|+1))^\eps.
  \end{align}

  Substituting \eqref{eq:trivprodbound} \eqref{eq:pysubconvexity} \eqref{eq:mv1overL} into \eqref{eq:LoverL2int},
  \begin{align*}
    &\frac{1}{2\pi i} \int_{\sigma_1 - iT}^{\sigma_1 + iT} \frac{L(s-z,\chi)}{L(2s-2z,\chi^2)} \prod_{p\mid m} \left(1 + \frac{\chi(p)}{p^{s-z}}\right) \x^s\,\frac{ds}{s}
    \\
    &\ll
    m^\eps \int_{\sigma_1 - iT}^{\sigma_1 + iT} \left|\frac{L(s-z,\chi)}{L(2s-2z,\chi^2)}\right| \x^\sigma\,\frac{ds}{s}
    \\
    &\ll
    m^\eps x^{\sigma_1} \int_{\sigma_1 - iT}^{\sigma_1 + iT} (\ff (|t|+1))^{\frac{1}{6} + \eps} \,\frac{ds}{s}
    \\
    &\ll
    x^{\sigma_1} \ff^{\frac{1}{6}} \left(T^{\frac{1}{6}} + |\Im(z)|^{\frac{1}{6}}\right) (m\ff T|\Im(z)|)^\eps
    .
  \end{align*}

  The contribution from the horizontal segments is the Perron error term plus the error term above times a factor of $T^{-1}$, and so may be omitted.

  Thus,
  \begin{align*}
    \sum_{n < \x} n^z \mu(n)^2 \chi(n)\chi_m^0(n)
    &=
    \one{\text{$\chi$ is trivial}}\frac{\x^{z+1}}{z+1}\frac{6}{\pi^2}\prod_{p \mid m\ff}\frac{p}{p+1}
    \\
    &+
    O\left(\left(\x^{1 + \Re(z)} T^{-1} + x^{\sigma_1} \ff^{\frac{1}{6}} \left(T^{\frac{1}{6}} + |\Im(z)|^{\frac{1}{6}}\right) \right) (\x m\ff T|\Im(z)|)^\eps \right)
    .
  \end{align*}
  Taking $T = \max\{\x^{\frac{3}{7}}, 2|\Im(z)|\}$ gives \cref{lemma:sqfree_character_estimate}.
\end{proof}

\hypertarget{proof:disc_power_sum}{}
\begin{proof}[Proof of \cref{lemma:disc_power_sum}]
  We show that \cref{lemma:disc_power_sum} holds with an error term of
  \begin{align}
    \label{eq:disc_sum_et1}
    D^{\Re(z) + \half}\ff^{\frac{1}{6}}\Big(D^{\frac{1}{14}} + |\Im(z)|^{\frac{1}{6}}\Big)(m\ff D(|z|+1))^\eps
    ,
  \end{align}
  as well as with an error term of
  \begin{align}
    \label{eq:disc_sum_et2}
    D^{\Re(z) + \half}\ff^{\frac{1}{2}}(|z|+1) \ff^\eps\divcount(m)
    .
  \end{align}
  We may then take the $\min$.
  The \hyperlink{proof:disc_power_sum_et1}{proof for (\ref*{eq:disc_sum_et1})} and the \hyperlink{proof:disc_power_sum_et2}{proof for (\ref*{eq:disc_sum_et2})} are given separately below, and proceed using different methods.
  \renewcommand{\qedsymbol}{}
\end{proof}

\hypertarget{proof:disc_power_sum_et1}{}
  \begin{proof}[Proof of \cref{lemma:disc_power_sum} with error \eqref{eq:disc_sum_et1}]
  Let
  \begin{align}
    \label{eq:fundamental_discriminant_classification}
    &
    \begin{aligned}
    \cF_1 &\coloneqq \{0 < d < D \,:\, \text{$d$ is squarefree and $1 \mod 4$}\}
    \\
    \cF_2 &\coloneqq \{0 < d < D \,:\, \text{$d = 8\ell$ for some odd squarefree $\ell$}\}
    \\
    \cF_3 &\coloneqq \{0 < d < D \,:\, \text{$d = 4\ell$ for some squarefree $\ell$ which is $3 \mod 4$}\}
    .
    \end{aligned}
  \end{align}
  It is standard that every fundamental discriminant is in exactly one of the families $\cF_1$, $\cF_2$, or $\cF_3$ \cite[Thm.\ 9.13]{MV}.

  We evaluate the left hand side of \cref{lemma:disc_power_sum} for each family individually.

  $\fbox{$\cF_1$}$
  Let $\chi_4^-$ denote the nontrivial character mod $4$.
  Write
  \begin{align*}
    &\one{\text{$d$ is squarefree, $d = 1 \mod 4$, $(d,m) = 1$, $d = a \mod q$}}
    \,=\,
    \mu(d)^2 \frac{\chi_4^0(d) + \chi_4^-(d)}{2} \chi_m^0(d) \frac{1}{\phi(q)} \sum_{\chi_q \mod q} \overline{\chi}_q(a) \chi_q(d)
    .
  \end{align*}
  Substituting,
  \begin{align*}
    \starsum*_{\substack{d \in \cF_1 \\ d = a \mod q \\ (d,m) = 1}} d^z
    &=
    \sum_{d = 1}^D d^z \mu(d)^2 \frac{\chi_4^0(d) + \chi_4^-(d)}{2} \chi_m^0(d) \frac{1}{\phi(q)} \sum_{\chi_q \mod q} \overline{\chi}_q(a) \chi_q(d)
    \\
    &=
    \frac{1}{2\phi(q)} \sum_{\chi_4 \mod 4} \sum_{\chi_q \mod q} \overline{\chi}_q(a) \sum_{d = 1}^D d^z \mu(d)^2 \chi_4\chi_m^0\chi_q(d)
    .
  \end{align*}
  By \cref{lemma:sqfree_character_estimate},
  \begin{align*}
    \frac{1}{2\phi(q)}
    \sum_{\chi_4 \mod 4} \sum_{\chi_q \mod q}
    &
    \overline{\chi}_q(a)
    \sum_{d = 1}^D d^z \mu(d)^2 \chi_4\chi_m^0\chi_q(d)
    \\
    ={}& \frac{1}{2\phi(q)} \sum_{\chi_4 \mod 4} \sum_{\chi_q \mod q} \overline{\chi}_q(a)
    \one{\text{$\chi_4 \chi_q$ is trivial}} \frac{D^{z+1}}{z+1} \frac{6}{\pi^2} \prod_{p \mid 4qm} \frac{p}{p+1}
    \\
    &+ O\!\left(q^{\frac{1}{6}} D^{\Re(z) + \half}\!\left(D^{\frac{1}{14}} + |\Im(z)|^{\frac{1}{6}}\right) \big(mqD|\Im(z)|\big)^\eps\right)
    .
  \end{align*}
  
  If $4 \nmid q$, then $\chi_4 \chi_q$ is trivial iff $\chi_4$ and $\chi_q$ are both trivial. If $4 \mid q$, then $\chi_4 \chi_q$ is trivial iff $\chi_4$ and $\chi_q$ are both trivial or $\chi_4 = \chi_4^-$ and $\chi_q$ is the lift of $\chi_4^-$ to the modulus $q$. It follows that
  \begin{align}
    \nonumber
    &\frac{1}{2\phi(q)} \sum_{\chi_4 \mod 4} \sum_{\chi_q \mod q} \overline{\chi}_q(a) \one{\text{$\chi_4 \chi_q$ is trivial}} \frac{D^{z+1}}{z+1} \frac{6}{\pi^2} \prod_{p \mid 4qm} \frac{p}{p+1}
    \\
    \label{eq:F1_power_sum}
    &= \frac{D^{z+1}}{z+1} \frac{6}{\pi^2\phi(q)} \prod_{p \mid 4qm} \frac{p}{p+1} \cdot
    \begin{cases}
      \thalf & \text{if $4 \nmid q$} \\
      \one{a = 1 \mod 4} & \text{if $4 \mid q$.}
    \end{cases}
  \end{align}

  $\fbox{$\cF_2$}$
  Similarly,
  \begin{align}
    \nonumber
    \starsum*_{\substack{d \in \cF_2 \\ d = a \mod q \\ (d,m) = 1}} d^z
    {}={}&
    \sum_{\ell = 1}^{D/8} (8\ell)^z \mu(\ell)^2 \chi_4^0(\ell) \chi_m^0(8\ell) \frac{1}{\phi(q)} \sum_{\chi_q \mod q} \overline{\chi}_q(a) \chi_q(8\ell)
    \\
    \nonumber
    {}={}&
    \frac{8^z\chi_m^0(8)}{\phi(q)} \sum_{\chi_q \mod q} \chi_q(8) \overline{\chi}_q(a) \sum_{\ell = 1}^{\frac{D}{8}} \ell^z \mu(\ell)^2 \chi_4^0\chi_m^0\chi_q(\ell)
    \\
    \nonumber
    {}={}&
    \frac{8^z\chi_m^0(8)}{\phi(q)} \sum_{\chi_q \mod q} \chi_q(8) \overline{\chi}_q(a)
    \one{\text{$\chi_q$ is trivial}} \frac{D^{z+1}}{8^{z+1}(z+1)} \frac{6}{\pi^2} \prod_{p \mid 4qm} \frac{p}{p+1}
    \\
    \nonumber
    &+ O\!\left(q^{\frac{1}{6}} D^{\Re(z) + \half}\!\left(D^{\frac{1}{14}} + |\Im(z)|^{\frac{1}{6}}\right) \big(mqD|\Im(z)|\big)^\eps\right)
    \\
    \label{eq:F2_power_sum}
    {}={}&
    \frac{1}{8}\one{(qm,2) = 1} \frac{D^{z+1}}{(z+1)} \frac{6}{\pi^2\phi(q)} \prod_{p \mid 4qm} \frac{p}{p+1}
    \\
    \nonumber
    &+ O\!\left(q^{\frac{1}{6}} D^{\Re(z) + \half}\!\left(D^{\frac{1}{14}} + |\Im(z)|^{\frac{1}{6}}\right) \big(mqD|\Im(z)|\big)^\eps\right)
    .
  \end{align}

  $\fbox{$\cF_3$}$
  Similarly,
  \begin{align*}
    \starsum*_{\substack{d \in \cF_3 \\ d = a \mod q \\ (d,m) = 1}} d^z
    {}={}&
    \sum_{\ell = 1}^{D/4} (4\ell)^z \mu(\ell)^2 \frac{\chi_4^0(\ell) - \chi_4^-(\ell)}{2} \chi_m^0(4\ell) \frac{1}{\phi(q)} \sum_{\chi_q \mod q} \overline{\chi}_q(a) \chi_q(4\ell)
    \\
    {}={}&
    \frac{4^z\chi_m^0(4)}{2\phi(q)} \sum_{\chi_4 \mod 4} \sum_{\chi_q \mod q} \chi_q(4) \overline{\chi}_q(a) \sum_{\ell = 1}^{\frac{D}{4}} \ell^z \mu(\ell)^2 \chi_4\chi_m^0\chi_q(\ell)
    \\
    {}={}&
    \frac{4^z\chi_m^0(4)}{2\phi(q)} \sum_{\chi_4 \mod 4} \sum_{\chi_q \mod q} \chi_q(4) \overline{\chi}_q(a)
    \one{\text{$\chi_4\chi_q$ is trivial}} \frac{D^{z+1}}{4^{z+1}(z+1)} \frac{6}{\pi^2} \prod_{p \mid 4qm} \frac{p}{p+1}
    \\
    &+ O\!\left(q^{\frac{1}{6}} D^{\Re(z) + \half}\!\left(D^{\frac{1}{14}} + |\Im(z)|^{\frac{1}{6}}\right) \big(mqD|\Im(z)|\big)^\eps\right)
    .
  \end{align*}
  If $2 \mid q$, then the factor $\chi_q(4)$ makes the above sum $0$. If $2 \mid m$, then the factor $\chi_m^0(4)$ makes the sum $0$. If $2\nmid q$, then $\chi_4\chi_q$ is trivial iff both $\chi_4$ and $\chi_q$ are individually trivial. It follows that
  \begin{align}
    \nonumber
    &\frac{4^z\chi_m^0(4)}{2\phi(q)} \sum_{\chi_4 \mod 4} \sum_{\chi_q \mod q} \chi_q(4) \overline{\chi}_q(a)
    \one{\text{$\chi_4\chi_q$ is trivial}} \frac{D^{z+1}}{4^{z+1}(z+1)} \frac{6}{\pi^2} \prod_{p \mid 4qm} \frac{p}{p+1}
    \\
    \label{eq:F3_power_sum}
    &\hspace{7cm}=
    \frac{1}{8}\one{(qm,2) = 1} \frac{D^{z+1}}{(z+1)} \frac{6}{\pi^2\phi(q)} \prod_{p \mid 4qm} \frac{p}{p+1}
    .
  \end{align}
  Summing \eqref{eq:F1_power_sum}, \eqref{eq:F2_power_sum}, and \eqref{eq:F3_power_sum} yields \cref{lemma:disc_power_sum} with an error term of \eqref{eq:disc_sum_et1}.
\end{proof}

\hypertarget{proof:disc_power_sum_et2}{}
  \begin{proof}[Proof of \cref{lemma:disc_power_sum} with error \eqref{eq:disc_sum_et2}]
  The case $q = 1$ and $\Re(s) \eqqcolon \sigma \geqslant 0$ is due to Jutila \cite[Lemma 1]{jutila}. The case $q > 1$ odd and $\sigma \geqslant 0$ is due to Stankus \cite[Lemma 1]{stankus}. We will show that
  \begin{enumerate}[label=(\roman*)]
    \item
      the range $\sigma \geqslant 0$ can be extended to $\sigma > -1$, and
    \item
      The conclusion of \cite[Lemma 1]{stankus} holds for even $q$ as well.
  \end{enumerate}

  \textbf{(i).}
  Extending the permissible range of $\sigma$ is a straightforward application of summation by parts \cite[Thm.\ 4.2]{apostol}; viz.\ if $c(n)$ is a sequence supported on the positive integers and $f$ continuously differentiable on the interval $[x,y]$, then
  \begin{align}
    \label{eq:summation_by_parts}
    \sum_{x < n \leqslant y} c(n) f(n) = f(y)\sum_{0 \leqslant n \leqslant y} c(n) \,-\, f(x)\sum_{0 \leqslant n \leqslant x} c(n) \,-\, \int_x^y f'(t)\sum_{0 \leqslant n \leqslant t} c(n)\,dt.
  \end{align}
  This is a special case of integration by parts for Lebesgue--Stieltjes integrals.
  
  Taking $c(d) \coloneqq \one{\text{$d$ is a fundamental discriminant},\, (d,m) = 1,\, d = \ell \mod q}$ and $f(t) = t^s$ in \eqref{eq:summation_by_parts} gives
  \begin{align}
    \label{eq:summation_by_parts_subbed}
    \starsum*_{\substack{x < d \leqslant y \\ (d,m) = 1 \\ d = \ell \mod q}} d^s = y^s\starsum*_{\substack{0 \leqslant d \leqslant y \\ (d,m) = 1 \\ d = \ell \mod q}} 1 \,-\, x^s\starsum*_{\substack{0 \leqslant d \leqslant x \\ (d,m) = 1 \\ d = \ell \mod q}} 1\,-\, \int_x^y st^{s-1}\starsum*_{\substack{0 \leqslant d \leqslant t \\ (d,m) = 1 \\ d = \ell \mod q}} 1\,dt.
  \end{align}

  For brevity, set
  \begin{align*}
    C \coloneqq \frac{\eta_{m,q,\ell}}{1+s} \frac{6}{\pi^2\phi(q)} \prod_{p\mid 2qm} \frac{p}{p+1},
  \end{align*}
  the coefficient appearing in \cref{lemma:disc_power_sum}.
  Define the function $R: \R \to \R$ via
  \begin{align}
    \label{eq:Rdef_partial_summation}
    R(D) \coloneqq \starsum*_{\substack{0 < d < D \\ (d,m) = 1 \\ d = \ell \mod q}} 1 \,-\, CD^{1+s}.
  \end{align}
  
  Rewriting \eqref{eq:summation_by_parts_subbed} in terms of this new notation,
  \begin{align}
    \nonumber
    \starsum*_{\substack{x < d \leqslant y \\ (d,m) = 1 \\ d = \ell \mod q}} d^s
    &= y^s(Cy + R(y)) - x^s(Cx + R(x)) - \int_x^y st^{s-1}(Ct + R(t))\,dt
    \\
    \label{eq:summation_by_parts_after_notation}
    &= \frac{C}{s+1}y^{s+1} - \frac{C}{s+1}x^{s+1} + y^s R(y) - x^s R(x) - \int_x^y st^{s-1} R(t)\,dt.
  \end{align}
  It is clear from the definition \eqref{eq:Rdef_partial_summation} that $R(t) = -Ct$ for $0 \leqslant t < 1$. It follows that
  \begin{align}
    \label{eq:summation_by_parts_small_x}
    \int_x^1 st^{s-1} R(t)\,dt = \frac{Cs}{s+1}(x^{s+1} - 1)
    \quad\quad\text{and}\quad\quad
    x^s R(x) = -Cx^{s+1}
  \end{align}
  if $x < 1$. Hence, if $\sigma > -1$, we may take the limit $x \to 0$ in \eqref{eq:summation_by_parts_after_notation}, obtaining
  \begin{align}
    \label{eq:summation_by_parts_after_limit}
    \starsum*_{\substack{0 < d \leqslant y \\ (d,m) = 1 \\ d = \ell \mod q}} d^s
    &= \frac{C}{s+1}y^{s+1} + y^s R(y) + \frac{Cs}{s+1} - \int_1^y st^{s-1} R(t)\,dt.
  \end{align}

  In the range $t \geqslant 1$, using the bound
  \begin{align}
    \label{eq:Rbound}
    R(t) \ll t^\half \divcount(m) q^{\half + \eps}
  \end{align}
  of \cite[Lemma 1]{jutila} and \cite[Lemma 1]{stankus} gives
  \begin{align}
    \nonumber
    \left| \int_1^y st^{s-1} R(t)\,dt \right|
    &\leqslant \int_1^y |s|t^{\sigma-1} |R(t)|\,dt
    \\
    \label{eq:summation_by_parts_large_x}
    &\ll |s|\divcount(m) q^{\half + \eps} \int_1^y t^{\sigma-1} t^\half\,dt
  \end{align}
  
  Substituting \eqref{eq:Rbound} and \eqref{eq:summation_by_parts_large_x} into \eqref{eq:summation_by_parts_after_limit} gives
  \begin{align*}
    \starsum*_{\substack{0 < d \leqslant y \\ (d,m) = 1 \\ d = \ell \mod q}} d^s
    &= \frac{C}{s+1}\big(y^{s+1} + s\big) + O\!\left((1 + |s|) \divcount(m) q^{\half + \eps} \left(1 + \int_1^y t^{\sigma-\half}\,dt\right)\right).
  \end{align*}  
  
  \textbf{(ii).} To show that the conclusion of \cite[Lemma 1]{stankus} holds for even $q$ as well we follow Stankus's proof. Consider the quantity
  \begin{align*}
        \starsum*_{\substack{0 < d < D \\ (d,m) = 1 \\ d = \ell \mod q}} 1
  \end{align*}
  with $q$ even (and all variables conforming to the hypothesis of \cref{lemma:disc_power_sum}).
  
  \Cref{lemma:disc_power_sum} pertains only in cases where $\ell$ is coprime to $q$. If $q$ is even then any such $\ell$ must be odd. If $d$ is even, then it is impossible to have $d = \ell \mod q$ for such $q$ and $\ell$. For $d$ a fundamental discriminant, recall the classification \eqref{eq:fundamental_discriminant_classification}: either
  \begin{align}
    \nonumber
    &
    \begin{aligned}
    &\text{$d$ is squarefree and $1 \mod 4$,}
    \\
    &\text{$d = 4n$ for some squarefree $n$ which is $2 \mod 4$ (i.e.\ $d = 8n'$, for some odd squarefree $n'$), or}
    \\
    &\text{$d = 4n$ for some squarefree $n$ which is $3 \mod 4$.}
    \end{aligned}
  \end{align}
  Hence, if $q$ is even, for our immediate purposes we need only consider the first of these three cases, where $d = 1 \mod 4$.

  For any positive integer $k$, let $\chi_k^0$ denote the trivial Dirichlet character modulo $k$. Let $\chi_4$ denote the nontrivial character modulo $4$. Then $n = 1 \mod 4$ iff $\thalf(\chi_4^0(n) + \chi_4(n)) = 1$. Hence,
  \begin{align*}
    \starsum*_{\substack{0 < d < D \\ (d,m) = 1 \\ d = \ell \mod q \\ d = 1 \mod 4}} 1
    =
    \sum_{0 < n < D} \mu(n)^2 \,\frac{\chi_4^0(n) + \chi_4(n)}{2} \chi_m^0(n) \frac{1}{\phi(q)}\sum_{\chi_q \mod q} \overline{\chi_q(\ell)} \chi_q(n).
  \end{align*}
  The squarefree indicator function may be written \cite[(2.4)]{MV}
  \begin{align*}
    \mu(n)^2 = \sum_{a^2\mid n}\mu(a).
  \end{align*}
  Substituting,
  \begin{align}
    \nonumber
    &
    \sum_{0 < n < D} \mu(n)^2
    \,\frac{\chi_4^0(n) + \chi_4(n)}{2}
    \chi_m^0(n) \frac{1}{\phi(q)}
    \sum_{\chi_q \mod q} \overline{\chi_q(\ell)} \chi_q(n)
    \\
    \nonumber
    &=
    \frac{1}{2\phi(q)}\sum_{0 < n < D} \sum_{a^2\mid n} \mu(a) (\chi_4^0(n) + \chi_4(n)) \chi_m^0(n) \sum_{\chi_q \mod q} \overline{\chi_q(\ell)} \chi_q(n)
    \\
    \nonumber
    &=
    \frac{1}{2\phi(q)}\sum_{a^2 < D} \sum_{r < D/a^2} \mu(a) (\chi_4^0(a^2r) + \chi_4(a^2r)) \chi_m^0(a^2r) \sum_{\chi_q \mod q} \overline{\chi_q(\ell)} \chi_q(a^2r)
    \\
    \nonumber
    &=
    \frac{1}{2\phi(q)}\sum_{\substack{a^2 < D \\ (a, qm) = 1}} \sum_{\substack{r < D/a^2 \\ (r, qm) = 1}} \mu(a) (1 + \chi_4(r)) \sum_{\chi_q \mod q} \overline{\chi_q(\ell)} \chi_q(a^2r)
    \\
    \label{eq:stankus1}
    &= \frac{1}{2\phi(q)}\sum_{\substack{a^2 < D \\ (a, qm) = 1}} \sum_{\substack{r < D/a^2 \\ (r, qm) = 1}} \mu(a)
    \\
    \label{eq:stankus2}
    &+ \frac{1}{2\phi(q)}\sum_{\substack{a^2 < D \\ (a, qm) = 1}} \sum_{\substack{r < D/a^2 \\ (r, qm) = 1}} \mu(a) \chi_4(r)
    \\
    \label{eq:stankus3}
    &+ \frac{1}{2\phi(q)}\sum_{\substack{a^2 < D \\ (a, qm) = 1}} \sum_{\substack{r < D/a^2 \\ (r, qm) = 1}} \mu(a) \sum_{\substack{\chi_q \mod q \\ \chi_q \neq \chi_q^0}} \overline{\chi_q(\ell)} \chi_q(a^2r)
    \\
    \label{eq:stankus4}
    &+ \frac{1}{2\phi(q)}\sum_{\substack{a^2 < D \\ (a, qm) = 1}} \sum_{\substack{r < D/a^2 \\ (r, qm) = 1}} \mu(a) \chi_4(r) \sum_{\substack{\chi_q \mod q \\ \chi_q \neq \chi_q^0}} \overline{\chi_q(\ell)} \chi_q(a^2r)
    .
  \end{align}

  As is pointed out in \cite[Proof of lemma 1]{stankus},
  \begin{align}
    \nonumber
    \sum_{\substack{0 < n < N \\ (n,k) = 1}} 1 = \frac{\phi(k)}{k}N + O(\divcount(k)).
  \end{align}
  With this, \eqref{eq:stankus1} can be evaluated:
  \begin{align*}
    &\frac{1}{2\phi(q)}\sum_{\substack{a^2 < D \\ (a, qm) = 1}} \sum_{\substack{r < D/a^2 \\ (r, qm) = 1}} \mu(a)
    \\
    &=
    \frac{1}{2\phi(q)}\sum_{\substack{a^2 < D \\ (a, qm) = 1}} \frac{\phi(qm)}{qm}D \frac{\mu(a)}{a^2} + O(\divcount(qm))
    \\
    &=
    \frac{1}{2\phi(q)}\left( \frac{\phi(qm)}{qm}D \left(\frac{1}{\zeta^{(qm)}(2)} - \sum_{\substack{a^2 > D \\ (a, qm) = 1}}  \frac{\mu(a)}{a^2} \right) + \sum_{\substack{a^2 < D \\ (a, qm) = 1}}O(\divcount(qm))\right)
    \\
    &=
    \frac{1}{2\phi(q)}\left( \frac{\phi(qm)}{qm}D \left(\frac{6}{\pi^2}\prod_{p\mid qm} (1 - p^{-2})^{-1} + O\!\left(\frac{\phi(qm)}{qm}D^{-\half}\right)\right) + O\!\left(\frac{\phi(qm)}{qm}D^{\half}\divcount(qm)\right)\right)
    \\
    &=
    \frac{1}{2\phi(q)}\left( \frac{6}{\pi^2}\prod_{p \mid qm} 1 - p^{-1} \prod_{p\mid qm} \frac{1}{1 - p^{-2}}\cdot D  + O\!\left(\frac{\phi(qm)}{qm}D^{\half}\divcount(qm)\right)\right)
    \\
    &=
    \frac{3}{\pi^2\phi(q)}\prod_{p\mid qm} \frac{1}{1 + p^{-1}}\cdot D  + O\!\left(\frac{\phi(qm)\divcount(qm)}{qm\phi(q)}D^{\half}\right)
    .
  \end{align*}
  
  Next we bound \eqref{eq:stankus2}:
  \begin{align*}
    &\frac{1}{2\phi(q)}\sum_{\substack{a^2 < D \\ (a, qm) = 1}} \sum_{\substack{r < D/a^2 \\ (r, qm) = 1}} \mu(a) \chi_4(r)
    \\
    &= \frac{1}{2\phi(q)}\sum_{\substack{a^2 < D \\ (a, qm) = 1}} \mu(a) \left(\sum_{\substack{r < D/a^2 \\ (r, qm) = 1 \\ r = 1 \mod 4}} 1 - \sum_{\substack{r < D/a^2 \\ (r, qm) = 1 \\ r = 3 \mod 4}} 1\right)
    \\
    &= \frac{1}{2\phi(q)}\sum_{\substack{a^2 < D \\ (a, qm) = 1}} \mu(a) \left(\frac{\phi(qm)}{2qm}\frac{D}{a^2} + O(\divcount(2qm)) - \frac{\phi(qm)}{2qm}\frac{D}{a^2} + O(\divcount(2qm))\right)
    \\
    &= \frac{1}{2\phi(q)}\sum_{\substack{a^2 < D \\ (a, qm) = 1}} \mu(a) O(\divcount(2qm))
    \\
    &\ll \frac{1}{2\phi(q)} \left(\frac{\phi(qm)}{qm}D^\half\divcount(2qm) + \divcount(2qm)^2\right)
    \\
    &\ll \frac{\phi(qm)\divcount(2qm)}{qm\phi(q)}D^\half + \frac{\divcount(2qm)^2}{\phi(q)}
    .
  \end{align*}

  We now bound \eqref{eq:stankus3}:
  \begin{align*}
    &\frac{1}{2\phi(q)}\sum_{\substack{a^2 < D \\ (a, qm) = 1}} \sum_{\substack{r < D/a^2 \\ (r, qm) = 1}} \mu(a) \sum_{\substack{\chi_q \mod q \\ \chi_q \neq \chi_q^0}} \overline{\chi_q(\ell)} \chi_q(a^2r)
    \\
    &=
    \frac{1}{2\phi(q)} \sum_{\substack{\chi_q \mod q \\ \chi_q \neq \chi_q^0}} \overline{\chi_q(\ell)} \sum_{\substack{a^2 < D \\ (a, qm) = 1}} \mu(a) \chi_q(a^2) \sum_{\substack{r < D/a^2 \\ (r, qm) = 1}} \chi_q(r)
    .
    \shortintertext{Applying P\'olya--Vinogradov \cite[Thm.\ 9.18]{MV},}
    &\ll
    \frac{1}{2\phi(q)} \sum_{\substack{\chi_q \mod q \\ \chi_q \neq \chi_q^0}} \overline{\chi_q(\ell)} \sum_{\substack{a^2 < D \\ (a, qm) = 1}} \mu(a) \chi_q(a^2) \big(\sqrt{q}\log q + O(\divcount(m))\big) 
    \\
    &\ll
    \frac{1}{2\phi(q)} \sum_{\substack{\chi_q \mod q \\ \chi_q \neq \chi_q^0}} \overline{\chi_q(\ell)} \frac{\phi(qm)}{qm} D^\half \big(\sqrt{q}\log q + O(\divcount(m))\big)
    \\
    &\ll
    \frac{\phi(qm)}{qm} D^\half \big(\sqrt{q}\log q + O(\divcount(m))\big)
    .
  \end{align*}

  We analyze \eqref{eq:stankus4} in much the same way. There is a key difference however: writing
  \begin{align*}
    &\frac{1}{2\phi(q)}\sum_{\substack{a^2 < D \\ (a, qm) = 1}} \sum_{\substack{r < D/a^2 \\ (r, qm) = 1}} \mu(a) \chi_4(r) \sum_{\substack{\chi_q \mod q \\ \chi_q \neq \chi_q^0}} \overline{\chi_q(\ell)} \chi_q(a^2r)
    \\
    &=
    \frac{1}{2\phi(q)} \sum_{\substack{\chi_q \mod q \\ \chi_q \neq \chi_q^0}} \overline{\chi_q(\ell)} \sum_{\substack{a^2 < D \\ (a, qm) = 1}} \mu(a) \chi_q(a^2) \sum_{\substack{r < D/a^2 \\ (r, qm) = 1}} \chi_q(r)\chi_4(r),
  \end{align*}
  if $q$ is even, then it may be that $\chi_q \chi_4$ is trivial, preventing us from applying P\'olya--Vinogradov. This happens only when $4 \mid q$, and for exactly one character in such cases: the lift of $\chi_4$ to the modulus $q$. For this particular character, the corresponding term in the sum above is
  \begin{align*}
    &\frac{\chi_4(\ell)}{2\phi(q)} \sum_{\substack{a^2 < D \\ (a, qm) = 1}} \mu(a) \left(\frac{\phi(qm)}{qm}\frac{D}{a^2} + O(\divcount(qm))\right)
    \\
    &=
    \frac{\chi_4(\ell)}{2\phi(q)} \frac{\phi(qm)}{qm} D \left(\frac{1}{\zeta^{(qm)}(2)} - \sum_{\substack{a^2 > D \\ (a, qm) = 1}} \frac{\mu(a)}{a^2}\right) + O\!\left(\frac{\phi(qm)\divcount(qm)}{qm\phi(q)}D^\half\right)
    \\
    &=
    \frac{3\chi_4(\ell)}{\pi^2\phi(q)}\prod_{p\mid qm} \frac{1}{1 + p^{-1}}\cdot D + O\!\left(\frac{\phi(qm)\divcount(qm)}{qm\phi(q)}D^\half\right)
    .
  \end{align*}
  The other characters modulo $q$ have their contribution bounded in the same way as above.
\end{proof}

\hypertarget{proof:stankus_2nd_moment}{}
\begin{proof}[Proof of \cref{lemma:stankus_2nd_moment}]
  \begin{align}
    \nonumber
    \sum_{\substack{n < N \\ n \neq \square \\ (n,q) = 1}} \left|\starsum*_{\substack{0 < d < D \\ d = \ell \mod q }} \chi_d(n)\right|^2
    &=
    \sum_{k \geqslant 0} \sum_{\substack{m < N/2^k \\ 2^km \neq \square \\ (m,2q) = 1}} \left|\starsum*_{\substack{0 < d < D \\ d = \ell \mod q }} \chi_d(2^km)\right|^2
    \\
    &=
    \label{eq:discs_n_odd}
    \sum_{\substack{m < N \\ m \neq \square \\ (m,2q) = 1}} \left|\starsum*_{\substack{0 < d < D \\ d = \ell \mod q }} \chi_d(m)\right|^2
    \\
    &+
    \label{eq:discs_n_even}
    \sum_{k \geqslant 1} \sum_{\substack{m < N/2^k \\ 2^km \neq \square \\ (m,2q) = 1}} \left|\starsum*_{\substack{0 < d < D \\ d = \ell \mod q }} \chi_d(2^k) \chi_d(m)\right|^2
    .
  \end{align}

  By \cite[Lemma 5]{stankus},
  \begin{align}
    \label{eq:discs_n_odd_eval}
    \eqref{eq:discs_n_odd} \ll qND(\log D)^4.
  \end{align}

  Looking at \eqref{eq:discs_n_even},
  \begin{align}
    \nonumber
    \sum_{k \geqslant 1} \sum_{\substack{m < N/2^k \\ 2^km \neq \square \\ (m,2q) = 1}} \left|\starsum*_{\substack{0 < d < D \\ d = \ell \mod q }} \chi_d(2^k) \chi_d(m)\right|^2
    ={}&
    \sum_{\substack{k \geqslant 1 \\ \text{$k$ even}}} \sum_{\substack{m < N/2^k \\ 2^km \neq \square \\ (m,2q) = 1}} \left|\starsum*_{\substack{0 < d < D \\ d = \ell \mod q \\ (d,2) = 1}} \chi_d(m)\right|^2
    \\
    \nonumber
    &
    +
    \sum_{\substack{k \geqslant 1 \\ \text{$k$ odd}}} \sum_{\substack{m < N/2^k \\ 2^km \neq \square \\ (m,2q) = 1}} \left|\starsum*_{\substack{0 < d < D \\ d = \ell \mod q}} \chi_d(2) \chi_d(m)\right|^2
    \\
    \label{eq:discs_n_even_k_even}
    \ll{}&
    \log N \cdot \sum_{\substack{m < N/2 \\ m \neq \square \\ (m,2q) = 1}} \left|\starsum*_{\substack{0 < d < D \\ d = \ell \mod q \\ (d,2) = 1}} \chi_d(m)\right|^2
    \\
    \label{eq:discs_n_even_k_odd}
    &+
    \log N \cdot \sum_{\substack{m < N/2 \\ 2m \neq \square \\ (m,2q) = 1}} \left|\starsum*_{\substack{0 < d < D \\ d = \ell \mod q}} \chi_d(2) \chi_d(m)\right|^2
    .
  \end{align}

  Looking at the sum in \eqref{eq:discs_n_even_k_odd},
  \begin{align}
    \nonumber
    \sum_{\substack{m < N/2 \\ 2m \neq \square \\ (m,2q) = 1}} \left|\starsum*_{\substack{0 < d < D \\ d = \ell \mod q}} \chi_d(2) \chi_d(m)\right|^2
    &=
    \sum_{\substack{m < N/2 \\ 2m \neq \square \\ (m,2q) = 1}} \left|\starsum*_{\substack{0 < d < D \\ d = \ell \mod q \\ \chi_d(2) = 1}} \chi_d(m) - \starsum*_{\substack{0 < d < D \\ d = \ell \mod q \\ \chi_d(2) = -1}} \chi_d(m)\right|^2
    \\
    \label{eq:discs_n_even_k_odd_m_square}
    &=
    \sum_{\substack{m < N/2 \\ m = \square \\ (m,2q) = 1}} \left|\starsum*_{\substack{0 < d < D \\ d = \ell \mod q \\ \chi_d(2) = 1}} \chi_d(m) - \starsum*_{\substack{0 < d < D \\ d = \ell \mod q \\ \chi_d(2) = -1}} \chi_d(m)\right|^2
    \\
    \label{eq:discs_n_even_k_odd_m_nonsquare}
    &
    +
    \sum_{\substack{m < N/2 \\ m \neq \square \\ (m,2q) = 1}} \left|\starsum*_{\substack{0 < d < D \\ d = \ell \mod q \\ \chi_d(2) = 1}} \chi_d(m) - \starsum*_{\substack{0 < d < D \\ d = \ell \mod q \\ \chi_d(2) = -1}} \chi_d(m)\right|^2
    .
  \end{align}
  The second equality immediately above uses the fact that $2m \neq \square$ is automatic if $m$ is odd.

  Recall the classification of fundamental discriminants \eqref{eq:fundamental_discriminant_classification}. Looking at \eqref{eq:discs_n_even_k_odd_m_square},
  \begin{align}
    \label{eq:discs_n_even_k_odd_m_square_eval}
    \sum_{\substack{m < N/2 \\ m = \square \\ (m,2q) = 1}} \left|\starsum*_{\substack{0 < d < D \\ d = \ell \mod q \\ \chi_d(2) = 1}} \chi_d(m) - \starsum*_{\substack{0 < d < D \\ d = \ell \mod q \\ \chi_d(2) = -1}} \chi_d(m) \right|^2
    &=
    \sum_{\substack{m < N/2 \\ m = \square \\ (m,2q) = 1}} \left|\starsum*_{\substack{0 < d < D \\ d = \ell \mod q \\ d = 1 \mod 8 \\ (d,m) = 1}} 1 - \starsum*_{\substack{0 < d < D \\ d = \ell \mod q \\ d = 5 \mod 8 \\ (d,m) = 1}} 1 \right|^2
    .
  \end{align}

  By \cref{lemma:disc_power_sum}, with the more precise error terms \eqref{eq:disc_sum_et1} and \eqref{eq:disc_sum_et2},
  \begin{align}
    \label{eq:discs_count_1mod8}
    &\starsum*_{\substack{0 < d < D \\ d = \ell \mod q \\ d = 1 \mod 8 \\ (d,m) = 1}} 1
    = CD
    + O\!\left(
    \min\!\left\{
    D^{\frac{4}{7}} q^{\frac{1}{6} + \eps}
    ,
    D^{\half} q^{\frac{1}{2} + \eps}
    \right\}
    \right)
    \shortintertext{and}
    \label{eq:discs_count_5mod8}
    &\starsum*_{\substack{0 < d < D \\ d = \ell \mod q \\ d = 5 \mod 8 \\ (d,m) = 1}} 1
    = CD
    + O\!\left(
    \min\!\left\{
    D^{\frac{4}{7}} q^{\frac{1}{6} + \eps}
    ,
    D^{\half} q^{\frac{1}{2} + \eps}
    \right\}
    \right)
    ,
  \end{align}
  where $C$ is the coefficient from \cref{lemma:disc_power_sum}. 
  
  Substituting \eqref{eq:discs_count_1mod8} and \eqref{eq:discs_count_5mod8} into \eqref{eq:discs_n_even_k_odd_m_square_eval},
  \begin{align}
    \nonumber
    \eqref{eq:discs_n_even_k_odd_m_square}
    &=
    \sum_{\substack{m < N/2 \\ m = \square \\ (m,2q) = 1}} \left|\starsum*_{\substack{0 < d < D \\ d = \ell \mod q \\ \chi_d(2) = 1}} \chi_d(m) - \starsum*_{\substack{0 < d < D \\ d = \ell \mod q \\ \chi_d(2) = -1}} \chi_d(m) \right|^2
    \\
    \nonumber
    &=
    \sum_{\substack{m < N/2 \\ m = \square \\ (m,2q) = 1}}
    \left| CD
    + O\!\left(
    \min\!\left\{
    D^{\frac{4}{7}} q^{\frac{1}{6} + \eps}
    ,
    D^{\half} q^{\frac{1}{2} + \eps}
    \right\}
    \right)
    - CD
    + O\!\left(
    \min\!\left\{
    D^{\frac{4}{7}} q^{\frac{1}{6} + \eps}
    ,
    D^{\half} q^{\frac{1}{2} + \eps}
    \right\}
    \right)
    \right|^2
    \\
    \nonumber
    &\ll
    \sum_{\substack{m < N/2 \\ m = \square \\ (m,2q) = 1}}
    + O\!\left(
    \min\!\left\{
    D^{\frac{8}{7}} q^{\frac{1}{3} + \eps}
    ,
    D q^{1 + \eps}
    \right\}
    \right)
    \\
    \label{eq:discs_n_even_k_odd_m_square_eval_2}
    &\ll N^\half
    \min\!\left\{
    D^{\frac{8}{7}} q^{\frac{1}{3} + \eps}
    ,
    D q^{1 + \eps}
    \right\}
    .
  \end{align}

  Looking next at \eqref{eq:discs_n_even_k_odd_m_nonsquare},
  \begin{align}
    \nonumber
    \eqref{eq:discs_n_even_k_odd_m_nonsquare}
    &=
    \sum_{\substack{m < N/2 \\ m \neq \square \\ (m,2q) = 1}} \left|\starsum*_{\substack{0 < d < D \\ d = \ell \mod q \\ \chi_d(2) = 1}} \chi_d(m) - \starsum*_{\substack{0 < d < D \\ d = \ell \mod q \\ \chi_d(2) = -1}} \chi_d(m)\right|^2
    \\
    \label{eq:discs_n_even_k_odd_m_nonsquare_mod8_split}
    &\ll
    \sum_{\substack{m < N/2 \\ m \neq \square \\ (m,2q) = 1}} \left|\starsum*_{\substack{0 < d < D \\ d = \ell \mod q \\ \chi_d(2) = 1}} \chi_d(m)\right|^2
    +
    \sum_{\substack{m < N/2 \\ m \neq \square \\ (m,2q) = 1}} \left|\starsum*_{\substack{0 < d < D \\ d = \ell \mod q \\ \chi_d(2) = -1}} \chi_d(m)\right|^2
  \end{align}
  
  Fainleib and Saparniyazov prove the following
  \begin{theorem}[{Fainleib--Saparniyazov \cite[\foreignlanguage{russian}{Теорема} 1]{FS}}] 
    \label{thm:FS}
    Let $M, K \in \Z_{\geqslant 2}$ and $q, \ell \in \Z_{>0}$ with $(\ell, q) = 1$. Let $a(k)$ be a sequence of complex numbers, and set $b(k) \coloneqq \sum_{r\mid k} \mu\big(\tfrac{k}{r}\big) a_r$.
    \begin{align*}
      \sum_{\substack{m < M \\ m \neq \square \\ (m,2q) = 1}} \left|\sum_{\substack{k < K \\ k = \ell \mod q}} a(k) \left(\frac{k}{m}\right)\right|^2
      \ll
      qMK(\log K)^2 \left(\sum_{k < K} \frac{|b(k)|}{\sqrt{k}}\right)^2
      .
    \end{align*}
  \end{theorem}
  Let $a(k) = \mu(k)^2$ be the indicator function of squarefree integers. For this choice,
  \begin{align*}
    b(k)
    =
    \sum_{r\mid k} \mu\!\left(\tfrac{k}{r}\right) \mu(r)^2
    =
    \one{\sqrt{k} \in \Z}\mu(\sqrt{k}).
  \end{align*}
  
  Consider the first term of \eqref{eq:discs_n_even_k_odd_m_nonsquare_mod8_split},
  \begin{align*}
    \sum_{\substack{m < N/2 \\ m \neq \square \\ (m,2q) = 1}} \left|\starsum*_{\substack{0 < d < D \\ d = \ell \mod q \\ \chi_d(2) = 1}} \chi_d(m)\right|^2
    =
    \sum_{\substack{m < N/2 \\ m \neq \square \\ (m,2q) = 1}} \left|\sum_{\substack{0 < k < D \\ k = \ell \mod q \\ n = 1 \mod 8}} a(k) \left(\frac{k}{m}\right)\right|^2
    .
  \end{align*}
  Applying \cref{thm:FS},
  \begin{align}
    \nonumber
    \sum_{\substack{m < N/2 \\ m \neq \square \\ (m,2q) = 1}} \left|\starsum*_{\substack{0 < d < D \\ d = \ell \mod q \\ \chi_d(2) = 1}} \chi_d(m)\right|^2
    &\ll
    qND(\log D)^2\left(\sum_{0 < k < D} \one{\sqrt{k} \in \Z} \frac{|\mu(\sqrt{k})|}{\sqrt{k}}\right)^2
    \\
    \nonumber
    &\ll
    qND(\log D)^2\left(\sum_{0 < r < D^\half}\frac{1}{r}\right)^2
    \\
    \label{eq:discs_n_even_k_odd_m_nonsquare_1mod8_bound}
    &\ll
    qND(\log D)^4
    .
  \end{align}

  The second term of \eqref{eq:discs_n_even_k_odd_m_nonsquare_mod8_split} is bounded in the same way:
  \begin{align}
    \label{eq:discs_n_even_k_odd_m_nonsquare_5mod8_bound}
    \sum_{\substack{m < N/2 \\ m \neq \square \\ (m,2q) = 1}} \left|\starsum*_{\substack{0 < d < D \\ d = \ell \mod q \\ \chi_d(2) = -1}} \chi_d(m)\right|^2
    &\ll
    qND(\log D)^4
    .
  \end{align}

  Substituting \eqref{eq:discs_n_even_k_odd_m_nonsquare_1mod8_bound} and \eqref{eq:discs_n_even_k_odd_m_nonsquare_5mod8_bound} into \eqref{eq:discs_n_even_k_odd_m_nonsquare_mod8_split} gives
  \begin{align}
    \label{eq:discs_n_even_k_odd_m_nonsquare_1mod4_bound}
    \sum_{\substack{m < N/2 \\ m \neq \square \\ (m,2q) = 1}} \left|\starsum*_{\substack{0 < d < D \\ d = \ell \mod q \\ d = 1 \mod 4}} \chi_d(m)\right|^2
    &\ll
    qND(\log D)^4
    .
  \end{align}

  Moreover, \eqref{eq:discs_n_even_k_odd_m_nonsquare_1mod4_bound} bounds \eqref{eq:discs_n_even_k_even}:
  \begin{align}
    \nonumber
    \eqref{eq:discs_n_even_k_even}
    &=
    \log N \cdot \sum_{\substack{m < N/2 \\ m \neq \square \\ (m,2q) = 1}} \left|\starsum*_{\substack{0 < d < D \\ d = \ell \mod q \\ (d,2) = 1}} \chi_d(m)\right|^2
    \\
    \nonumber
    &=
    \log N \cdot \sum_{\substack{m < N/2 \\ m \neq \square \\ (m,2q) = 1}} \left|\starsum*_{\substack{0 < d < D \\ d = \ell \mod q \\ d = 1 \mod 4}} \chi_d(m)\right|^2
    \\
    \label{eq:discs_n_even_k_even_bound}
    &\ll
    qND(\log D)^4\log N
    .
  \end{align}

  Combining \eqref{eq:discs_n_odd_eval}, \eqref{eq:discs_n_even_k_odd_m_square_eval_2}, \eqref{eq:discs_n_even_k_odd_m_nonsquare_1mod4_bound}, and \eqref{eq:discs_n_even_k_even_bound} proves \cref{lemma:stankus_2nd_moment}.
\end{proof}

\subsection{Analyzing terms in the approximate functional equation averaged over a family of quadratic twists}
\label{subsec:proofs_average_afe}

\hypertarget{proof:t1s_approx}{}
\begin{proof}[Proof of \cref{lemma:t1s_approx}]
  Using the assumption $\sigma > \thalf + \theta$ to ensure absolute convergence,
  \begin{align}
    \Favg \sum_{n=1}^\infty \frac{a_{\rep}(n^2)\chi_d(n^2)}{n^{2s}} \Vpsq
    \label{eq:t1s_approx_mt1}
    &= \sum_{n=1}^\infty \frac{a_{\rep}(n^2)}{n^{2s}} \Favg \chi_d(n^2)
    \\
    \label{eq:t1s_approx_et1}
    &+ \Favg \sum_{n=1}^\infty \frac{a_{\rep}(n^2)\chi_d(n^2)}{n^{2s}} \left(\Vpsq - 1\right)
    .
  \end{align}

  The error term \eqref{eq:t1s_approx_et1} is bounded using \cref{lemma:V_approx}:
  \begin{align}
    \nonumber
    \eqref{eq:t1s_approx_et1}
    &\ll \sum_{n^2 > \mupnp} n^{2\theta - 2\sigma} \;\epsfac
    \\
    \label{eq:t1s_approx_et1_eval}
    &\ll
    \mup^{\half - \sigma + \theta} \epsfac
    .
  \end{align}

  Applying \cref{lemma:disc_power_sum} to the main term \eqref{eq:t1s_approx_mt1},
  \begin{align}
    \nonumber
    \eqref{eq:t1s_approx_mt1}
    &=
    \sum_{n=1}^\infty \frac{a_{\rep}(n^2)}{n^{2s}} \frac{1}{\#\cF} \starsum*_{\substack{D_0 < d < D \\ d = \ell \mod \qrep \\ (d,n) = 1}} 1
    \\
    \label{eq:t1s_approx_mt2_et2}
    &=
    \sum_{n=1}^\infty \frac{a_{\rep}(n^2)}{n^{2s}} \frac{1}{\#\cF} \left(\Delta D\frac{6\Eta_{\cF}}{\pi^2\phi(\qrep)} \prod_{p\mid 2\qrep n} \frac{p}{p+1} + O\!\left(D^\half \qet (n\qrep D)^\eps \right)\right)
    ,
  \end{align}
  where $\Eta_{\cF}$ is as defined in \eqref{eq:Etadef}.

  We first process the error term in \eqref{eq:t1s_approx_mt2_et2}:
  \begin{align}
    \nonumber
    \sum_{n=1}^\infty \frac{a_{\rep}(n^2)}{n^{2s}} \frac{1}{\#\cF} \,O\!\left(D^\half \qet (n\qrep D)^\eps \right)
    &\ll
    \sum_{n=1}^\infty n^{2\theta - 2\sigma + \eps} \,\frac{D^{\half}}{\#\cF} \qet (qD)^\eps
    \\
    \label{eq:t1s_approx_et2_eval}
    &\ll
    \frac{D^{\half}}{\#\cF} \qet (qD)^\eps
    .
  \end{align}

  Next we apply \cref{lemma:Fsize} to the main term of \eqref{eq:t1s_approx_mt2_et2}, obtaining
  \begin{align}
    \nonumber
    &\sum_{n=1}^\infty \frac{a_{\rep}(n^2)}{n^{2s}}
    \frac{\Delta D}{\#\cF} \frac{6\Eta_{\cF}}{\pi^2\phi(\qrep)} \prod_{p\mid 2\qrep n} \frac{p}{p+1}
    \\
    \nonumber
    &=
    \sum_{n=1}^\infty \frac{a_{\rep}(n^2)}{n^{2s}} \frac{1}{\#\cF}
    \left(\#\cF + O\!\left(D^\half \qet (\qrep D)^\eps \right)\right)
    \left(\frac{6\Eta_{\cF}}{\pi^2\phi(\qrep)} \prod_{p\mid 2\qrep} \frac{p}{p+1}\right)^{-1}
    \frac{6\Eta_{\cF}}{\pi^2\phi(\qrep)} \prod_{p\mid 2\qrep n} \frac{p}{p+1}
    \\
    \label{eq:t1s_approx_mt3}
    &=
    \sum_{n=1}^\infty \frac{a_{\rep}(n^2)}{n^{2s}} \prod_{p\mid \frac{n}{(2\qrep, n)}} \frac{p}{p+1}
    \\
    \label{eq:t1s_approx_et3}
    &+
    \sum_{n=1}^\infty \frac{a_{\rep}(n^2)}{n^{2s}} \;O\!\left(\frac{D^\half}{\#\cF} \qet (n\qrep D)^\eps\right)
    .
  \end{align}

  \Cref{lemma:t1s_approx}'s main term is \eqref{eq:t1s_approx_mt3}.
  The Dirichlet series in \eqref{eq:t1s_approx_et3} converges absolutely, per the assumption $\sigma > \thalf + \theta$, resulting in a copy of the error term \eqref{eq:t1s_approx_et2_eval}. Collecting \eqref{eq:t1s_approx_et1_eval}, \eqref{eq:t1s_approx_et2_eval}, and \eqref{eq:t1s_approx_mt3} completes the proof.
\end{proof}

\hypertarget{proof:t2s_approx_step1}{}
\begin{proof}[Proof of \cref{lemma:t2s_approx_step1}]
  We have the decomposition
  \begin{align}
    \nonumber
    &\Favg \sum_{n=1}^\infty \frac{a_{\bar\rep}(n^2)\chi_d(n^2)}{n^{2-2s}} \Vmsq
    \\
    \nonumber
    &= \sum_{n=1}^\infty \frac{a_{\bar\rep}(n^2)}{n^{2-2s}} \VmDsq \Favg \chi_d(n^2)
    \\
    \label{eq:t2s_approx_step1_1}
    &+ \Favg \sum_{n=1}^\infty \frac{a_{\bar\rep}(n^2)\chi_d(n^2)}{n^{2-2s}} \left(\Vmsq - \VmDsq\right)
    .
  \end{align}

  We bound the error term \eqref{eq:t2s_approx_step1_1} using \cref{lemma:V_deriv_approx}:
  \begin{align}
    \nonumber
    &\Favg \sum_{n=1}^\infty \frac{a_{\bar\rep}(n^2)\chi_d(n^2)}{n^{2-2s}} \left(\Vmsq - \VmDsq\right)
    \\
    \nonumber
    &\ll \frac{\Delta D}{D} \Favg \sum_{n=1}^\infty \frac{|a_{\bar\rep}(n^2)\chi_d(n^2)|}{n^{2-2\sigma}}
    \left(\frac{\xi^{-1} \sqrt{QD^\repdim}}{n}\right)^{\!\sigma_w}
    \\
    \nonumber
    &\ll \frac{\Delta D}{D} \sum_{n^2 < \mum^{1+\eps}} \frac{n^{2\theta}}{n^{2-2\sigma}}
    \\
    \nonumber
    &\ll \frac{\#\cF}{D} \left(1 + \mum^{\sigma - \half + \theta}\right) \epsfac
    .
  \end{align}
\end{proof}

\hypertarget{proof:t2s_approx_step2}{}
\begin{proof}[Proof of \cref{lemma:t2s_approx_step2}]
  Write
  \begin{align}
    \nonumber
    &\sum_{n=1}^\infty \frac{a_{\bar\rep}(n^2)}{n^{2-2s}} \VmDsq \Favg \chi_d(n^2) \left(\frac{\pi^\repdim \x}{\qrep d^\repdim}\right)^{\!s - \half}
    \\
    \label{eq:t2s_approx_step2_1}
    &\hspace{2cm}
    = \sum_{n=1}^\infty \frac{a_{\bar\rep}(n^2)}{n^{2-2s}} \VmDsq \Favg \chi_d(n^2)d^{-\repdim(s-\half)} \left(\frac{\pi^\repdim \x}{\qrep}\right)^{\!s - \half}
    .
  \end{align}
  We will apply \cref{lemma:disc_power_sum} momentarily. Recall the definitions \eqref{eq:etadef} and \eqref{eq:Etadef},
  \begin{align*}
    \eta_{m,q,a} &\coloneqq \thalf\one{4 \nmid q} + \one{\text{$4\mid q$ and $a = 1\mod 4$}} + \tfrac{1}{4}\one{(mq,2) = 1}
    \\
    \Eta_{\cF} &\coloneqq 
    \eta_{\qrep,\qrep,\ell}
    .
  \end{align*}
  Recall also that we write $|s_*| \coloneqq \max_j\{|s+\kappa_j|, |1-s+\bar\kappa_j|, 1\}$.

  Applying \cref{lemma:disc_power_sum} to \eqref{eq:t2s_approx_step2_1},
  \begin{align}
    \nonumber
    \eqref{eq:t2s_approx_step2_1}
    &= \sum_{n=1}^\infty \frac{a_{\bar\rep}(n^2)}{n^{2-2s}} \VmDsq \frac{1}{\#\cF} \starsum*_{\substack{D_0 < d < D \\ d = \ell \mod \qrep \\ (d,n) = 1}} d^{-\repdim(s-\half)} \left(\frac{\pi^\repdim \x}{\qrep}\right)^{\!s - \half}
    \\
    \nonumber
    &= \sum_{n=1}^\infty \frac{a_{\bar\rep}(n^2)}{n^{2-2s}} \VmDsq \frac{1}{\#\cF}
    \Bigg[
    \frac{6\Eta_{\cF}}{\pi^2\phi(q)} \prod_{p\mid 2qn}\frac{p}{p+1}
    \left(\frac{D^{1 - r(s-\half)}}{1 - r(s-\half)} - \frac{D_0^{1 - r(s-\half)}}{1 - r(s-\half)}\right)
    \\
    \nonumber
    &\hspace{5.5cm}
    + O\!\left(D^{\half - r(\sigma-\half)}\qrep^{\frac{1}{6}} \left(D^{\frac{1}{14}} + (r|s_*|)^{\frac{1}{6}}\right) (nQD)^\eps \right)\Bigg]
    \left(\frac{\pi^\repdim \x}{\qrep}\right)^{\!s - \half}
    \\
    \nonumber
    &= \sum_{n=1}^\infty \frac{a_{\bar\rep}(n^2)}{n^{2-2s}} \VmDsq \frac{6\Eta_{\cF} D}{\pi^2\phi(q)\#\cF} \prod_{p\mid 2qn}\frac{p}{p+1}
    \frac{1 - \left(\frac{D_0}{D}\right)^{1 - r(s-\half)}}{1 - r(s-\half)} \xD^{\!s - \half}
    \\
    \label{eq:t2s_approx_step2_2}
    &+ \sum_{n=1}^\infty \frac{a_{\bar\rep}(n^2)}{n^{2-2s}} \VmDsq \frac{1}{\#\cF} O\!\left(D^{\half}\qrep^{\frac{1}{6}} \left(D^{\frac{1}{14}} + (r|s_*|)^{\frac{1}{6}}\right) (nQD)^\eps \right) \xD^{\!\sigma - \half}
    .
  \end{align}
  The error term \eqref{eq:t2s_approx_step2_2} can be bounded in a familiar way:
  \begin{align}
    \nonumber
    \eqref{eq:t2s_approx_step2_2}
    &\ll
    \sum_{n^2 < \mum^{1+\eps}} \frac{n^{2\theta}}{n^{2-2\sigma}} \frac{\qrep^{\frac{1}{6}} D^\half}{\#\cF} \left(D^{\frac{1}{14}} + (r|s_*|)^{\frac{1}{6}}\right) (nQD)^\eps \xD^{\!\sigma - \half}
    \\
    \nonumber
    &\ll
    \left(1 + \mum^{\sigma - \half + \theta}\right) \frac{\qrep^{\frac{1}{6}} D^\half}{\#\cF} \left(D^{\frac{1}{14}} + (r|s_*|)^{\frac{1}{6}}\right) (QD)^\eps \xD^{\!\sigma - \half}
    .
    \qedhere
  \end{align}
\end{proof}

\hypertarget{proof:t2s_approx_step3}{}
\begin{proof}[Proof of \cref{lemma:t2s_approx_step3}]
  \Cref{lemma:Fsize} implies
  \begin{align*}
    \frac{\Delta D}{\#\cF} \frac{6\Eta_{\cF}}{\pi^2\phi(q)} \prod_{p\mid 2q} \frac{p}{p+1} = 1 + O\!\left(\#\cF^{-1} 
    D^\half \qet
    (\qrep D)^\eps\right)
    .
  \end{align*}
  Substituting into the left hand side of \cref{lemma:t2s_approx_step3},
  \begin{align}
    \nonumber
    &\sum_{n=1}^\infty
    \frac{a_{\bar\rep}(n^2)}{n^{2-2s}}
    \prod_{p \mid 2qn} \frac{1}{1 + p^{-1}} \VmDsq \frac{6\Eta_{\cF}\Delta D}{\pi^2 \phi(\qrep)\#\cF} \frac{1 - \left(1 - \frac{\Delta D}{D}\right)^{1 - r(s-\half)}}{(1 - r(s-\half))\frac{\Delta D}{D}}
    \\
    \label{eq:t2s_size3_1}
    &=
    \sum_{n=1}^\infty
    \frac{a_{\bar\rep}(n^2)}{n^{2-2s}}
    \frac{\prod_{p \mid 2qn} \frac{p}{p+1}}{\prod_{p \mid 2q} \frac{p}{p+1}} \VmDsq \frac{1 - \left(1 - \frac{\Delta D}{D}\right)^{1 - r(s-\half)}}{(1 - r(s-\half))\frac{\Delta D}{D}}
    \\
    \label{eq:t2s_size3_2}
    &+ \sum_{n=1}^\infty \frac{a_{\bar\rep}(n^2)}{n^{2-2s}}
    \frac{\prod_{p \mid 2qn} \frac{p}{p+1}}{\prod_{p \mid 2q} \frac{p}{p+1}} \VmDsq O\!\left(\#\cF^{-1}
    D^\half \qet
    (\qrep D)^\eps\right) \frac{1 - \left(1 - \frac{\Delta D}{D}\right)^{1 - r(s-\half)}}{(1 - r(s-\half))\frac{\Delta D}{D}}
    .
  \end{align}
  Upon noting that
  \begin{align*}
    \frac{\prod_{p \mid 2qn} \frac{p}{p+1}}{\prod_{p \mid 2q} \frac{p}{p+1}}
    =
    \prod_{p \mid \frac{n}{(n,2q)}} \frac{p}{p+1}
    ,
  \end{align*}
  we recover the main term of \cref{lemma:t2s_approx_step3} in \eqref{eq:t2s_size3_1}.

  To bound the error term \eqref{eq:t2s_size3_2}, let's look at the factor $\frac{1 - \left(1 - \frac{\Delta D}{D}\right)^{1 - r(s-\half)}}{(1 - r(s-\half))\frac{\Delta D}{D}}$ more closely.
  When $r|s_*|\Delta \ll D$, a first order series expansion gives
  \begin{align}
    \nonumber
    \frac{1 - \left(1 - \frac{\Delta D}{D}\right)^{1 - r(s-\half)}}{(1 - r(s-\half))\frac{\Delta D}{D}}
    &=
    \frac{1 - \left(1 - \big(1 - r(s-\half)\big)\frac{\Delta D}{D} + O\!\left(\left(\frac{r|s_*|\Delta D}{D}\right)\right)^2\right)}{(1 - r(s-\half))\frac{\Delta D}{D}}
    \\
    \label{eq:t2s_size3_3}
    &= 1 + O\!\left(\frac{r|s_*|\Delta D}{D}\right)
  \end{align}
  In the complementary regime, the numerator is bounded as
  \begin{align*}
    1 - \left(1 - \frac{\Delta D}{D}\right)^{1 - r(s-\half)}
    \ll
    1 + \exp\!\left(\big(1 - r(\sigma - \thalf)\big)\frac{\Delta D}{D}\right)
    ,
  \end{align*}  
  so
  \begin{align}
    \nonumber
    \frac{1 - \left(1 - \frac{\Delta D}{D}\right)^{1 - r(s-\half)}}{(1 - r(s-\half))\frac{\Delta D}{D}}
    &\ll
    \frac{D}{r|s_*|\Delta D}\left(1 + \exp\!\left(\big(1 - r(\sigma - \thalf)\big)\frac{\Delta D}{D}\right)\right)
    \\
    \label{eq:t2s_size3_4}
    &\ll_{r,\sigma}
    \frac{D}{|s_*|\Delta D}
    .
  \end{align}
  We substitute \eqref{eq:t2s_size3_3} and \eqref{eq:t2s_size3_4} into the error term \eqref{eq:t2s_size3_2}, which we can now bound:
  \begin{align*}
    \eqref{eq:t2s_size3_2}
    &\ll
    \sum_{n=1}^\infty \frac{a_{\bar\rep}(n^2)}{n^{2-2s}}
    \prod_{p \mid \frac{n}{(n,2q)}} \frac{p}{p+1} \VmDsq 
    \frac{D^{\frac{1}{2}}}{\#\cF} \qet (\qrep D)^\eps \min\!\left\{1,\, \frac{\Delta D}{|s_*| D}\right\}
    \\
    &\ll
    \sum_{n^2 < \mum^{1+\eps}} \frac{n^{2\theta}}{n^{2-2\sigma}}
    \frac{D^{\frac{1}{2}}}{\#\cF} \qet (\qrep D)^\eps \min\!\left\{1,\, \frac{\Delta D}{|s_*| D}\right\}
    \\
    &\ll
    \frac{D^\half}{\#\cF} \qet \left(1 + \mum^{\sigma - \half + \theta}\right) \min\!\left\{1,\, \frac{\Delta D}{|s_*| D}\right\} (\qrep D)^\eps
    .
  \end{align*}
\end{proof}

\hypertarget{proof:t2s_approx_step4}{}
\begin{proof}[Proof of \cref{lemma:t2s_approx_step4}]
  The Dirichlet series in the left hand side of \cref{lemma:t2s_approx_step4} decomposes as
  \begin{align}
    \sum_{n=1}^\infty \frac{a_{\bar\rep}(n^2)}{n^{2-2s}} \prod_{p \mid \frac{n}{(n,2q)}} \frac{1}{1 + p^{-1}} \VmDsq
    \label{eq:t2s_approx_step4_1}
    &
    = \sum_{n=1}^\infty \frac{a_{\bar\rep}(n^2)}{n^{2-2s}} \prod_{p \mid \frac{n}{(n,2q)}} \frac{1}{1 + p^{-1}}
    \\
    \label{eq:t2s_approx_step4_2}
    &
    + \sum_{n=1}^\infty \frac{a_{\bar\rep}(n^2)}{n^{2-2s}} \prod_{p \mid \frac{n}{(n,2q)}} \frac{1}{1 + p^{-1}} \left(\VmDsq - 1\right)
    ,
  \end{align}
  since the hypothesis of \cref{lemma:t2s_approx_step4} requires that $\sigma < \thalf - \theta$, ensuring that the Dirichlet series \eqref{eq:t2s_approx_step4_1} converges absolutely.

  The term \eqref{eq:t2s_approx_step4_2} will end up being a factor of \cref{lemma:t2s_approx_step4}'s error term. Applying \cref{lemma:V_approx},
  \begin{align}
    \nonumber
    \eqref{eq:t2s_approx_step4_2}
    &\ll \sum_{n^2 > \mum^{1-\eps}} \frac{|a_{\bar\rep}(n^2)|}{n^{2-2\sigma}} \prod_{p \mid \frac{n}{(n,2q)}} \frac{p}{p + 1}
    \\
    \nonumber
    &\ll \sum_{n^2 > \mum^{1-\eps}} \frac{n^{2\theta}}{n^{2-2\sigma}}
    \\
    \label{eq:t2s_approx_step4_3}
    &\ll \mum^{\sigma - \half + \theta} \epsfac
    .
  \end{align}
  
  The other factor in the error term, $\frac{1 - \left(1 - \frac{\Delta D}{D}\right)^{1 - r(s-\half)}}{(1 - r(s-\half))\frac{\Delta D}{D}}$, is bounded as
  \begin{align}
    \label{eq:t2s_approx_step4_4}
    \frac{1 - \left(1 - \frac{\Delta D}{D}\right)^{1 - r(s-\half)}}{(1 - r(s-\half))\frac{\Delta D}{D}}
    \ll
    \min\!\left\{1,\, \frac{\Delta D}{|s_*| D}\right\}
    ;
  \end{align}
  see \eqref{eq:t2s_size3_3} and \eqref{eq:t2s_size3_4} from the \hyperlink{proof:t2s_approx_step3}{proof of \cref*{lemma:t2s_approx_step3}}.

  Combining \eqref{eq:t2s_approx_step4_1}, \eqref{eq:t2s_approx_step4_3}, and \eqref{eq:t2s_approx_step4_4} proves \cref{lemma:t2s_approx_step4}.
\end{proof}

\hypertarget{proof:eliminate_derivative_factor}{}
\begin{proof}[Proof of \cref{lemma:eliminate_derivative_factor}]
  
  Our integral is over imaginary parts $|t| < T$. We introduce a cutoff $T_1$, to be optimized later, and split the line of integration into the regions $|t| < T_1$ and $T_1 < |t| < T$. Ultimately we will take $T_1 \asymp D^{1 - \delta}$.

  Decomposing,
  \begin{align}
    \nonumber
    &\frac{1}{2\pi i}\int_{\sigma - iT}^{\sigma + iT} \frac{\halfGamma{1-s+\bar\kappa}}{\halfGamma{s+\kappa}}
    \frac{L(2-2s, \barrepsq)}{L^{(2\qrep)}(3-2s, \barrepsq)}
    \mspace{1mu}
    \prod_{p \nmid 2\qrep}
    \left(1 - \frac{1}{(p+1)(1 - \bar\alpha_1(p)^{-2}p^{3-2s})}\right)
    \\
    \nonumber
    &\hspace{8cm}\cdot
    \frac{1 - \left(1 - \frac{\Delta D}{D}\right)^{1 - (s-\half)}}{(1 - (s-\half))\frac{\Delta D}{D}} \left(\frac{\pi \x}{\qrep D}\right)^{\!s-\half} \frac{ds}{s}
    \\
    \label{eq:deriv_approx_step1_mt}
    &=
    \frac{1}{2\pi i}\int_{\sigma - iT_1}^{\sigma + iT_1} \frac{\halfGamma{1-s+\bar\kappa}}{\halfGamma{s+\kappa}}
    \frac{L(2-2s, \barrepsq)}{L^{(2\qrep)}(3-2s, \barrepsq)}
    \mspace{1mu}
    \prod_{p \nmid 2\qrep}
    \left(1 - \frac{1}{(p+1)(1 - \bar\alpha_1(p)^{-2}p^{3-2s})}\right)
    \\
    \nonumber
    &\hspace{8cm}\cdot
    \frac{1 - \left(1 - \frac{\Delta D}{D}\right)^{1 - (s-\half)}}{(1 - (s-\half))\frac{\Delta D}{D}} \left(\frac{\pi \x}{\qrep D}\right)^{\!s-\half} \frac{ds}{s}
    \\
    \label{eq:deriv_approx_step1_et}
    &+
    \frac{1}{2\pi i}\left(\int_{\sigma - iT}^{\sigma - iT_1} + \int_{\sigma + iT_1}^{\sigma + iT}\right) \frac{\halfGamma{1-s+\bar\kappa}}{\halfGamma{s+\kappa}}
    \frac{L(2-2s, \barrepsq)}{L^{(2\qrep)}(3-2s, \barrepsq)}
    \mspace{1mu}
    \prod_{p \nmid 2\qrep}
    \left(1 - \frac{1}{(p+1)(1 - \bar\alpha_1(p)^{-2}p^{3-2s})}\right)
    \\
    \nonumber
    &\hspace{8cm}\cdot
    \frac{1 - \left(1 - \frac{\Delta D}{D}\right)^{1 - (s-\half)}}{(1 - (s-\half))\frac{\Delta D}{D}} \left(\frac{\pi \x}{\qrep D}\right)^{\!s-\half} \frac{ds}{s}
    .
  \end{align}
  
  We further decompose the fledgling main term \eqref{eq:deriv_approx_step1_mt} as
  \begin{align}
    \label{eq:deriv_approx_step2_mt}
    \eqref{eq:deriv_approx_step1_mt}
    &=
    \frac{1}{2\pi i}\int_{\sigma - iT_1}^{\sigma + iT_1} \frac{\halfGamma{1-s+\bar\kappa}}{\halfGamma{s+\kappa}}
    \frac{L(2-2s, \barrepsq)}{L^{(2\qrep)}(3-2s, \barrepsq)}
    \mspace{1mu}
    \prod_{p \nmid 2\qrep}
    \left(1 - \frac{1}{(p+1)(1 - \bar\alpha_1(p)^{-2}p^{3-2s})}\right)
    \left(\frac{\pi \x}{\qrep D}\right)^{\!s-\half} \frac{ds}{s}
    \\
    \label{eq:deriv_approx_step2_et}
    &
    \begin{aligned}
    &
    \mspace{2mu}+
    \frac{1}{2\pi i}\int_{\sigma - iT_1}^{\sigma + iT_1} \frac{\halfGamma{1-s+\bar\kappa}}{\halfGamma{s+\kappa}}
    \frac{L(2-2s, \barrepsq)}{L^{(2\qrep)}(3-2s, \barrepsq)}
    \mspace{1mu}
    \prod_{p \nmid 2\qrep}
    \left(1 - \frac{1}{(p+1)(1 - \bar\alpha_1(p)^{-2}p^{3-2s})}\right)
    \\
    &\hspace{7.5cm}\cdot
    \left(\frac{1 - \left(1 - \frac{\Delta D}{D}\right)^{1 - (s-\half)}}{(1 - (s-\half))\frac{\Delta D}{D}} - 1\right) \left(\frac{\pi \x}{\qrep D}\right)^{\!s-\half} \frac{ds}{s}
    .
    \end{aligned}
  \end{align}

  In \eqref{eq:deriv_approx_step2_mt} we extend the line of integration to infinity and record the error term
  \begin{align}
    \label{eq:deriv_approx_step3_mt}
    \eqref{eq:deriv_approx_step2_mt}
    &=
    \frac{1}{2\pi i}\int_{\sigma - i\infty}^{\sigma + i\infty} \frac{\halfGamma{1-s+\bar\kappa}}{\halfGamma{s+\kappa}}
    \frac{L(2-2s, \barrepsq)}{L^{(2\qrep)}(3-2s, \barrepsq)}
    \mspace{1mu}
    \prod_{p \nmid 2\qrep}
    \left(1 - \frac{1}{(p+1)(1 - \bar\alpha_1(p)^{-2}p^{3-2s})}\right)
    \left(\frac{\pi \x}{\qrep D}\right)^{\!s-\half} \frac{ds}{s}
    \\
    \label{eq:deriv_approx_step3_et}
    &
    \begin{aligned}
      &
    +
    \frac{1}{2\pi i}\left(\int_{\sigma - i\infty}^{\sigma - iT_1} + \int_{\sigma + iT_1}^{\sigma + i\infty} \right)\frac{\halfGamma{1-s+\bar\kappa}}{\halfGamma{s+\kappa}}
    \frac{L(2-2s, \barrepsq)}{L^{(2\qrep)}(3-2s, \barrepsq)}
    \\
    &\hspace{6.5cm}
    \cdot
    \prod_{p \nmid 2\qrep}
    \left(1 - \frac{1}{(p+1)(1 - \bar\alpha_1(p)^{-2}p^{3-2s})}\right)
    \left(\frac{\pi \x}{\qrep D}\right)^{\!s-\half} \frac{ds}{s}
    .
    \end{aligned}
  \end{align}

  The main term of \cref{lemma:eliminate_derivative_factor}'s right hand side is \eqref{eq:deriv_approx_step3_mt}. We must now bound the three error terms \eqref{eq:deriv_approx_step1_et}, \eqref{eq:deriv_approx_step2_et}, and \eqref{eq:deriv_approx_step3_et}. Let's collect some auxiliary bounds in preparation, on the individual factors of the integrand.
  
  As mentioned in the \hyperlink{proof:t2s_approx_step3}{proof of \cref*{lemma:t2s_approx_step3}}, 
  \begin{align}
    \label{eq:derivproof_deriv_bound_1}
    \frac{1 - \left(1 - \frac{\Delta D}{D}\right)^{\frac{3}{2} - s}}{(\frac{3}{2} - s)\frac{\Delta D}{D}}
    &=
    1 + O\!\left(\frac{|s|\Delta D}{D}\right)
    \shortintertext{
      when $D \gg |s|\Delta D$, and
    }
    \frac{1 - \left(1 - \frac{\Delta D}{D}\right)^{\frac{3}{2} - s}}{(\frac{3}{2} - s)\frac{\Delta D}{D}}
    &\ll
    \frac{D}{|s|\Delta D}
    \label{eq:derivproof_deriv_bound_2}
  \end{align}
  in the complementary regime $D \gg |s|\Delta D$.

  By Stirling's formula \eqref{eq:stirling},
  \begin{align}
    \label{eq:derivproof_gamma_bound}
    \frac{\halfGamma{1-s+\bar\kappa}}{\halfGamma{s+\kappa}}
    \ll
    (|t + \tau| + 1)^{\half - \sigma}
    .
  \end{align}
  
  Given a ``subconvexity bound'' $L(\thalf + it, \chi^2) \ll |t|^{\subconvexityparam + \eps}$ with $\subconvexityparam < \tfrac{1}{4}$, the Phragm\'en--Lindel\"of principle \cite[Thm.\ 8.2.1]{goldfeld:book} implies, for $\thalf < \sigma < 1$,
  \begin{align}
    \nonumber
    L(\sigma + it, \bar\chi^2)
    &\ll |t|^{2\subconvexityparam(1 - \sigma) + \eps}
    ,
    \shortintertext{and so}
    \label{eq:derivproof_L_bound}
    L(2 - 2s, \barrepsq)
    =
    L(2 - 2s + 2i\tau, \bar\chi^2)
    &
    \ll |t - \tau|^{4\subconvexityparam\sigma - 2\subconvexityparam + \eps}
  \end{align}
  for $\thalf < 2 - 2\sigma < 1 \Longleftrightarrow \thalf < \sigma < \tfrac{3}{4}$.

  Moreover, for $\sigma < 1$,
  \begin{align}
    \label{eq:derivproof_prod_bound}
    \begin{aligned}
      \prod_{p \nmid 2\qrep}
      \left(1 - \frac{1}{(p+1)(1 - \bar\alpha_1(p)^{-2}p^{3-2s})}\right)
      &\ll 1
    \end{aligned}
  \end{align}
  by \cref{lemma:prod_convergence}. 

  Combining the bounds \eqref{eq:derivproof_deriv_bound_1}, \eqref{eq:derivproof_deriv_bound_2}, \eqref{eq:derivproof_gamma_bound}, \eqref{eq:derivproof_L_bound}, and \eqref{eq:derivproof_prod_bound}, for $\thalf < \sigma < \tfrac{3}{4}$, $\x \asymp (qD)^{1 \pm \eps}$, and $|t| \gg \tau^{1 + \eps}$
  \begin{align}
    \label{eq:derivproof_integrand_bound}
    \begin{aligned}
      &\frac{\halfGamma{1-s+\bar\kappa}}{\halfGamma{s+\kappa}}
      \frac{L(2-2s, \barrepsq)}{L^{(2\qrep)}(3-2s, \barrepsq)}
      \mspace{1mu}
      \prod_{p \nmid 2\qrep}
      \left(1 - \frac{1}{(p+1)(1 - \bar\alpha_1(p)^{-2}p^{3-2s})}\right)
      \frac{1 - \left(1 - \frac{\Delta D}{D}\right)^{1 - (s-\half)}}{(1 - (s-\half))\frac{\Delta D}{D}} \left(\frac{\pi \x}{\qrep D}\right)^{\!s-\half}
    \\
    &\hspace{7cm}
    \ll
    |s|^{(4\subconvexityparam-1)(\sigma - \half) + \eps} \min\!\left\{1 + O\!\left(\frac{|s|\Delta D}{D}\right), \frac{D}{|s|\Delta D}\right\}
    .
    \end{aligned}
  \end{align}
  With \eqref{eq:derivproof_integrand_bound} we are ready to bound the contribution of our three error terms \eqref{eq:deriv_approx_step1_et}, \eqref{eq:deriv_approx_step2_et}, and \eqref{eq:deriv_approx_step3_et}.

  Beginning with \eqref{eq:deriv_approx_step1_et},
  if $\tau^{1+\eps} \ll T_1 \ll \frac{D}{\Delta D}$ (and using the assumption that $ \frac{D}{\Delta D} \ll T^{1-\eps}$),
  \begin{align}
    \nonumber
    \eqref{eq:deriv_approx_step1_et}
    &\ll \int_{T_1}^T |t|^{(4\subconvexityparam-1)(\sigma - \half) + \eps} \min\!\left\{1 + O\!\left(\frac{|t|\Delta D}{D}\right), \frac{D}{|t|\Delta D}\right\} \,\frac{dt}{t}
    \\
    \nonumber
    &\ll
    \int_{T_1}^{\frac{D}{\Delta D}} |t|^{(4\subconvexityparam-1)(\sigma - \half) + \eps} \left(1 + O\!\left(\frac{|t|\Delta D}{D}\right)\right) \,\frac{dt}{t}
    \,+\,
    \int_{\frac{D}{\Delta D}}^T |t|^{(4\subconvexityparam-1)(\sigma - \half) + \eps} \frac{D}{|t|\Delta D} \,\frac{dt}{t}
    \\
    &
    \label{eq:deriv_approx_step1_eval}
    \begin{aligned}
    \ll& \frac{\Delta D}{D} \left(\left(\frac{D}{\Delta D}\right)^{(4\subconvexityparam-1)(\sigma - \half) + 1 + \eps} - T_1^{(4\subconvexityparam-1)(\sigma - \half) + 1 + \eps} \right)
    \\
    &+ \left(\frac{D}{\Delta D}\right)^{(4\subconvexityparam-1)(\sigma - \half) + \eps} - T_1^{(4\subconvexityparam-1)(\sigma - \half) + \eps}
    \\
    &+ \frac{D}{\Delta D} \left(T^{(4\subconvexityparam-1)(\sigma - \half) - 1 + \eps} - \left(\frac{D}{\Delta D}\right)^{(4\subconvexityparam-1)(\sigma - \half) - 1 + \eps} \right)
    .
    \end{aligned}
  \end{align}
  The term on the last line of \eqref{eq:deriv_approx_step1_eval} is dwarfed by the term on the middle line, and so may be omitted.

  Onto \eqref{eq:deriv_approx_step2_et}.
  If $\tau^{1+\eps} \ll T_1 \ll \frac{D}{\Delta D}$,
  \begin{align}
    \nonumber
    \eqref{eq:deriv_approx_step2_et}
    &= \frac{1}{2\pi i}\int_{\sigma - iT_1}^{\sigma + iT_1} \frac{\halfGamma{1-s+\bar\kappa}}{\halfGamma{s+\kappa}}
    \frac{L(2-2s, \barrepsq)}{L^{(2\qrep)}(3-2s, \barrepsq)}
    \mspace{1mu}
    \prod_{p \nmid 2\qrep}
    \left(1 - \frac{1}{(p+1)(1 - \bar\alpha_1(p)^{-2}p^{3-2s})}\right)
    \\
    \nonumber
    &\hspace{7.5cm}\cdot
    \left(\frac{1 - \left(1 - \frac{\Delta D}{D}\right)^{1 - (s-\half)}}{(1 - (s-\half))\frac{\Delta D}{D}} - 1\right) \left(\frac{\pi \x}{\qrep D}\right)^{\!s-\half} \frac{ds}{s}
    \\
    \nonumber
    &\ll
    \int_{\sigma - iT_1}^{\sigma + iT_1} |s|^{(4\subconvexityparam-1)(\sigma - \half) + \eps} \frac{|s|\Delta D}{D} \,\frac{ds}{s}
    \\
    &\ll
    \label{eq:deriv_approx_step2_eval}
    \frac{\Delta D}{D} + \frac{\Delta D}{D} T_1^{(4\subconvexityparam-1)(\sigma - \half) + 1 + \eps}
    .
  \end{align}

  Now \eqref{eq:deriv_approx_step3_et}:
  \begin{align}
    \nonumber
    \eqref{eq:deriv_approx_step3_et}
    &=
    \frac{1}{2\pi i}\left(\int_{\sigma - i\infty}^{\sigma - iT_1} + \int_{\sigma + iT_1}^{\sigma + i\infty} \right)\frac{\halfGamma{1-s+\bar\kappa}}{\halfGamma{s+\kappa}}
    \frac{L(2-2s, \barrepsq)}{L^{(2\qrep)}(3-2s, \barrepsq)}
    \\
    \nonumber
    &\hspace{6.5cm}
    \cdot
    \prod_{p \nmid 2\qrep}
    \left(1 - \frac{1}{(p+1)(1 - \bar\alpha_1(p)^{-2}p^{3-2s})}\right)
    \left(\frac{\pi \x}{\qrep D}\right)^{\!s-\half} \frac{ds}{s}
    \\
    \nonumber
    &\ll
    \int_{\sigma - i\infty}^{\sigma - iT_1} |s|^{(4\subconvexityparam-1)(\sigma - \half) + \eps} \,\frac{ds}{s}
    \\
    &\ll
    \label{eq:deriv_approx_step3_eval}
    T_1^{(4\subconvexityparam-1)(\sigma - \half) + \eps}
    .
  \end{align}

  Write $T_1 \asymp D^\beta$ and $\gamma \coloneqq \subconvexityparam - \tfrac{1}{4}$. Combining with $\Delta D \asymp D^{\delta}$ and taking $\sigma = \tfrac{3}{4}$, the bounds on the three error terms become
  \begin{align*}
    D^{-\eps}\eqref{eq:deriv_approx_step1_eval}
    &\ll
    D^{\gamma(1-\delta)}
    + D^{-(1-\delta) + (\gamma + 1)\beta}
    + D^{\gamma\beta}
    \\
    D^{-\eps}\eqref{eq:deriv_approx_step2_eval}
    &\ll D^{-(1-\delta)} + D^{-(1-\delta) + (\gamma + 1)\beta}
    \\
    D^{-\eps}\eqref{eq:deriv_approx_step3_eval}
    &\ll D^{\gamma\beta}
    .
  \end{align*}
  Take $\beta = 1-\delta$. The error term is of size $D^{\gamma(1-\delta)}$. One can verify that accommodating the possibility of $T_1 \gg D^{1 - \delta + \eps}$ gives no benefit.
\end{proof}


~

\section{Glossary}
\label{sec:glossary}
\vspace{-2\baselineskip}
{\allowdisplaybreaks
\begin{align*}
  \\&A && \text{\eqref{eq:psidef} Parameter in $\xi$}
  \\&a_\rep && \text{\eqref{eq:rep_Lfunc_series} Dirichlet coefficient with analytic normalization}
  \\&B && \text{\eqref{eq:psidef} Parameter in $\xi$}
  \\&c && \text{\hyperlink{hyper:cdef}{Abscissa for integrals along vertical lines or line segments}}
  \\&c_g && \text{\hyperlink{hyper:cgdef}{Abscissa of integration in \cref*{sec:outline_construction}}}
  \\&c_V && \text{\eqref{eq:Vdef} Abscissa of integration in $V$}
  \\&D && \text{\eqref{eq:Fdef} Upper bound of $d \in \cF$}
  \\&D_0 && \text{\eqref{eq:Fdef} Lower bound of $d \in \cF$}
  \\&\Delta D && \text{\hyperlink{hyper:DeltaD}{$D - D_0$}}
  \\&d && \text{Positive fundamental discriminant}
  \\&e && 2.718281828459045...
  \\&\cF && \text{\eqref{eq:Fdef} Family}
  \\&\mf && \text{\Cref{thm:gamma_murmurations} Meromorphic function}
  \\&\ff && \text{\Cref{lemma:sqfree_character_estimate} Modulus of Dirichlet character}
  \\&\Gp && \text{\eqref{eq:Gdef} Gamma factors of $L(s,\rep)$}
  \\&\Gm && \text{\eqref{eq:Gdef} Gamma factors of $L(s,\bar\rep)$}  
  \\&g && \text{\hyperlink{hyper:gdef}{Approximation to average Dirichlet series in \cref*{sec:outline_construction}}}
  \\&i && \text{$\sqrt{-1}$}
  \\&K_1 && \text{\hyperlink{hyper:K1K2def}{Approximation to average of first term of AFE in \cref*{sec:outline_construction}}}
  \\&K_2 && \text{\hyperlink{hyper:K1K2def}{Approximation to residue of $g$ in \cref*{sec:outline_construction}}}
  \\&L(s,\rep) && \text{\eqref{eq:rep_Lfunc_series} $L$-function}
  \\&L^{(m)} && \text{$L$-function with Euler factors at $p \mid m$ removed}
  \\&L(s,\rep^{(2)}) && \text{\eqref{eq:repL2def} $\textstyle{\sum_n a_\rep(n^2)n^{-s}}$}
  \\&\ell && \text{\eqref{eq:Fdef} $d = \ell \mod \qrep$}
  \\&M && \text{\eqref{eq:murmurationsdef} Scale-invariant murmuration asymptotic}
  \\&\cM^{-1} && \text{Inverse Mellin transform}
  \\&m && \text{\Cref{lemma:sqfree_character_estimate} sum over $n$ coprime to $m$}
  \\&N_E && \text{\hyperlink{hyper:NEdef}{Conductor of elliptic curve}}
  \\&n && \text{\eqref{eq:rep_Lfunc_series} Index of summation or of Dirichlet coefficient}
  \\&\nmur && \text{\hyperlink{hyper:murdefs}{Controls horizontal aspect in \cref*{sec:outline_construction}}}
  \\&O && \text{Big $O$ notation}
  \\&p && \text{Prime number ($\in \Z_{>0}$)}
  \\&Q && \text{\eqref{eq:Qdef} Reduced analytic conductor}
  \\&\qrep && \text{\eqref{eq:rep_Lfunc_series} Conductor of $L(s,\rep)$}
  \\&\qmur && \text{\hyperlink{hyper:murdefs}{Controls vertical aspect in \cref*{sec:outline_construction}}}
  \\&\repdim && \text{\eqref{eq:rep_Lfunc_series} Degree of $L(s,\rep)$}
  \\&\cS && \text{\hyperlink{hyper:cSdef}{Set of $n$'s summed over in \cref*{sec:outline_construction}}}
  \\&s && \text{Complex number}
  \\&s_0 && \text{\eqref{eq:xidef} Shift for $s$ in $\xi$ to remain holomorphic}
  \\&|s_*| && \textstyle{\max_j\!\left\{ |s + \kappa_j|, |1 - s + \bar\kappa_j|, 1 \right\}}
  \\&T && \text{\hyperlink{hyper:cdef}{Bound on $t$ for integrals along vertical line segments.}}
  \\&t && \Im(s)
  \\&t_w && \text{\eqref{eq:sigmawdef} $\Im(w)$}
  \\&t_j && \text{\eqref{eq:sigmawdef} $\Im(\kappa_j)$}
  \\&V_s && \text{\eqref{eq:Vdef} Weight in approximate functional equation's first term}
  \\&V_s^* && \text{\eqref{eq:Vdef} Weight in approximate functional equation's second term}
  \\&V_\pm && \text{\eqref{eq:wtVdef} Weights for more general form of approximate functional equation}
  \\&\x && \text{\Cref{thm:quadratic_murmurations} Controls horizontal aspect}
  \\&\alpha && \text{\text{\eqref{eq:xidef}} Exponent of $D^2$ in $\xi$}
  \\&\hat\alpha && \text{\eqref{eq:alphahatdef} $0$}
  \\&\alpha_j && \text{\eqref{eq:rep_Lfunc_series} Euler product coefficient}
  \\&\alpha_1 && \text{\eqref{eq:rep_Lfunc_series} $\alpha_j$ when $\repdim = 1$}
  \\&\beta && \text{\eqref{eq:xidef} Exponent of $(s - s_0)^2$ in $\xi$}
  \\&\hat\beta && \text{\eqref{eq:betahatdef} Value of $\beta$ which purportedly optimizes error term}
  \\&\Gamma && \text{$\Gamma$ function}
  \\&\gamma && \text{\eqref{eq:gammadeltadef} $T = D^\gamma$}
  \\&\hat\gamma && \text{\eqref{eq:gammahatdef} Optimal value of $\gamma$}
  \\&\delta && \text{\eqref{eq:gammadeltadef} $\#\cF = D^\delta$}
  \\&\eps && \text{\hyperlink{hyper:epsdef}{Sufficiently small positive real which may change from line to line}}
  \\&\zeta && \text{Riemann zeta function}
  \\&\eta_\ell && \text{\eqref{eq:etadef} Coefficient reflecting different densities of fundamental discriminants}
  \\&\Eta_{\cF} && \text{\eqref{eq:Etadef} $\eta_{\qrep,\qrep,\ell}$}
  \\&\theta && \text{\hyperlink{hyper:thetadef}{Progress towards the Ramanujan--Petersson conjecture}}
  \\&\kappa_j && \text{\eqref{eq:rep_Lfunc_series} Local parameters at infinity}
  \\&\kappa && \text{\hyperlink{hyper:kappadef}{Local parameter at infinity when $\repdim = 1$}}
  \\&\Lambda(s,\rep) && \text{\eqref{eq:Lambdadef} Completed $L$-function}
  \\&\subconvexityparam && \text{\Cref{lemma:eliminate_derivative_factor} Progress towards the Lindel\"of hypothesis}
  \\&\mu && \text{M\"{o}bius function} 
  \\&\centralchar && \text{\hyperlink{hyper:centralchardef}{Central character of $\rep$}}
  \\&\xi && \text{\eqref{eq:xidef} Special function imbued in $V_s$}
  \\&\pi && 3.14159265358979323...
  \\&\rho && \text{\eqref{eq:after_contrib_lemmas_et} Exponent of error term in $D$ aspect}
  \\&\hat\rho && \text{\eqref{eq:rhohatdef} Optimal value of $\rho$}
  \\&\hat\rho_f && \text{\eqref{eq:rhofhatdef} Optimal value of $\rho$ when the weight function $f$ is present}
  \\&\sigma && \text{$\Re(s)$}
  \\&\divcount && \text{Divisor counting function: $\divcount(n) \coloneqq \#\{d\in \Z_{>0} \,:\, d\mid n\}$}
  \\&\sigma_w && \text{\eqref{eq:sigmawdef} $\Re(w)$}
  \\&\sigma_j && \text{\eqref{eq:sigmawdef} $\Re(\kappa_j)$}
  \\&\tau && \text{\hyperlink{hyper:gl1repdef}{Exponent of unramified part of $\mathrm{GL}_1$ representation at $\infty$}}
  \\&\Phi &&\text{\hyperlink{hyper:Phidef}{Reasonable function}}
  \\&\phi && \text{Euler totient function}
  \\&\rep && \text{\eqref{eq:rep_Lfunc_series} Irreducible unitary cuspidal automorphic representation}
  \\&\chi && \text{\hyperlink{hyper:gl1repdef}{Dirichlet character corresponding to $\mathrm{GL}_1$ representation}}
  \\&\chi_d && \text{Kronecker symbol $\big(\tfrac{d}{\cdot}\big)$}
  \\&\chi_k^0 && \text{Trivial character mod $k$}
  \\&\psi && \text{\Cref{cond:afe_general} Special function in more general form of approximate functional equation}
  \\&\ET && \text{\eqref{eq:ETdef} Error term of \cref*{thm:r1size,thm:r2size}}
  \\&\ET_1,\dots,\ET_5 && \text{\eqref{eq:ETdef} Individual error terms in precursor to \cref*{thm:r1size,thm:r2size}}
  \\&\rootnum_{\cF} && \text{\hyperlink{hyper:rootnumFdef}{$\rootnum_{\rep\otimes\chi_d}$}}
  \\&\rootnum_{\chi_d} && \text{\hyperlink{hyper:rootnumchid}{Root number of $L(s,\chi_d)$ for $d > 0$.}}
  \\&\starsum && \text{Sum over fundamental discriminants}  
  \\&\int_{(c)} && \text{\eqref{eq:vertical_line_integral} Integral from $c - i\infty$ to $c + i\infty$}
  \\&\otimes && \text{\eqref{eq:rankin_selberg_gl1} Rankin--Selberg convolution}
  \\&\onesymb && \text{$1$ if statement is true, $0$ otherwise}
  \\&n\neq\square && \text{There does not exist $m \in \Z$ such that $n = m^2$.}
  \\&(\,\cdot\,,\cdot) && \text{Owl}
  \\&f \ll g && f = O(g)
  \\&f \asymp g && \text{$f \ll g \ll f$, with the implied constants absolute}
  \\&f \asymp g^{b\pm\eps} && \text{\hyperlink{hyper:asymppmdef}{$g^{b-\eps} \ll f \ll g^{b+\eps}$}}
  \\&\# && \text{Cardinality of a set}
\end{align*}
} 

\section*{Acknowledgements}
We thank Mike Rubinstein and Jerry Wang for their help.

\renewcommand{\bibliofont}{\normalfont\small} 
\bibliographystyle{amsalpha}
\bibliography{murmurationsbib}{}

\end{document}